\documentclass[reqno,10pt]{amsart}
\usepackage{amsmath,mathrsfs}
\usepackage{amssymb}
\usepackage{graphicx}
\usepackage{dutchcal}
 \usepackage{pdfsync}

\usepackage{color}
\usepackage{colonequals}
\usepackage[colorlinks,
            linkcolor=red,
            anchorcolor=blue,
            citecolor=green
            ]{hyperref}
 
\usepackage[utf8]{inputenc}

\usepackage{booktabs}
\usepackage{geometry}

\usepackage[backend=biber, sorting=nty, doi=false, url=false, isbn=false, maxbibnames=5]{biblatex}

\allowdisplaybreaks[1]

\def\inte#1{
\displaystyle\mathop{#1\kern0pt}^\circ }

\let\pa=\partial
\let\al=\alpha

\let\Ga=\Gamma

\let\lam=\lambda

\let\f=\frac

\let\om=\omega

\let\D=\Delta
\let\Lam=\Lambda
\let\Om=\Omega

\let\ka=\kappa
\def\mP{\mathbf{P}}

\def\cD{{\mathcal D}}
\def\cE{{\mathcal E}}

\def\cG{{\mathcal G}}
\def\cH{{\mathcal H}}

\def\sfT{\mathsf{T}}

\def\cL{{\mathcal L}}

\def\cP{{\mathcal P}}

\def\cT{{\mathcal T}}

\def\pa{\partial}

\def\virgp{\raise 2pt\hbox{,}}
\def\cdotpv{\raise 2pt\hbox{;}}

\def\C{\mathop{\mathbb C\kern 0pt}\nolimits}
\def\DD{\mathop{\mathbb D\kern 0pt}\nolimits}
\def\EE{\mathop{{\mathbb E \kern 0pt}}\nolimits}
\def\K{\mathop{\mathbb K\kern 0pt}\nolimits}
\def\N{\mathop{\mathbb N\kern 0pt}\nolimits}
\def\Q{\mathop{\mathbb Q\kern 0pt}\nolimits}
\def\R{\mathop{\mathbb R\kern 0pt}\nolimits}
\def\SS{\mathop{\mathbb S\kern 0pt}\nolimits}

\def\<{\langle}
\def\>{\rangle}
\def\S{\mathbb{S}}
\def\R{\mathbb{R}}
\def\T{\mathbb{T}}
\def\Z{\mathbb{Z}}
\def\N{\mathbb{N}}

\def\th{\theta}

\def\al{\alpha}

\def\de{\delta}
\def\vphi{\varphi}

\def\gs{\gtrsim}
\def\ls{\lesssim}
\def\pa{\partial}
\def\vep{\varepsilon}

\def\ZZ{\mathop{\mathbb Z\kern 0pt}\nolimits}
\def\TT{\mathop{\mathbb T\kern 0pt}\nolimits}
\def\P{\mathop{\mathbb P\kern 0pt}\nolimits}

\def\dv{\mbox{div}}

\newcommand{\lr}[1]{\langle #1 \rangle}

\def\na{\nabla}

\def\th{\theta}
\def\mc{\mathbf{c}}
\def\mC{\mathbf{C}}

\newcommand{\erf}{\operatorname{erf}}

\newcommand{\beq}{\begin{equation}}
\newcommand{\eeq}{\end{equation}}
\newcommand{\ben}{\begin{eqnarray}}
\newcommand{\een}{\end{eqnarray}}
\newcommand{\beno}{\begin{eqnarray*}}
\newcommand{\eeno}{\end{eqnarray*}}

\newtheorem{defi}{Definition}[section]
\newtheorem{thm}{Theorem}[section]
\newtheorem{lem}{Lemma}[section]
\newtheorem{rmk}{Remark}[section]
\newtheorem{col}{Corollary}[section]
\newtheorem{prop}{Proposition}[section]

\definecolor{darkblue}{rgb}{0,0,0.7} 
\definecolor{darkred}{rgb}{0.9,0.1,0.1}
\definecolor{darkgreen}{rgb}{0,0.5,0}

\bibliography{Inhom_Runaway}

\begin{document}
\title[Runaway Avalanches]{Runaway avalanches in plasmas with external electric fields: spatially inhomogeneous case in a perturbation framework
}

\author[L.-B. He, R. M. H\"ofer, J. Ji and R. Winter]{Ling-Bing He, Richard M. H\"ofer, Jie Ji and  Raphael Winter}

\address[L.-B. He]{Department of Mathematical Sciences, Tsinghua University\\
	Beijing 100084,  P. R.  China.} \email{hlb@tsinghua.edu.cn}

\address[R. M. H\"ofer]{Faculty of Mathematics, University of Regensburg, Germany} \email{richard.hoefer@ur.de}

\address[J. Ji]{School of Mathematics, Nanjing University  of Aeronautics and Astronautics\\
	Nanjing 211106,  P. R.  China.} \email{jij\_24@nuaa.edu.cn}
	
\address[Raphael Winter]{Cardiff School of Mathematics, Cardiff University, UK} \email{WinterR6@cardiff.ac.uk}

\begin{abstract}
We consider the Landau-Coulomb equation for a (hydrogen) plasma heated by an external electric field. In this setting, theoretical and experimental results in plasma physics show the emergence of so-called \emph{runaway electrons} which are linearly accelerating but only lead to a minimal increase of the plasma temperature. Runaway electrons are a major obstacle in nuclear fusion since they can overcome the confinement and  damage the structure of the reactor. 

We rigorously prove the well-posedness of the underlying nonlinear \emph{open} Landau-Coulomb system in a perturbative setting and  the  conjectured growth bounds for the mean velocity and plasma temperature. We show that the mean velocity satisfies $V(t)=V_0+Et+O(1)$ as $t\to\infty$, and capture the sharp logarithmic growth of the temperature. Furthermore, we prove that the electron distribution can be asymptotically described by a scattering-type Maxwellian. 
Due to the different nature of the electron-electron and electron-ion interactions, we recast the equation as a novel coupled system that allows us to isolate the dissipation structures of the two operators. For the coupled system, we perform a micro-macro decomposition to show convergence to the scattering-type Maxwellian.

\end{abstract}

\maketitle

\setcounter{tocdepth}{1}
\tableofcontents

\setlength{\abovedisplayskip}{3.5pt plus 2pt minus 1pt}
\setlength{\belowdisplayskip}{3.5pt plus 2pt minus 1pt}
\setlength{\abovedisplayshortskip}{0pt plus 1pt}
\setlength{\belowdisplayshortskip}{2pt plus 1pt minus .5pt}
\setlength{\jot}{2pt}



\section{Introduction}

A plasma is a collection of fast-moving, charged particles whose long-range Coulomb interactions dominate over infrequent collisions. Nuclear fusion requires plasma of very high temperature, so the confined plasma in fusion reactors is heated by an external source.  
More precisely, the temperature of the electron distribution increases as a result of acceleration through the electric field in combination with the friction induced by the interaction with the much slower ions. This process can be interrupted by a \emph{runaway avalanche}, where the electron-ion friction is insufficient to stop an uninhibited acceleration of electrons. The emergence of runaway avalanches is one of the main obstacles in nuclear fusion, and a major concern for the operation of large Tokamak reactors such as ITER~\cite{Boozer15,reux2015runaway}. For a comprehensive review of the physics of runaway electrons see~\cite{Breizman19,Martin17}.


Remarkably, the emergence of \emph{runaway electrons} in plasma physics has been explained by Dreicer with a classical model in kinetic theory~\cite{Dre}. To this end, consider a hydrogen plasma driven by an electric field \(E\), where the coupled system is described by two Landau equations that determine the distribution functions \(F_{-}(t,x,v)\) for electrons and \(F_{+}(t,x,v)\) for ions (protons):
\[
\begin{aligned}
& \pa_t F_-+v\cdot \na_xF_--\f{\mathrm{e}}{m_-}E\cdot \na_v F_-=Q_{--}(F_-,F_-)+Q_{+-}(F_+,F_-), \\
& \pa_t F_++v\cdot \na_xF_++\f{\mathrm{e}}{m_+}E\cdot \na_v F_+=Q_{++}(F_+,F_+)+Q_{-+}(F_-,F_+).
\end{aligned}
\]
In the above, $\mathrm{e}$ is the magnitude of the electron charge, $m_-$ and $m_+$ are the mass of the electron and ion respectively. Moreover, $Q_{--},Q_{-+},Q_{+-}$ and $Q_{++}$ are the Landau-Coulomb collision operator, derived by Landau in 1936 \cite{Landau1936},
\[
Q_{ij}(F_i,F_j)(v) =\f1{m_j}\dv_v\Big(\int_{\R^3}c_{ij}a(v-v_*)\Big(\f1{m_j}F_i(v_*)\na_vF_j(v)-\f1{m_i}F_j(v)\na_{v_*}F_i(v_*)\Big)dv_*\Big)
\]
with 
\begin{equation}\label{az}
a(z)=|z|^{-1} \Pi(z):=|z|^{-1}\left(\mathbf{I} -\frac{z\otimes z}{|z|^2} \right),\quad c_{ij}=\f{|\ln \Lam_c|n_in_j}{8\pi \vep_0^2},\quad i,j=-,+,
\end{equation}
where $|\ln \Lam_c|$ is the Coulomb logarithm, and $\vep_0$ is the vacuum permittivity. The constants $n_-$ and $n_+$ are the number densities of the electrons and ions so that we can assume $F_\pm$ to be probability densities. 

In all physical situations, the electron mass $m_-$ is much smaller than the ion mass $m_+$. For instance, $m_-/m_+\approx 5.4\times 10^{-4}$ for a hydrogen plasma. In the limit $m_-/m_+\rightarrow 0$, the ions can be considered static relative to the electrons and the electron-ion collision operator becomes a  spherical diffusion in velocity space. This approximation is commonly used in plasma physics to obtain a closed equation for the electron distribution, see for example  \cite[Section 6.3.3]{Pieter}. From now on, we absorb the electron charge-to-mass factor and its sign into $E$, so that $E$ denotes the prescribed acceleration. The kinetic equation for the electron distribution $F$ reads 
\begin{equation}\label{main}
\pa_t F+v\cdot\na_xF+E\cdot\na_v F=Q(F,F)+ \dv_v \bigg(\frac{\Pi(v)}{|v|} \nabla_v F\bigg).
\end{equation}
We work on $\mathbb T^3=\mathbb R^3/\mathbb Z^3$ with normalized Lebesgue measure.

 \subsection{Runaway problem}\label{runawayproblem}

 The system~\eqref{main} contains all the ingredients to explain the emergence of runaway electrons. Runaway occurs if the kinetic energy of the plasma grows much more rapidly than its thermal energy. In terms of the spatially averaged kinetic and thermal energies, this means
\begin{equation}\label{VT}
\begin{aligned}
&\frac{|V(t)|^2 }{T(t)} \to \infty,\quad\mbox{where} \\
&V(t):=\int_{\T^3\times\R^3} vF\,dv\,dx, \\
&T(t):=\frac{1}{3}\int_{\T^3\times\R^3} |v-V(t)|^2 F\,dv\,dx.
\end{aligned}
\end{equation}

 We first give an intuitive explanation of the various terms in equation~\eqref{main} and their significance for the runaway phenomenon:
 \begin{enumerate}
     \item  The term $v\cdot\na_x $ represents the free transport. The transport on the torus $\mathbb{T}^3$ together with collisions will drive the system to a spatially homogeneous state.
     \item The term $E\cdot\na_v$ accounts for linear acceleration through the electric field $E$. This  acts as the source of energy for the system. This term only changes the \emph{kinetic} energy, though.
     \item The standard Landau collision operator $Q$ accounts for the grazing electron-electron  collisions. It drives the velocity distribution towards the (local) Maxwellian. It affects neither the  linear momentum nor the temperature, hence leaving both kinetic and thermal energy invariant.  
     \item The spherical term specifically represents electrons scattering off the much heavier ions. This induces a friction that \emph{transfers} kinetic energy into thermal energy without changing the total energy. Hence, through the combination of the acceleration and spherical diffusion, the plasma can be heated which is the goal in many applications.  However, the friction force experienced by a particle is quadratically decaying in its velocity. This makes it possible for particles to accelerate (almost) uninhibited once they have reached a critical velocity depending on the magnitude $|E|$ of the electric field.
 \end{enumerate}
In this paper, we tackle the challenge of giving a mathematically rigorous description of the runaway phenomenon. We actually prove a very  precise description of the phenomenon in a perturbative regime as follows:

{\bf I. Runaway does occur.} This can be equivalently rephrased in terms of the average velocity $|V(t)|$ and the average thermal velocity $\sqrt{T(t)}$ as
\begin{equation}\label{eq:runawayrepeat}
\lim_{t \rightarrow \infty } \frac{ |V(t)|}{\sqrt{T(t)}} = \infty.
\end{equation}
{\bf II. Sharp bounds for the average and thermal velocity.} We capture the precise asymptotic growth rate of both $|V(t)|$ and $\sqrt{T(t)}$ as
\begin{equation}
|V(t)|\sim |E|t, \quad \sqrt{T(t)} \sim \log^\frac12(t).
\end{equation}
{\bf III. Asymptotic profile.}
Finally, we give a precise description of the asymptotics of the velocity distribution $F(t)$ as $t\rightarrow \infty$. To this end, we first recall that we expect the system to become spatially homogeneous as a result of transport and collisions. Therefore, the thermal velocity $\sqrt{T(t)}$ represents the typical width of the velocity distribution as $t\rightarrow \infty$ throughout the spatial domain. 

We can therefore hope that $F$ converges to an asymptotic runaway profile $\Phi(v)$ in the sense that 
\begin{equation}\label{scattering:maxwellian}
\Bigl\| T(t)^{3/2}F\big(t,x,V(t)+\sqrt{T(t)}v\big)- \Phi(v)\Bigr\| \longrightarrow 0,\quad\text{as}\quad t\rightarrow \infty,
\end{equation}
where $\Phi$ has the same mass as $F(t)$; our result gives $\Phi=\mu$.

For a simplified, linear version of~\eqref{main} some of these questions have been considered in~\cite{Pia1,Pia2}. The authors obtain quantitative asymptotic bounds for a Lorentz gas under sufficiently strong uniform electric field. 
Our goal is to provide a precise mathematical description of the runaway dynamics for the nonlinear Landau-Coulomb model. We will consider solutions to the equation \eqref{main}, 
where $F=F(t,x,v)$ is the density distribution function of particles with position $x\in\T^3$ and velocity $v\in\R^3$ at time $t$ and $E$ is the prescribed constant acceleration with the sign convention used in \eqref{main}. The first term on the right-hand side of \eqref{main} is given by the Landau-Coulomb collision operator
\begin{equation}\label{12d}
Q(F,G)(v) =\dv_v \Big (\int_{\R^3}a(v-v_*)\big(F(v_*)\na_v G(v)-G(v)\na_{v_*}F(v_*)\big)dv_*\Big ),
\end{equation}
with $a(z)$ defined in \eqref{az}.

\subsection{Brief review of previous results}
Extensive literature exists on the Landau equation, addressing both well-posedness and regularity. Here we survey only the well-posedness results. 

\noindent$\bullet$ For the spatially homogeneous case, global existence for hard potentials and Maxwell molecules was settled in \cite{DV1,DV2,CV}, while the regularity estimates of \cite{GG,Wu,LS} yield smooth solutions for the moderately soft potential cases. The very-soft-potential case—including the physically central Coulomb interaction—remained open until the recent breakthrough of Guillen and Silvestre \cite{GS25}, who proved that the Fisher information is monotone decreasing and thereby constructed global smooth solutions without any smallness assumption. Building on this insight, the follow-up papers \cite{GL,GGL2,GSun,HJL} have extended well-posedness to critical Lebesgue and Sobolev spaces, respectively. For the multi-species Landau--Coulomb equation, the Fisher information need not be monotone. In the spatially homogeneous two-species limit with one species of infinite mass, \cite{JWY} constructs a new Lyapunov functional and proves global existence under the hypotheses of its Theorem 1.3, including the additional condition in the Coulomb case.

\noindent$\bullet$ The spatially inhomogeneous Landau equation remains open for general initial data, yet its well-posedness and convergence near equilibrium have attracted intensive recent study; see, e.g., \cite{Guo1,SG1,SG2,DSSS} working in torus, and \cite{CM1,CTW} where the authors employ semi-group techniques to upgrade the former exponential weights to polynomial ones. 
Apart from the semi-group approach, an alternative coupled-system decomposition of the solution enables us to work with polynomial weights and will be employed throughout this paper. This line of research goes back to Caflisch \cite{Caf}, who first split the solution as $g=\sqrt{M_1}g_1+\sqrt{M_2}g_2$ with two distinct Maxwellians. As shown in \cite{DL} for the Boltzmann equation, one of the two exponential weights can in fact be reduced to polynomial growth.

In this work we establish global well-posedness for the runaway equation \eqref{main} in a perturbation framework and, moreover, characterize its long-time behavior, pinpointing the sharp growth rates of both momentum and temperature.

\subsection{Main results} Now we present the main results, which answer the three questions raised above.
{\begin{thm}\label{longtimebehavior1}
Consider the runaway problem \eqref{main} with initial data
$F_0\geq0$ and $\|F_0\|_{L^1_{x,v}}=1$. Let $V_0$ and $T_0>0$
be its mean velocity and temperature, defined as in \eqref{VT}.
There exists an integer $k_0\geq17$ such that, for any $k_1>k_0+3$,
there are positive constants $\vep_1,\vep_2$, depending only on
$k_0,k_1,V_0,T_0$, with the following property. If
\begin{equation}\label{small}
\begin{aligned}
&|E|^{-1}<\vep_1, \\
&T_0^{\f32}\sum_{|\al|=0,2}\int_{\T^3\times\R^3}\big(\pa^\al_x(F_0-M_{V_0,T_0})\big)^2\big\<T_0^{-\f12}(v-V_0)\big\>^{2(k_1-4|\al|)}dvdx<\vep_2,
\end{aligned}
\end{equation}
then \eqref{main} admits a unique global solution $F(t)\geq0$. Uniqueness is understood in the class specified by Theorem~\ref{longtimebehavior} under the change of variables \eqref{changeofvariable}.
Let $V(t),T(t)$ be its moments defined in \eqref{VT}, and write
$M_{V,T}:=(2\pi T)^{-3/2}\exp\{-|v-V|^2/(2T)\}$.
For $t\geq1$, the solution satisfies
\[
\begin{aligned}
&T(t)^{\f32}\sum_{|\al|=0,2}\int_{\T^3\times\R^3}\big(\pa^\al_x(F(t)-M_{V(t),T(t)})\big)^2\big\<T(t)^{-\f12}(v-V(t))\big\>^{2(k_0-4|\al|)}dvdx \\
&\quad\ls(1+t)^{-\f{k_1-k_0}{k_1-k_0+3/2}-\left(\f{k_1-k_0}{k_1-k_0+3/2}\right)^2}.
\end{aligned}
\]
Moreover, for $t\geq1$,
\begin{equation}\label{VTlongtime}
\begin{aligned}
\f14 T_0+\f{c_1}{|E|}\ln(3+|E|t)\leq T(t)
&\leq\f32 T_0+\f{c_2}{|E|}\ln(3+|E|t), \\
|V(t)-V_0-Et|&\ls T_0^{-1}+|E|^{-1}.
\end{aligned}
\end{equation}
Here $c_1,c_2>0$ and the implicit constants depend only on
$k_0,k_1,V_0,T_0$, and are independent of $E,t$.
In particular, the decay exponent of the squared norm approaches
$-2$ as $k_1-k_0\to\infty$.
\end{thm}}

\begin{rmk}
\textnormal{
The proof of the theorem requires detailed estimates for the quantities introduced in~\eqref{VT}.
The mass $\rho(t):=\iint F(t,x,v)\,dv\,dx$ is conserved and equals one under the assumptions of Theorem~\ref{longtimebehavior1}. Formally, the moment equations are
\begin{align}
& \f d{dt}V(t)=E-\int_{\T^3\times\R^3}\f{\Pi(v)}{|v|}\na_v Fdvdx=E-2\int_{\T^3\times \R^3}\f v{|v|^3}Fdvdx=:E-2R;\label{dtV} \\
& \f d{dt} T(t)={\frac 2 3} \left( E\cdot V-V\cdot \f d{dt} V \right)=\f43R\cdot V.\label{dtT}
\end{align}
Here we use the identities $\Pi(v)\,v=0$ and $\displaystyle \nabla_{\!v}\!\cdot\!\Bigl(\frac{\Pi(v)}{|v|}\Bigr)=-2\,\frac{v}{|v|^{3}}$. }
\end{rmk}

\subsection{Strategy  of the paper}
Let us provide a brief overview of the key strategies employed in our proof.
\smallskip

\noindent \underline{\it Step 1. Reduction of the problem.}  
We reduce the original Cauchy problem of \eqref{main} to an equivalent coupled system $(G(t,x,v), T(t), V(t))$, where $G$ is the profile of the distribution function. On one hand, this transformation enables us not only to focus on the evolution of bulk velocity and temperature to study runaway acceleration but also to analyze the asymptotic profile to determine the long-time dynamics. On the other hand, the transformation reduces the scattering-type, time-varying Maxwellian to a centered, normalized Maxwellian, thereby preserving the dissipative structure — in particular, the coercivity estimate of the collision operator, which is well known to rely on the mass and energy of the Maxwellian. This plays an essential role in the perturbative framework. We finally emphasize that the transformation relies heavily on the Galilean invariance of the collision operator and we refer readers to Proposition \ref{prop2} for detailed proof.
 
\smallskip

\noindent \underline{\it Step 2.  Local well-posedness.}  
We first cut off the singularity of the spherical diffusion using the function in Definition~\ref{cutofffunction}. For each fixed cutoff, the moment equations are locally Lipschitz, and a difference estimate proves convergence of the successive approximations. Nonnegativity follows by smoothing the linear problem and applying the maximum principle. The uniform energy estimates then extend these solutions to the common interval of Theorem~\ref{thm:nu-local}. We apply the Aubin--Lions lemma to the distributions $F_\nu$ and remove the cutoff, obtaining Theorem~\ref{localwellposedness}. Uniqueness is proved by an energy estimate for $F$.
\smallskip

\noindent \underline{\it Step 3.  Global well-posedness.}  
We reduce the global well-posedness problem to proving three key properties: \(g\) remains small, \(T\) grows at most logarithmically, and \(V\) grows linearly in time. In particular, the growth estimates for \(T\) and \(V\) follow directly from the smallness of \(g\). The propagation of smallness of \(g\), however, is more subtle:

- In the perturbation framework, the standard route to global well-posedness combines symmetric linearization with micro-macro decomposition (see for example \cite{Guo1,DSSS}); the former exploits the weighted coercivity of the linearized collision operator to yield dissipation of the microscopic part in an exponentially weighted space, while the latter uses the transport operator to control the macroscopic part by the microscopic one, see \eqref{cE0cD0} and \eqref{macroscopic}, respectively. For the present model, however, exponential weights are no longer admissible, because the spherical-diffusion operator in \eqref{equg} raises the weight by two orders. Indeed,
  \[
\begin{aligned}
&\Big(\nabla_v \cdot \bigg(\frac{\Pi(V(t)+T(t)^{\frac12}v)}{| V(t)+T(t)^{\frac12}v|} \nabla_v (\mu^{\frac12}g)\bigg), \mu^{-\frac12} g\Big)_{L^2_v} \\
&= -\Big(\frac{\Pi(V(t)+T(t)^{\frac12}v)}{| V(t)+T(t)^{\frac12}v|}\nabla_v g,\nabla_v g\Big)_{L^2_v} + \frac14\Big(\frac{\Pi(V(t)+T(t)^{\frac12}v)}{| V(t)+T(t)^{\frac12}v|} v g, v g\Big)_{L^2_v}.
\end{aligned}
\]
  The second term on the right-hand side does not appear amenable to control. A natural alternative is to work in polynomially weighted spaces using semigroup techniques, but the time-dependent coefficients(see \eqref{eqG}) introduced by the new coordinates complicate the semigroup approach.

- To overcome this difficulty, we use a decomposition inspired by Caflisch \cite{Caf}; see also \cite{GW2025} for a related two-scale construction for the Landau equation.
We write \(g=g_1+\mu^{1/2}g_2\) and choose the coupled system \eqref{coupg1}--\eqref{coupg2} so that the spherical terms appear only in the equation for \(g_1\).
The equation for \(g_2\) then retains the standard linearized collision operator and admits a micro-macro estimate.
Using the cutoffs from Definition~\ref{cutofffunction}, we estimate \(\chi_{4N}g_1\) in the interior and \(\chi_N^c g\) in the exterior.
In the interior, \eqref{VTcondition0} places the singularity of spherical diffusion outside the support of \(\chi_{4N}\), allowing the term involving \(\mu^{1/2}g_2\) to be estimated against the collision dissipation.
In the exterior, we work directly with \(g\) and retain the dissipative spherical term.
The Maxwellian tails yield small remainders for large \(N\), which are absorbed when these estimates are combined with the micro-macro estimate for \(g_2\).
This gives Theorem~\ref{energyestimate}, whose energy controls \(\|g\|_{X_k}^2\) with a constant independent of time.
\smallskip

  \noindent \underline{\it Step 4. Continuation argument.}
We combine the energy estimates at two velocity weights.
The high-order bound retains the square of the initial smallness parameter.
Interpolating the low-order energy between this high-order norm and the
dissipation gives a scalar differential inequality.
A decaying supersolution then yields both low-order smallness and the
time integrability needed in the moment estimates.
We first fix the perturbation size and the corresponding time \(t_0\),
and only then choose a sufficiently strong field. This gives an
initial-data threshold independent of \(E\).
The estimates are closed on finite intervals, and
Remark~\ref{nu:continuation} excludes a finite maximal existence time.

We close this section with one final point. Because the coefficients of the transformed equation \eqref{equg} contain negative powers of \(T(t)\), we need to keep \(T(t)\) bounded away from zero. By the detailed growth estimates for \(T\) in Section \ref{growth}, this requirement is guaranteed by choosing a sufficiently strong electric field (see the proof of Lemma \ref{largetime} and Theorem \ref{energyestimate} for details).

 \subsection{Organization of the paper}

The rest of this paper is organised as follows.  

In Section \ref{reduction}, we introduce suitable coordinates and transformations that convert the scattering-type Maxwellian into the familiar global Maxwellian and rewrite the equation as a perturbation system around this profile.  
Section \ref{notation} introduces the notation used throughout the paper, with a focus on polynomially weighted function spaces that form the analytical framework for the subsequent estimates.  
Section \ref{Landau} provides basic estimates for the Landau collision operator.  
In Section \ref{local}, we establish the local well-posedness and non-negativity of the solution.  
Section \ref{growth} derives growth estimates for the macroscopic quantities.  
Section \ref{coupledsystem} addresses the core challenge of propagating the smallness of \(g\).  
Finally, Section \ref{proof} presents the proof of global well-posedness and decay estimates. Technical tools and lemmas are collected in the appendix.

\section{Reformulation of the equation as a perturbation of a standard Maxwellian}\label{reduction}
We introduce coordinates in which the scattering-type Maxwellian becomes the normalized Maxwellian and derive the transformed equation. The calculation uses the translation and scaling properties of ${Q}$. The following identities follow directly from the definition \eqref{12d}.
\begin{lem}\label{TransQ}
	Let $g = g(v)$ and $h = h(v)$ be smooth functions.
	
	\indent$\bullet$ Let $u \in \mathbb{R}^3$ and $g_1(v) = g(v + u), h_1(v) = h(v + u)$. Then
	\[
{Q}(g, h)(v + u) = {Q}(g_1, h_1)(v).
\]
	
	\indent$\bullet$ Let $r>0$ and $g_2(v) = g(rv), h_2(v) = h(rv)$. Then
	\[
{Q}(g, h)(rv) = {Q}(g_2, h_2)(v).
\]
	
\end{lem}

\medskip

Recall the scattering-type Maxwellian and its normalized counterpart:
\[
M_{V(t),T(t)}{(v)}=\f1{(2\pi T(t))^{\f32}}\exp\Big\{-\f{|v-V(t)|^2}{2T(t)}\Big\},\quad \mu=e^{-|v|^2/2}/(2\pi)^{\f32}.
\]	
We introduce 
\begin{equation}\label{F1}
\begin{aligned}
& F_1(t,x,v)=T(t)^{\f32}F(t,x,V(t)+T(t)^{\f12}v), \\
& M_1(t,x,v)=T(t)^{\f32}M_{V(t),T(t)}(V(t)+T(t)^{\f12}v)=e^{-|v|^2/2}/(2\pi)^{\f32}=\mu.
\end{aligned}
\end{equation}
The chain rule gives 
\[
\begin{aligned}
& \pa_t F_1(t,x,v)=\f32T(t)^{\f12}T'(t)F(t,x,V(t)+T(t)^{\f12}v) \\
&\quad{} +T(t)^{\f32}\pa_tF(t,x,V(t)+T(t)^{\f12}v) \\
&\quad{} + T(t)^{\f32}\big(V'(t)+\f12T(t)^{-\f12}T'(t)v\big)\cdot\na_v F(t,x,V(t)+T(t)^{\f12}v), \\
& \na_vF_1(t,x,v)=T(t)^2\na_v F(t,x,V(t)+T(t)^{\f12}v), \\
&\na_xF_1(t,x,v)=T(t)^{\f32}\na_xF(t,x,V(t)+T(t)^{\f12}v),
\end{aligned}
\]
which implies that
\[
\begin{aligned}
& \na_vF(t,x,V(t)+T(t)^{\f12}v)=T(t)^{-2}\na_v F_1(t,x,v), \\
&\na_xF(t,x,V(t)+T(t)^{\f12}v)=T(t)^{-\f32}\na_xF_1(t,x,v), \\
& \pa_tF(t,x,V(t)+T(t)^{\f12}v)
=T(t)^{-\f32}\pa_tF_1(t,x,v)
-\f32T(t)^{-\f52}T'(t)F_1(t,x,v) \\
&\quad{} - T(t)^{-2}(V'(t)+\f12T(t)^{-\f12}T'(t)v)\cdot\na_vF_1(t,x,v).
\end{aligned}
\]
Substituting these identities into \eqref{main} and applying Lemma \ref{TransQ} together with \eqref{dtV} yields
\begin{equation}\label{equF1}
\begin{aligned}
& \pa_tF_1(t,x,v)+\big(V(t)+T(t)^{\f12}v\big)\cdot\na_xF_1(t,x,v)+\big(-\f12T(t)^{-1}T'(t)v+2T(t)^{-\f12}R(t)\big)\cdot\na_v F_1(t,x,v) \\
&= T(t)^{-\f32}Q(F_1,F_1)(t,x,v)+\f32 T(t)^{-1}T'(t)F_1(t,x,v)+T(t)^{-\f32}\dv_v \Big(\f{\Pi(v+\om(t))}{|v+\om(t)|}\cdot \na_v F_1\Big)(t,x,v),
\end{aligned}
\end{equation}
where $\om(t)=V(t)/\sqrt{T(t)}$ and the orthogonal projection is given by 
$
\Pi(v+\om(t))=\mathbf{I}-\f{(v+\om(t))\otimes (v+\om(t))}{|v+\om(t)|^2}.
$

Next, we set
\begin{equation}\label{F2}
F_2(t,x,v)=F_1(t,x+H(t),v) \quad M_2(t,x,v)=M_1(t,x+H(t),v)=\mu\quad\mbox{with}\quad H(t):=\int_0^t V(\tau)d\tau,
\end{equation}
so that
\[
\begin{aligned}
& \pa_tF_2(t,x,v)=\pa_t F_1(t,x+H(t),v)+V(t)\cdot\na_xF_1(t,x+H(t),v), \\
& \na_xF_2(t,x,v)=\na_x F_1(t,x+H(t),v),\quad \na_vF_2(t,x,v)=\na_v F_1(t,x+H(t),v),
\end{aligned}
\]
which implies 
\[
\begin{aligned}
& \pa_tF_1(t,x+H(t),v)=\pa_t F_2(t,x,v)-V(t)\cdot\na_xF_2(t,x,v), \\
& \na_x F_1(t,x+H(t),v)=\na_xF_2(t,x,v),\quad\na_v F_1(t,x+H(t),v)= \na_vF_2(t,x,v).
\end{aligned}
\]
Substituting these identities into \eqref{equF1}, we obtain
\[
\begin{aligned}
& \pa_tF_2(t,x,v)+T(t)^{\f12}v\cdot\na_xF_2(t,x,v)+\big(-\f12T(t)^{-1}T'(t)v+2T(t)^{-\f12}R(t)\big)\cdot\na_v F_2(t,x,v) \\
&= T(t)^{-\f32}Q(F_2,F_2)(t,x,v)+\f32 T(t)^{-1}T'(t)F_2(t,x,v)+T(t)^{-\f32}\dv \Big(\f{\Pi(v+\om(t))}{|v+\om(t)|}\cdot \na_v F_2\Big)(t,x,v).
\end{aligned}
\]
Moreover, \eqref{VT} and the changes of variables \eqref{F1} and \eqref{F2} give
\[
\begin{aligned}
& \int_{\T^3\times\R^3}F_2(t,x,v)(1,v,|v|^2)dvdx=T(t)^{\f32}\int_{\T^3\times\R^3}F(t,x+H(t),V(t)+T(t)^{\f12}v)(1,v,|v|^2)dvdx \\
&= \int_{\T^3\times\R^3}F(t,x,v)(1,T(t)^{-\f12}(v-V(t)),T(t)^{-1}|v-V(t)|^2)dvdx=(1,0,3).
\end{aligned}
\]

Writing \(G\) in place of \(F_2\), we set
\begin{equation}\label{changeofvariable}
G(t,x,v)=T(t)^{3/2}F\bigl(t,x+H(t),V(t)+T(t)^{1/2}v\bigr),\quad H(t)=\int_0^t V(\tau)\,d\tau,
\end{equation}
Combining the preceding calculation with \eqref{dtV} and \eqref{dtT}, we obtain the following proposition:

\begin{prop}\label{prop2}
The Cauchy problem \eqref{main} with initial data $F_0$ is equivalent to the coupled Cauchy problem for $(G(t,x,v),V(t),T(t))$ as follows:
{
\begin{equation}\label{eqG}
\left\{\begin{aligned}
&\pa_tG(t,x,v)+T(t)^{\f12}v\cdot\na_xG(t,x,v) \\
&\quad{} +\big(-\f12T(t)^{-1}T'(t)v+2T(t)^{-\f12}R(t)\big)\cdot\na_v G(t,x,v) \\
&=T(t)^{-\f32}Q(G,G)(t,x,v) \\
&\quad{} + \f32 T(t)^{-1}T'(t)G(t,x,v)+T(t)^{-\f32}\mathrm{div} \Big(\f{\Pi(v+\om(t))}{|v+\om(t)|}\cdot \na_v G\Big)(t,x,v), \\
&V'(t)=E-2R(t),\quad T'(t)=\f43V(t)\cdot R(t), \\
&R(t)=T(t)^{-1}\int_{\T^3\times\R^3}\f{v+\om(t)}{|v+\om(t)|^3
}G(t,x,v)dxdv
\end{aligned}\right.
\end{equation}
}
 with initial data 
{
\[
\begin{aligned}
&(G_0(x,v),V_0,T_0) =\Big(T_0^{\f32}F_0(x,V_0+T_0^{\f12}v),~\int_{\T^3\times\R^3}vF_0(x,v)dvdx, \\
&\qquad ~\f13\int_{\T^3\times\R^3}|v-V_0|^2F_0(x,v)dvdx\Big).
\end{aligned}
\]
}
 Moreover, $G$ satisfies the conservation laws
\[
\int_{\T^3\times\R^3} G(t,x,v)(1,v,|v|^2)dvdx=(1,0,3)=\int_{\T^3\times\R^3}\mu(1,v,|v|^2)dvdx.
\]
\end{prop}

By Proposition \ref{prop2}, the stability analysis of $M_{V(t),T(t)}$ with respect to \eqref{main} reduces to the stability of $\mu$ with respect to \eqref{eqG}. Let $G:=\mu+g$ and substitute this decomposition into \eqref{eqG}; we obtain the following equation for $g$:
\begin{equation}\label{equg}
\begin{aligned}
&\pa_t g(t,x,v)+T(t)^{\f12}v\cdot\na_x g(t,x,v)+(-\f12 T(t)^{-1}T'(t)v+2T(t)^{-\f12}R(t))\cdot\na_v g(t,x,v) \\
&= T(t)^{-\f32}L(g)+T(t)^{-\f32}Q(g,g)(t,x,v)+\f32 T(t)^{-1} T'(t) g(t,x,v)+T(t)^{-\f32}\dv \Big(\f{\Pi(v+\om(t))}{|v+\om(t)|}\cdot \na_v g\Big) \\
&\quad{} - (-\f12 T(t)^{-1}T'(t)v+2T(t)^{-\f12}R(t))\cdot\na_v \mu+\f32 T(t)^{-1} T'(t) \mu+T(t)^{-\f32}\dv \Big(\f{\Pi(v+\om(t))}{|v+\om(t)|}\cdot \na_v \mu\Big),
\end{aligned}
\end{equation}
where we have used $Q(\mu,\mu) = 0$ and the linear Landau operator is defined by
\begin{equation}\label{linearoperator}
L(g):=Q(\mu,g)+Q(g,\mu).
\end{equation}
Moreover, the moment constraints for $g$ are
\begin{equation}\label{conforg}
\int_{\T^3\times\R^3} g(t,x,v)(1,v,|v|^2)dvdx=0.
\end{equation}
Therefore, to establish Theorem \ref{longtimebehavior1}, it suffices to prove the following theorem for the transformed system.
\begin{thm}\label{longtimebehavior}
Let $k_0\geq17$ be as in Theorem~\ref{longtimebehavior1} and
let $k_1>k_0+3$. Consider the joint Cauchy problem
\eqref{equg}, \eqref{dtV} and \eqref{dtT} with initial data
$(g_0,V_0,T_0)\in X_{k_1}\times\mathbb R^3\times(0,\infty)$,
where $g_0=G_0-\mu$ satisfies
$\int_{\T^3\times\mathbb R^3}(1,v,|v|^2)g_0\,dv\,dx=0$
and $\mu+g_0\geq0$.
There are positive constants $\vep_1,\vep_2$, depending only on
$k_0,k_1,V_0,T_0$, such that if
\[
|E|^{-1}<\vep_1\quad\mbox{and}\quad\|g_0\|_{X_{k_1}}^2<\vep_2,
\]
then the system admits a unique global solution $(g,V,T)$ with
$\mu+g\geq0$, $T>0$, and
$g\in L^\infty([0,S],X_{k_1})\cap L^2([0,S],Y_{k_1})$
for every finite $S>0$. It satisfies
\[
\|g(t)\|_{X_{k_0}}^2
 \ls(1+t)^{-\f{k_1-k_0}{k_1-k_0+3/2}-\left(\f{k_1-k_0}{k_1-k_0+3/2}\right)^2},
 \qquad t\geq1,
\]
with an implicit constant depending only on $k_0,k_1,V_0,T_0$.
The moments $V(t),T(t)$ satisfy \eqref{VTlongtime} for $t\geq1$.
\end{thm}

\section{Notation and function spaces}\label{notation}
We introduce the notation and function spaces used in the proofs.
\subsection{Notation}
  We use the notation $a\ls b$ ($a\gs b$) and $a\ls_c b ~(a\gs_c b)$ to indicate that there is a constant $C$ which is uniform or depends on parameter $c$ and may be different on different lines, such that $a\leq Cb(a\geq Cb)$. We use the notation $a\sim b$ ($a\sim_c b$) whenever $a\ls b$ and $b\ls a$ ($a\ls_c b$ and $a\gs_cb$). Denote $C_{a_1,a_2,\cdots,a_n}$ (or $C(a_1,a_2,\cdots,a_n)$) by a constant depending on parameters $a_1,a_2,\cdots,a_n$. Moreover,  parameter $\varepsilon$ is used  to represent different positive numbers much less than 1 and determined in different cases. $\mathbf{1}_\Om$ is the characteristic function of the set $\Om$. We use $(\cdot,\cdot)_{L^2_v}$, and $(\cdot,\cdot)_{L^2_{x, v}}$ to denote the inner product in $L^2(\R^3)$ and $L^2(\T^3\times\R^3)$, respectively. We also use $\mathbf{I}$ to represent the unit matrix or identity operator.  






\subsection{Function spaces}
We define the function spaces used for the velocity, spatial, and time variables.
\smallskip

$(1)$ \textit{Function spaces in $v$ variable.} Let \( f(v) \) be a function of the variable \( v \in \mathbb{R}^3 \). For any $p\in[1,\infty),q\in\R$, the $L^p_q$
 norm is defined by 
 \[
\|f\|^p_{L^p_q}=\int_{\R^3}|f(v)|^p\<v\>^{pq}dv.
\]
Here we employ the Japanese bracket $\<v\>:= (1+|v|^2)^{1/2}$.

For $m,l\in\R $, we define the weighted Sobolev space $H_l^m$  as follows:
\[
H^m_l:=\Big\{f(v)|\|f\|_{ H^{m}_l}=\|\<D\>^m\<\cdot\>^l f\|_{L^2}<+\infty\Big\},
\]
where $a(D)$ is a pseudo-differential operator with the symbol $a(\xi)$ and is defined as
\[
(a(D)f)(v):=\f1{(2\pi)^3}\int_{\R^6} e^{i(v-v_*)\cdot\xi} a(\xi)f(v_*)\,dv_*\,d\xi.
\]
We remark that $\|\<D\>^m\<\cdot\>^lf\|_{L^2}\sim_{m,l} \|\<\cdot\>^l\<D\>^mf\|_{L^2}$ (see Lemma 5.2 in \cite{HL}). When $m\in\N$, this norm satisfies the equivalence
\begin{equation}\label{hms}
\|f\|^2_{H^m_l}\sim_{m,l}\sum_{|\al|\leq m}\int_{\R^3}\<v\>^{2l}|\pa^\al_v f|^2\,dv.
\end{equation}
In particular, we write $\|f\|_{L^2_l}:=\|f\|_{H^0_{l}},\|f\|_{H^m_v}:=\|f\|_{H^m_0}$ and $\|f\|_{L^2_v}:=\|f\|_{H^0_0}$.


$(2)$ \textit{Function spaces in $x,v$ variables.} Let \( f(x,v) \) be a function of the variable \( (x,v) \in \T^3\times \mathbb{R}^3\). The differential operator on $x$, $\<D_x\>^r,r\in\R$ is defined by
\[
\<D_x\>^r f:=\sum_{q\in\Z^3}\<q\>^r\mathcal{F}_x(f)(q)e^{2\pi iq\cdot x},
\]
where $\mathcal{F}_x$ is the Fourier transform with respect to $x$. For $m,r,l\in\R $, we define the weighted Sobolev space $H^r_xH_l^m$  as follows:
\begin{equation}\label{hrhms}
H^r_xH^{m}_l:=\Big\{f(x,v)|\|f\|^2_{ H^r_xH^{m}_l}=\sum_{q\in\Z^3}\<q\>^{2r}\|\mathcal{F}_x(f)(q)\|^2_{H^m_l}<+\infty\Big\}.
\end{equation}
When $r\in \N$, the norm satisfies the equivalence
\[
\|f\|^2_{H^r_xH^m_l}\sim_r\sum_{|\al|\leq r}\int_{\T^3}\|\pa^\al_x f\|^2_{H^m_l}\,dx.
\]
In particular, we write $\|f\|_{L^2_xH^m_l}=\|f\|_{H^0_xH^m_l},\|f\|_{L^2_{x,v}}:=\|f\|_{H^0_xH^0_0}$.

$(3)$ \textit{Function spaces in $t,x,v$ variables.} Let $f=f(t,x,v)$ and $X$ be a function space in $x,v$ variables. Then $L^p([0,S],X)$ and $ L^\infty([0,S],X)$ are defined as follows:
\[
L^p([0,S],X):=\bigg\{f(t,x,v)\big|\|f\|^p_{ L^p([0,S],X)}=\int_0^S\|f(t)\|^p_{X}dt<+\infty\bigg\},\quad 1\leq p<\infty,
\]
\[
L^\infty([0,S],X):=\Big\{f(t,x,v)|\|f\|_{ L^\infty([0,S],X)}=\mathrm{esssup}_{t\in [0,S]}\|f(t)\|_{X}<+\infty\Big\}.
\]
When the time interval is understood, we abbreviate $L^p([0,S],X)$ as $L_t^pX$, including $p=\infty$.

$(4)$ \textit{Solution spaces.} We shall prove the global well-posedness in weighted Sobolev spaces.
For \(k\in\mathbb R\) and a function \(f(v)\), we define the dissipation norms in \(v\in\mathbb R^3\) as
\begin{equation}\label{dissipationnorm}
\begin{aligned}
&\|f\|^2_{D_1(k)}:=\|f\|^2_{H^1_{k-\f32}}+\|(-\D_{\S^2})^{\f12}f\|^2_{L^2_{k-\f32}}, \\
&\|f\|^2_{D_2(k)}:=\|f\|^2_{D_1(k)}+\|f\|^2_{L^2_{k-\f12}},
\end{aligned}
\end{equation}
where $\D_{\S^2}$ is the Laplace-Beltrami operator. For a function \(f(x,v)\), we further introduce the energy functional and its associated dissipation functional as follows.
\begin{equation}\label{functional}
\begin{aligned}
& \|f\|^2_{X_k}:=\|f\|^2_{L^2_xL^2_k}+\sum_{|\al|=2}\|\pa^\al_x f\|^2_{L^2_xL^2_{k-4|\al|}},\quad k\in\mathbb R; \\
& \|f\|^2_{Y_k}:=\|f\|^2_{L^2_xD_1(k)}+\sum_{|\al|=2}\|\pa^\al_x f\|^2_{L^2_xD_1(k-4|\al|)},\quad k\in\mathbb R; \\
& \|g\|^2_{\mathcal{E}_k}:=\sum_{|\al|=0,2}\|\pa^\al_x g\|^2_{L^2_{x}L^2_k}, \\
&\|g\|^2_{\mathcal{D}_k}:=\sum_{|\al|=0,2}\|\pa^\al_xg\|^2_{L^2_xD_2(k)},~~k\geq0.
\end{aligned}
\end{equation}




\section{Landau operator and auxiliary tools}\label{Landau}
We collect the upper bounds and coercivity estimates for the Landau collision operator that will be used in the subsequent proofs. We begin by introducing the quantities
\begin{equation}\label{aijbidelta}
b_i(z)=\sum_{j=1}^3\pa_ja_{ij}(z)=-2z_i|z|^{-3},\quad c(z)=\sum_{i,j=1}^3\pa_{ij}a_{ij}(z)=-8\pi\de_0(z).
\end{equation}
 By Proposition 2.3 and Proposition 2.4 in \cite{PM}, the matrix $a*\mu$ has a simple eigenvalue $l_1(v)\sim \<v\>^{-3}$ associated with the eigenvector $v$ and a double eigenvalue $l_2(v)\sim \<v\>^{-1}$ associated with eigenspace $v^\perp$, where $l_1(v),l_2(v)$ are defined as follows:
\begin{equation}\label{eigenvalueamu}
\begin{aligned}
l_1(v)&=\int_{\R^3} \Big(1-\big(\f v{|v|}\cdot\f {v_*}{|v_*|}\big)^2\Big)|v_*|^{-1}\mu(v-v_*)dv_*\sim \<v\>^{-3}, \\
l_2(v)&=\int_{\R^3} \Big(1-\f12\Big|\f v{|v|}\times\f {v_*}{|v_*|}\Big|^2\Big)|v_*|^{-1}\mu(v-v_*)dv_*\sim \<v\>^{-1}.
\end{aligned}
\end{equation}
At $v=0$, the eigenvalues are understood by continuous extension. Moreover, the preceding formulas give
\begin{equation}\label{abv}
\begin{aligned}
\sum_{i,j=1}^3a_{ij}*\mu v_iv_j=l_1(v) |v|^2, ~~
\sum_{j=1}^3b_j*\mu v_j=-l_1(v) |v|^2, ~~
\sum_{i=1}^3a_{ii}*\mu=l_1(v)+2l_2(v).
\end{aligned}
\end{equation}

\begin{lem}\label{ghf2}
Recall the dissipation norm $D_1$ in \eqref{dissipationnorm}.	For functions $g,h$ and $f$, the following estimate holds: 
	\begin{equation}\label{upperQ1}
|(Q(g, h),f)_{L^2_v}| \ls \|g\|_{L^2_5}\|h\|_{D_1(0)}\|f\|_{D_1(0)} +\|g\|_{L^2_5}\|h \|_{H^{a_1}_{\om_1}}\|f \|_{H^{a_2}_{\om_2}},
\end{equation}
	where $a_1,a_2\geq0$, $\om_1,\om_2\in[-2,0]$ satisfy $a_1+a_2=1$ and $\om_1+\om_2=-2$. Moreover,
    \begin{equation}\label{upperQ2}
|(Q(g, h),f)_{L^2_v}| \ls\|g\|_{L^2_2}\min\{\|h\|_{L^2_\om}\|f\|_{H^2_{-\om}}, \|h\|_{H^2_\om}\|f\|_{L^2_{-\om}} \}
\end{equation}
    for any $\om\in \R$. As a direct consequence, for any $k\geq0$, we have
    	\begin{equation}\label{fmuf}
\big|\big(Q(g,\mu), f\<v\>^{2k}\big)_{L^2_v}\big|\leq C_k \|g\|_{L^2_{3}}\|f\|_{L^2_v};
\end{equation}
  
\end{lem}
\begin{proof}
The estimate \eqref{upperQ1} follows directly from Proposition 2.2 in \cite{HJL}. To establish \eqref{upperQ2}, we first observe that
\[
\begin{aligned}
(Q(g,h),f)_{L^2_v}&=-\sum_{i,j=1}^3\int_{\mathbb{R}^3}a_{ij}*g\,\partial_ih\,\partial_jf\,dv+\sum_{i=1}^3\int_{\mathbb{R}^3}b_i*g\,h\,\partial_if\,dv \\
&=\sum_{i,j=1}^3\int_{\mathbb{R}^3}a_{ij}*g\,h\,\partial_{ij}f\,dv+2\sum_{i=1}^3\int_{\mathbb{R}^3}b_i*g\,h\,\partial_if\,dv:=Q_1+Q_2.
\end{aligned}
\]
Since \(a(z)\ls|z|^{-1}\), the Cauchy–Schwarz inequality gives
\[
\begin{aligned}
|Q_1| &\lesssim \int_{|v-v_*|\le 1}\frac{1}{|v-v_*|}|g(v_*)\<v\>^\om h\<v\>^{-\om}\nabla^2f|\,dv_*dv \\
&\quad{} + \int_{|v-v_*|>1}\frac{1}{|v-v_*|}|g(v_*)\<v\>^\om h\<v\>^{-\om}\nabla^2f|\,dv_*dv \\
&\lesssim \|g\|_{L^2_2}\|h\|_{L^2_\om}\|\nabla^2f\|_{L^2_{-\om}}\lesssim \|g\|_{L^2_2}\|h\|_{L^2_\om}\|f\|_{H^2_{-\om}}.
\end{aligned}
\]

To estimate \(Q_2\), we use \(b(z)\lesssim |z|^{-2}\) and split the integral into the same two regions:
\[
\begin{aligned}
|Q_2| &\lesssim \int_{|v-v_*|\le 1}\frac{1}{|v-v_*|^2}|g(v_*)\<v\>^\om h(v)\<v\>^{-\om}\nabla f|\,dv_*dv \\
&\quad{} + \int_{|v-v_*|>1}\frac{1}{|v-v_*|^2}|g(v_*)\<v\>^\om h(v)\<v\>^{-\om}\nabla f|\,dv_*dv \\
&\lesssim \|g\|_{L^2}\|\<v\>^\om h\<v\>^{-\om}\nabla f\|_{L^{6/5}}+\|g\|_{L^2_2}\|h\|_{L^2_\om}\|\nabla f\|_{L^2_{-\om}} \\
&\lesssim \|g\|_{L^2}\|h\|_{L^2_\om}\|\nabla f\|_{L^3_{-\om}}+\|g\|_{L^2_2}\|h\|_{L^2_\om}\|\nabla f\|_{L^2_{-\om}} \\
&\lesssim \|g\|_{L^2_2}\|h\|_{L^2_\om}\|f\|_{H^{3/2}_{-\om}}.
\end{aligned}
\]
Here we have used the Hardy–Littlewood–Sobolev inequality, H\"older's inequality \(\|fg\|_{L^{6/5}}\le \|f\|_{L^2}\|g\|_{L^3}\), and the Sobolev embedding \(H^{1/2}\hookrightarrow L^3\). Combining the estimates for \(Q_1\) and \(Q_2\) yields 
\[
|(Q(g, h),f)_{L^2_v}| \ls\|g\|_{L^2_2}\|h\|_{L^2_\om}\|f\|_{H^2_{-\om}}.
\]

Alternatively, integration by parts gives
\[
(Q(g,h),f)_{L^2_v} =\sum_{i,j=1}^3\int_{\mathbb{R}^3}a_{ij}*g\,\pa_{ij}h\,f\,dv+8\pi\int_{\R^3}g\,h\,f\,dv.
\]
Estimating the two terms as above yields
\[
|(Q(g, h),f)_{L^2_v}| \ls\|g\|_{L^2_2}\|h\|_{H^2_\om}\|f\|_{L^2_{-\om}}.
\]
This completes the proof of the lemma.
\end{proof}

\begin{lem}\label{upperboundQ}
Suppose $ k\geq0$ and the non-negative function $G$ satisfies $\|G\|_{L^1_v}>\de, \|G\|_{L^2_5}<\lam$. Then 
\begin{equation}\label{general}
\big(Q(G,f), f\<v\>^{2k}\big)_{L^2_v}\leq-C_1\|f\|^2_{D_1(k)}+C_2\|f\|^2_{L^2_{k-\f12}},
\end{equation}
where the constants $C_1,C_2$ depend on $\de, \lam$ and $k$. 
\end{lem}
\begin{proof}
Integration by parts gives the decomposition 
\[
\begin{aligned}
& \big(Q(G,f), f\<v\>^{2k}\big)_{L^2_v}=-\sum_{i,j=1}^3\int_{\R^3}a_{ij}*G\pa_if \pa_j(f\<v\>^{2k})dv+\sum_{i=1}^3\int_{\R^3}b_i*Gf\pa_i(f\<v\>^{2k})dv \\
& =-\sum_{i,j=1}^3\int_{\R^3}a_{ij}*G\pa_i(f\<v\>^k) \pa_j(f\<v\>^{k})dv+k^2\sum_{i,j=1}^3\int_{\R^3}a_{ij}*Gf^2\<v\>^{2k-4}v_iv_jdv \\
&\quad{} + 4\pi\int_{\R^3}G f^2\<v\>^{2k}dv+k\sum_{i=1}^3\int_{\R^3}b_i*Gf^2\<v\>^{2k-2}v_idv:=G_1+G_2+G_3+G_4.
\end{aligned}
\]

For $G_1$, we first observe that
\[
\begin{aligned}
\int_{\R^3}G|\log G| dv &= \int_{G>1}G\log G dv+\int_{e^{-|v|^2}<G\leq1}G|\log G| dv+\int_{G\leq e^{-|v|^2}}G|\log G| dv \\
&\ls \int_{\R^3}G^2dv+\int_{\R^3} G|v|^2 dv+\int_{G \leq e^{-|v|^2}}G^{\f12}dv\ls C_\lam,
\end{aligned}
\]
since $\|G\|_{L^2_5}<\lam$. Proposition 2.1 in \cite{HJL} therefore gives $G_1\leq -C_1\|f\|^2_{D_1(k)}+C_k\|f\|^2_{L^2_{k-\f32}}$ with $C_1$ depending on $\de$ and $\lam$. For $G_2$, the bound $a(z)\ls |z|^{-1}$ gives 
\[
\begin{aligned}
|G_2| &\leq C_k \Big(\int_{|v-v_*|\leq 1}\f1{|v-v_*|}G(v_*)f^2\<v\>^{2k-2}dv_*dv+\int_{|v-v_*|> 1}\f1{|v-v_*|}G(v_*)f^2\<v\>^{2k-2}dv_*dv\Big) \\
&\leq C_k\|G\|_{L^2_5}\|f\|^2_{L^2_{k-1}}.
\end{aligned}
\]
For $G_3$, we estimate 
\[
\begin{aligned}
&|G_3|\leq C_k\|G\|_{L^2_5}\|f\|^2_{L^4_{k-2}} \leq C_k\|G\|_{L^2_5}\|f\|^2_{H^{3/4}_{k-2}} \\
&\leq \vep\|f\|^2_{H^1_{k-\f32}}+C_\vep \|G\|^4_{L^2_5}\|f\|^2_{L^2_{k-2}},\quad \forall \vep>0.
\end{aligned}
\]
The preceding bounds use H\"older's inequality, the Sobolev embedding $H^{3/4}_{k-2}\hookrightarrow L^4_{k-2}$, and interpolation followed by Young's inequality: $\|f\|_{H^{3/4}_{k-2}}\ls \|f\|^{3/4}_{H^{1}_{k-2}}\|f\|^{1/4}_{L^{2}_{k-2}}\ls \vep \|f\|_{H^{1}_{k-2}}+C_\vep  \|f\|_{L^{2}_{k-2}}$.
For $G_4$, since $b(z)\ls |z|^{-2}$, the Hardy-Littlewood-Sobolev inequality gives
\[
\begin{aligned}
|G_4| &\ls C_k \Big(\int_{|v-v_*|\leq 1}\f1{|v-v_*|^2}G(v_*)f^2\<v\>^{2k-1}dv_*dv+\int_{|v-v_*|> 1}\f1{|v-v_*|^2}G(v_*)f^2\<v\>^{2k-1}dv_*dv\Big) \\
&\ls C_k\|G\|_{L^2_5}(\|f\|^2_{L^{12/5}_{k-2}}+\|f\|^2_{L^2_{k-\f12}})\ls \vep \|f\|^2_{H^1_{k-\f32}}+C_{k,\vep}(1+\|G\|^4_{L^2_5})\|f\|^2_{L^2_{k-\f12}},
\end{aligned}
\]
where we also use interpolation and the fact that $\<v\>\sim \<v_*\>$ when $|v-v_*|\leq 1$.

Combining these estimates and using $\|f\|_{H^1_{k-\f32}}\leq \|f\|_{D_1(k)}$, we obtain the desired result by choosing $\vep>0$ sufficiently small.

This completes the proof.
\end{proof}

We next establish weighted estimates with velocity cutoffs.  
\begin{defi}[Cutoff function]\label{cutofffunction}
Let $0\leq\chi\leq 1$ be a smooth radial cutoff function such that $\chi(r)=1$ for $|r|\leq 1$, $\chi(r)=0$ for $|r|>2$ and $|\chi'|+|\chi''|<10$. Let $\chi_N(v):=\chi(|v|^2/N)$ and $\chi^c_N(v):=1-\chi_N(v)$ with $N>0$. We write $\chi_N'(v):=\chi'(|v|^2/N)$ and $(\chi_N^c)'(v):=-\chi_N'(v)$. Moreover, we denote by $\tilde{\chi}_N(v):=\tilde{\chi}(|v|^2/N)$, where $\tilde{\chi}$ is a smooth radial function satisfying $0\leq\tilde{\chi}\leq10$  and $\tilde{\chi}(r)=10$ on $\operatorname{supp}\chi'$ (for instance, one may take $\tilde{\chi}(r)=10\chi(|r|^2/4)-10\chi(2|r|^2)$ and then $\operatorname{supp}\tilde\chi\cap[0,\infty)\subset[1/4,4]$). 
\end{defi}

\begin{lem}\label{cutoffQcom}
 For any real $k\geq5$, any $N\geq1$, and any functions $g,h,f$, the following estimates hold: 
		\begin{equation}\label{cutoff1}
\begin{aligned}
\big(Q(\mu,f), f\chi^2_N\<v\>^{2k}\big)_{L^2_v} &\leq -\lam_0\|\chi_Nf\|^2_{D_1(k)}-\mc_0k\|\chi_N f\|^2_{L^2_{k-\f32}}+\f {\mC_0}k\|\tilde{\chi}_Nf\|^2_{L^2_{k-\f32}} \\
&\quad{} + C_k\|\chi_N f\|^2_{L^2_v}+C_k\|\tilde{\chi}_N f\|^2_{L^2_{k-\f52}},
\end{aligned}
\end{equation}
	where the positive constants $\lam_0$, $\mc_0$ and $\mC_0$ are independent of $k$ and $N$(hereafter, we treat them as uniform constants). Moreover,
	\begin{equation}\label{cutoffgfh}
\begin{aligned}
&|(Q(g,f),f\chi_N^2\<v\>^{2k})|\leq C_k\|g\|_{L^2_5}\big(\|\chi_N f\|^2_{D_{1}(k)}+\|\tilde{\chi}_N f\|^2_{L^2_{k-\f32}}\big), \\
&|(Q(g,h),f\chi_N^2\<v\>^{2k})| \\
&\leq C_k\|g\|_{L^2_5}\big(\|\chi_N h\|_{D_{1}(k+1)}+\|\tilde{\chi}_N h\|_{L^2_{k-\f32}}\big)\big(\|\chi_N f\|_{D_{1}(k)}+\|\tilde{\chi}_N f\|_{L^2_{k-\f32}}\big).
\end{aligned}
\end{equation}
The above estimates remain valid with $\chi_N$ replaced by $\chi^c_N$ and $\tilde{\chi}_N$ unchanged. Furthermore, letting $N\to\infty$ yields the corresponding estimates with $\chi_N\equiv1$, in which case the terms involving $\tilde{\chi}_N$ vanish, as they originate from $\nabla\chi_N$.
\end{lem}
\begin{proof}
We first prove \eqref{cutoff1}, retaining some dissipation for later use.
Set \(m=\chi_N\langle v\rangle^k\). Integration by parts gives the exact identity
\[
\begin{aligned}
(Q(\mu,f),fm^2) &= -\int_{\mathbb R^3}(a*\mu)\nabla_v(fm)\cdot\nabla_v(fm)\,dv +\int_{\mathbb R^3}f^2(a*\mu)\nabla_v m\cdot\nabla_v m\,dv \\
&\quad{} + 4\pi\int_{\mathbb R^3}\mu f^2m^2\,dv +\int_{\mathbb R^3}f^2m(b*\mu)\cdot\nabla_v m\,dv .
\end{aligned}
\]
The first term controls the weighted derivatives, but we retain a zero-order
remainder. Indeed, \eqref{eigenvalueamu}, the radiality of the weight, and
\(\nabla_v(\langle v\rangle^{-3/2}fm)
=\langle v\rangle^{-3/2}\nabla_v(fm)
-\frac32v\langle v\rangle^{-7/2}fm\) imply
\[
\int_{\mathbb R^3}(a*\mu)\nabla_v(fm)\cdot\nabla_v(fm)\,dv \geq c\|\chi_Nf\|_{D_1(k)}^2 -C\|\chi_Nf\|_{L^2_{k-3/2}}^2 ,
\]
with \(c,C>0\) independent of \(k,N\).

We next combine the terms where the derivative falls on
\(\langle v\rangle^k\). By \eqref{abv}, their sum is
\[
\int_{\mathbb R^3}l_1(v)|v|^2 f^2\chi_N^2 \left(k^2\langle v\rangle^{2k-4} -k\langle v\rangle^{2k-2}\right)\,dv \leq-c k\|\chi_Nf\|_{L^2_{k-3/2}}^2 +C_k\|\chi_Nf\|_{L^2_{k-5/2}}^2 +C_k\|\chi_Nf\|_{L^2}^2 .
\]
Thus the negative zero-order term comes from the drift.
The Maxwellian term is bounded by \(C_k\|\chi_Nf\|_{L^2}^2\).
For the remaining terms, use
\(\nabla_v\chi_N=2v\chi'(|v|^2/N)/N\).
On its support, \(N\leq |v|^2\leq2N\), and hence
\[
\begin{aligned}
&\left|\int f^2\left[ 2k\chi_N\langle v\rangle^{2k-2}(a*\mu)v\cdot\nabla_v\chi_N +\langle v\rangle^{2k}(a*\mu)\nabla_v\chi_N\cdot\nabla_v\chi_N \right]dv\right| \leq C_k\Bigl(\|\chi_Nf\|_{L^2_{k-5/2}}^2 +\|\widetilde\chi_Nf\|_{L^2_{k-5/2}}^2\Bigr), \\
&\left|\int f^2\chi_N\langle v\rangle^{2k} (b*\mu)\cdot\nabla_v\chi_N\,dv\right| \leq C\|\chi_Nf\|_{L^2_{k-3/2}} \|\widetilde\chi_Nf\|_{L^2_{k-3/2}}\leq \varepsilon k\|\chi_Nf\|_{L^2_{k-3/2}}^2 +\frac{C}{\varepsilon k}\|\widetilde\chi_Nf\|_{L^2_{k-3/2}}^2 .
\end{aligned}
\]
Here \(N\geq1\), as required in this collision estimate.
For every \(\varepsilon>0\), splitting at a sufficiently large radius depending
only on \(k,\varepsilon\) gives
\[
C_k\|\chi_Nf\|_{L^2_{k-5/2}}^2 \leq\varepsilon k\|\chi_Nf\|_{L^2_{k-3/2}}^2 +C_{k,\varepsilon}\|\chi_Nf\|_{L^2}^2 .
\]
Taking a sufficiently small fixed part of the gradient dissipation in the
first estimate, and then choosing \(\varepsilon\) small, proves
\[
\begin{aligned}
(Q(\mu,f),f\chi_N^2\langle v\rangle^{2k}) &\leq -4\lambda_0\|\chi_Nf\|_{D_1(k)}^2 -4\mathbf c_0 k\|\chi_Nf\|_{L^2_{k-3/2}}^2 +\frac{\mathbf C_0}{4k}\|\widetilde\chi_Nf\|_{L^2_{k-3/2}}^2 \\
&\quad{} + C_k\|\chi_Nf\|_{L^2}^2 +C_k\|\widetilde\chi_Nf\|_{L^2_{k-5/2}}^2 .
\end{aligned}
\]
Here \(\lambda_0,\mathbf c_0,\mathbf C_0>0\) are independent of \(k,N\);
we have chosen the first two sufficiently small and the third sufficiently
large. This proves \eqref{cutoff1} with additional dissipation retained for later use.
These factors also allow its use at the weights \(k\) and \(k-8\) in the
spatial energy sums, since \(k-8\geq k/2\) for \(k\geq17\).
The same calculation applies
to \(\chi_N^c\), since its derivatives have the same support and absolute
bounds. The uncut estimate follows by letting \(N\to\infty\), as stated above.

\medskip

We now prove \eqref{cutoffgfh}. A direct computation gives 
	\[
\begin{aligned}
(Q(g,h),fm^2)_{L^2_v} &= -\sum_{i,j=1}^3\int_{\R^3}a_{ij}*g \pa_i h\pa_j (fm^2)dv+\sum_{i=1}^3\int_{\R^3}b_i*g h\pa_i(fm^2)dv \\
&= -\sum_{i,j=1}^3\int_{\R^3}a_{ij}*g \pa_i (hm)\pa_j (fm)dv-\sum_{i,j=1}^3\int_{\R^3}a_{ij}*g\pa_i h f\pa_j(m^2)dv \\
&\quad{} + \sum_{i=1}^3\int_{\R^3}b_i*g h\pa_i f m^2dv+\f12\sum_{i=1}^3\int_{\R^3} b_i*g hf \pa_i(m^2)dv-\sum_{i,j=1}^3\int_{\R^3}a_{ij}*ghfm\pa_{ij}mdv \\
&=: T_1+T_2+T_3+T_4+T_5.
\end{aligned}
\]
	
For $T_1$, the diffusion estimate in Step 1 of the proof of Lemma 4.3 in \cite{CM1} gives
\[
|T_1|\leq C\|g\|_{L^2_5}\|mh\|_{D_1(0)}\|m f\|_{D_1(0)}\leq C_k\|g\|_{L^2_5}\|\chi_N h\|_{D_1(k)}\|\chi_N f\|_{D_1(k)}.
\]
	
For $T_2$, we use the product rule to write 
\[
\begin{aligned}
T_2 &= -2\sum_{i,j=1}^3\int_{\R^3}a_{ij}*g\pa_i(m h) f\pa_jmdv+2\sum_{i,j=1}^3\int_{\R^3}a_{ij}*g h f\pa_i m\pa_jmdv \\
&=: T_{2,1}+T_{2,2}.
\end{aligned}
\]
Since $a_{ij}(v)\ls|v|^{-1}, \sum_{i=1}^3a_{ij}v_i=0$ and $\pa_j m=k\chi_N\<v\>^{k-2}v_j+\f2 Nv_j\chi'_N\<v\>^k$, we have
\[
|T_{2,1}|\ls \|g\|_{L^2_5}\|\chi_N h\|_{D_1(k)}\big(\|\chi_N f\|_{D_{1}(k)}+\|\tilde{\chi}_N f\|_{L^2_{k-\f32}}\big).
\]
Similarly,
\[
|T_{2,2}|\ls \|g\|_{L^2_5}\big(\|\chi_N h\|_{D_{1}(k)}+\|\tilde{\chi}_N h\|_{L^2_{k-\f32}}\big)\big(\|\chi_N f\|_{D_{1}(k)}+\|\tilde{\chi}_N f\|_{L^2_{k-\f32}}\big).
\]
Hence,
\[
|T_{2}|\ls \|g\|_{L^2_5}\big(\|\chi_N h\|_{D_{1}(k)}+\|\tilde{\chi}_N h\|_{L^2_{k-\f32}}\big)\big(\|\chi_N f\|_{D_{1}(k)}+\|\tilde{\chi}_N f\|_{L^2_{k-\f32}}\big).
\]

For $T_3$, we similarly write 
\[
T_3 = \sum_{i=1}^3\int_{\R^3}b_i*g hm \pa_i (f m)dv-\sum_{i=1}^3\int_{\R^3}b_i*g mh f\pa_i mdv=:T_{3,1}+T_{3,2}.
\]
Since $b_i(v)\ls |v|^{-2}$, the Hardy-Littlewood-Sobolev inequality yields
\[
|T_{3,1}|\ls \|g\|_{L^2_5}\|\chi_N h\|_{D_1(k+1)}\|\chi_N f\|_{D_1(k)}
\]
and similarly,
\[
|T_{3,2}|\ls \|g\|_{L^2_5}\|\chi_N h\|_{D_1(k)}(\|\chi_N f\|_{D_1(k)}+\|\tilde{\chi}_N f\|_{L^2_{k-\f32}}).
\]
Therefore,
\[
|T_3|\ls \|g\|_{L^2_5}\|\chi_N h\|_{D_1(k+1)}(\|\chi_N f\|_{D_1(k)}+\|\tilde{\chi}_N f\|_{L^2_{k-\f32}}).
\]

The estimate for $T_4$ is analogous to that of $T_{3,2}$:
\[
|T_{4}|\ls \|g\|_{L^2_5}\|\chi_N h\|_{D_1(k)}(\|\chi_N f\|_{D_1(k)}+\|\tilde{\chi}_N f\|_{L^2_{k-\f32}})
\]

$T_5$ is estimated in the same way as $T_{2,2}$:
\[
|T_5|\ls \|g\|_{L^2_5}\big(\|\chi_N h\|_{D_{1}(k)}+\|\tilde{\chi}_N h\|_{L^2_{k-\f32}}\big)\big(\|\chi_N f\|_{D_{1}(k)}+\|\tilde{\chi}_N f\|_{L^2_{k-\f32}}\big)
\]
This proves the second inequality in \eqref{cutoffgfh}.

When $h=f$, a sharper estimate for $T_3$ is available. Indeed, 
\[
T_3 = \f12\sum_{i=1}^3\int_{\R^3}b_i*g \pa_i(f^2) m^2dv = 4\pi \int_{\R^3}g f^2m^2dv-\sum_{i=1}^3\int_{\R^3} b_i*g f^2m\pa_i m dv.
\]
It follows that
\[
|T_3|\leq C_k \|g\|_{L^2_5}\|\chi_N f\|^2_{D_1(k)}+C_k \|g\|_{L^2_5}\|\chi_N f\|_{D_1(k)}\|\tilde{\chi}_N f\|_{L^2_{k-\f32}}.
\]
This yields the first inequality in \eqref{cutoffgfh} and completes the proof of the lemma.
\end{proof}
	
\begin{lem}\label{QGgg}
	For $k\geq5$ and $N\geq1$, there exist sufficiently large $A$ and $M$, depending only on $k$, such that
	\begin{equation}\label{LAM}
\begin{aligned}
& \big((L-A\chi_M)f,f\chi^2_N\<v\>^{2k}\big)_{L^2_v}+\big(Q(g,f),f\chi^2_N\<v\>^{2k}\big)_{L^2_v} \\
&\leq -{\lam_0} \|\chi_Nf\|^2_{D_1(k)}-\mc_0k\|\chi_N f\|^2_{L^2_{k-\f32}} \\
&\quad{} + C_k\|g\|_{L^2_{5}}\|\chi_Nf\|^2_{D_1(k)}
+\f{\mC_0}k\|\tilde{\chi}_N f\|^2_{L^2_{k-\f32}}+C_k\|g\|_{L^2_5}\|\tilde{\chi}_N f\|^2_{L^2_{k-\f32}} \\
&\quad{} +C_k\|\chi^c_Nf\|^2_{L^2_3}+C_k\|\tilde{\chi}_N f\|^2_{L^2_{k-\f52}}
\end{aligned}
\end{equation}
	with constants $\lam_0,\mc_0,\mC_0$ defined in Lemma \ref{cutoffQcom} and $C_k>0$.
\end{lem}
\begin{proof}
	We use the stronger coercivity retained in the proof of Lemma~\ref{cutoffQcom} and the definition of $L$ in \eqref{linearoperator}. Together with \eqref{fmuf} in Lemma~\ref{ghf2} and the two estimates \eqref{cutoffgfh} and \eqref{cutoff1} in Lemma~\ref{cutoffQcom}, these give the following bound for the left-hand side of \eqref{LAM}:
	\[
\begin{aligned}
&\quad{} - 2\lam_0\|\chi_Nf\|^2_{D_1(k)}-2\mc_0k\|\chi_N f\|^2_{L^2_{k-\f32}}+C_k\|g\|_{L^2_{5}}\|\chi_Nf\|^2_{D_1(k)}+C_k\|\chi_Nf\|^2_{L^2_3} \\
&\quad{} -A\int_{\R^3}\chi_M f^2\chi_N^2\<v\>^{2k}dv \\
&\quad{} + \f{\mC_0}{4k}\|\tilde{\chi}_N f\|^2_{L^2_{k-\f32}}+C_k\|g\|_{L^2_5}\|\tilde{\chi}_N f\|^2_{L^2_{k-\f32}}+C_k\|\chi^c_Nf\|^2_{L^2_3} +C_k\|\tilde{\chi}_N f\|^2_{L^2_{k-\f52}}.
\end{aligned}
\]
	Here we have used $\|f\|^2_{L^2_3}\leq2\|\chi_N f\|^2_{L^2_3}+2\|\chi_N^c f\|^2_{L^2_3}$. 

Since \(\chi_M(v)=1\) for \(|v|\leq\sqrt M\), and \(k\geq5\), we have
\[
\begin{aligned}
C_k\|\chi_N f\|_{L^2_3}^2 -A\int\chi_M f^2\chi_N^2\langle v\rangle^{2k}\,dv &\leq C_k\int_{|v|>\sqrt M}f^2\chi_N^2\langle v\rangle^6\,dv \\
&\quad+\big(C_k(1+M)^3-A\big) \int_{|v|\leq\sqrt M}\chi_N^2f^2\,dv \\
&\leq C_kM^{-1/2}\|\chi_Nf\|_{L^2_{k-3/2}}^2 \leq\lambda_0\|\chi_Nf\|_{D_1(k)}^2 ,
\end{aligned}
\]
provided \(M\geq1\) is sufficiently large and \(A=C_k(1+M)^3\).
The reserved factor \(2\lambda_0\) leaves the coefficient
\(\lambda_0\) in \eqref{LAM}. All choices are independent of \(N\).
\end{proof}

\begin{lem}\label{linearizedoperator}
In contrast to the linear Landau operator \eqref{linearoperator}, we define the linearized Landau operator as
\begin{equation}\label{linearized}
\cL (f):=\mu^{-\f12}\big(Q(\mu,\mu^{\f12}f)+Q(\mu^{\f12}f,\mu)\big)\quad\mbox{and}\quad \Ga(g,h):=\mu^{-\f12}Q(\mu^{\f12}g,\mu^{\f12} h).
\end{equation}
	Recall the dissipation norm $D_2$ in \eqref{dissipationnorm}. For any real $k\geq0$  and  functions \(g,h,f\), we have
	\begin{align}
& \big(\cL(f),f\<v\>^{2k}\big)_{L^2_v}\leq-\lam_0\|f\|^2_{D_2(k)}+C_k\|f\|^2_{L^2_v},\label{cLf} \\
& \big|\big(\Ga(g,h),f\<v\>^{2k}\big)_{L^2_v}\big|\leq C_k \|g\|_{L^2_v}\|h\|_{D_2(k)}\|f\|_{D_2(k)}\label{Gaghf}
\end{align}
	with some constants $\lam_0>0$ and $C_k>0$. In addition,
	\begin{equation}\label{spectralgap}
\big(\cL(f),f\big)_{L^2_v}\leq-\lam_0\|(\mathbf{I-P})f\|^2_{D_2(0)},
\end{equation}
	 where the macroscopic projection $\mathbf{P}$ is given by 
	\begin{equation}\label{mP}
\mP f=[a^f+b^f\cdot v+c^f(|v|^2-3)]\mu^{\f12}
\end{equation}
	with
	\begin{equation}\label{afbfcf}
a^f=\int_{\R^3}f(v)\mu^{\f12}(v)dv,~b^f=\int_{\R^3}f(v)v\mu^{\f12}(v)dv\quad\mbox{and}\quad c^f=\f16\int_{\R^3}f(v)(|v|^2-3)\mu^{\f12}(v)dv.
\end{equation}
\end{lem}
\begin{proof}
Estimate (2.6) in \cite{CM1} gives
\[
\big(\cL(f),f\big)_{L^2_v}\ls -\|\<v\>^{-\f12}(\mathbf{I-P})f\|^2_{L^2_v}-\|\<v\>^{-\f32}\tilde\na_v (\mathbf{I-P})f\|^2_{L^2_v},
\]
where the anisotropic gradient $\tilde\na_v f$ of a function $f$ is defined by
\[
\tilde\na_v f=P_v\na_v f+\<v\>(\mathbf{I}-P_v)\na_v f
\]
with the projection operator $P_v$ on the $v-$direction defined by 
$
P_v\xi=\Big(\xi\cdot \f v{|v|}\Big)\f v{|v|},\forall \xi\in \R^3.
$
A direct calculation yields
\[
\begin{aligned}
\|\<v\>^{-\f32}\tilde \na_v f\|^2_{L^2_v} &\gs \|\na_v f\|^2_{L^2_{-\f32}}+\|(\mathbf{I}-P_v)\na_v f\|^2_{L^2_{-\f12}} \\
&\gs \|\na_v f\|^2_{L^2_{-\f32}}+\|v\times \na_v f\|^2_{L^2_{-\f32}}\gs \|\na_v f\|^2_{L^2_{-\f32}}+\|(-\D_{\S^2})^{\f12} f\|^2_{L^2_{-\f32}}.
\end{aligned}
\]
Thus \eqref{spectralgap} follows.

Next, we prove \eqref{cLf} and \eqref{Gaghf} in turn. The definition of the linearized Landau operator $\cL$ gives
\[
\big(\cL(f),f\<v\>^{2k}\big)_{L^2_v}=\big(Q(\mu,\mu^{\f12}f),f\mu^{-\f12}\<v\>^{2k}\big)_{L^2_v}+\big(Q(\mu^{\f12}f,\mu),f\mu^{-\f12}\<v\>^{2k}\big)_{L^2_v}:=\cL_1+\cL_2.
\]

For $\cL_2$, the argument used in the proof of the estimate in \eqref{fmuf} gives
\begin{equation}\label{cL2}
\cL_2\leq C_k\|f\|^2_{L^2_v}.
\end{equation}

For $\cL_1$, we write
\[
\cL_1=\sum_{i,j=1}^3\int_{\R^3}(a_{ij}*\mu)\pa_{ij}(\mu^{\f12}f) f\mu^{-\f12}\<v\>^{2k}dv+8\pi \int_{\R^3}\mu f^2\<v\>^{2k}dv.
\]
Let $g=f\<v\>^k$. Integration by parts gives
\[
\begin{aligned}
\cL_1 &= -\sum_{i,j=1}^3\int_{\R^3}(a_{ij}*\mu)\pa_ig\pa_jgdv+\Big(8\pi\int_{\R^3}\mu g^2dv-\sum_{j=1}^3\int_{\R^3}(b_j*\mu)g\pa_j gdv\Big) \\
&\quad{} - \int_{\R^3}\Big(\sum_{i,j=1}^3(a_{ij}*\mu)\pa_i(\mu^{\f12}\<v\>^{-k})\pa_j(\mu^{-\f12}\<v\>^k)+\sum_{i=1}^3(b_i*\mu)\pa_i(\mu^{\f12}\<v\>^{-k})\mu^{-\f12}\<v\>^k\Big)g^2dv \\
&:= \cL_{1,1}+\cL_{1,2}+\cL_{1,3}.
\end{aligned}
\]
For $\cL_{1,1}$,  the estimate for \(G_1\) in Lemma \ref{upperboundQ} yields a constant $\lam_0>0$ such that
\begin{equation}\label{cL11}
\cL_{1,1}\leq -\lam_0 (\|\na g\|^2_{L^2_{-\f32}}+\|(-\D_{\S^2})^{\f12}g\|^2_{L^2_{-\f32}}).
\end{equation}
To estimate $\cL_{1,2}$ and $\cL_{1,3}$, we first use \eqref{aijbidelta} and integration by parts to obtain
\[
\cL_{1,2}=8\pi\int_{\R^3}\mu g^2dv+\f12\int_{\R^3}(c*\mu) g^2dv=4\pi\int_{\R^3}\mu g^2 dv.
\]
By \eqref{abv}, the combined terms satisfy
\begin{equation}\label{cL1213}
\begin{aligned}
\cL_{1,2}+\cL_{1,3} &= \int_{\R^3}\Big(\sum_{i,j=1}^3(a_{ij}*\mu)v_iv_j(\f12+k\<v\>^{-2})^2+\sum_{i=1}^3(b_i*\mu)v_i(\f12+k\<v\>^{-2})+4\pi\mu\Big)g^2dv \\
&= \int_{\R^3}\Big((\f12+k\<v\>^{-2})^2|v|^2l_1(v)-(\f12+k\<v\>^{-2})|v|^2l_1(v)+4\pi\mu\Big)g^2dv \\
&\ls -\f14\int_{\R^3}\<v\>^{-1}g^2 dv+C_k\int_{\R^3}\<v\>^{-3}g^2dv.
\end{aligned}
\end{equation}
 Combining \eqref{cL2}, \eqref{cL11} and \eqref{cL1213} with $\|g\|^2_{L^2_{-\f32}}\leq \vep\|g\|^2_{L^2_{-\f12}}+C_\vep\|g\|^2_{L^2_{-k}}$, we obtain \eqref{cLf}.

Finally, we prove \eqref{Gaghf}. Let $\mu^{\f12} g=\tilde{g},\mu^{\f12} h=\tilde{h}$ and $\<v\>^k f=\tilde{f}$. We decompose the trilinear form as 
\[
\begin{aligned}
& \big(\Ga(g,h),f\<v\>^{2k}\big)_{L^2_v}=\big(Q(\tilde{g},\tilde{h}),\tilde{f}\mu^{-\f12}\<v\>^{k}\big)_{L^2_v} \\
&= \big(Q(\tilde{g},\mu^{-\f12}\<v\>^{k}\tilde{h}),\tilde{f}\big)_{L^2_v}+\big(\mu^{-\f12}\<v\>^{k}Q(\tilde{g},\tilde{h})-Q(\tilde{g},\mu^{-\f12}\<v\>^{k}\tilde{h}),\tilde{f}\big)_{L^2_v}:=\Ga_1+\Ga_2.
\end{aligned}
\]

For $\Ga_1$, Lemma \ref{ghf2} gives 
\begin{equation}\label{Ga1}
|\Ga_1|\ls \|\tilde{g}\|_{L^2_5}\|\mu^{-\f12}\<v\>^k\tilde{h}\|_{D_2(0)}\|\tilde{f}\|_{D_2(0)}\leq C_k \|g\|_{L^2_v}\|h\|_{D_2(k)}\|f\|_{D_2(k)}.
\end{equation}

For $\Ga_2$,  integration by parts gives 
\[
\begin{aligned}
\Ga_2 &= 2\sum_{i=1}^3\int_{\R^3}(b_i*\tilde{g})\pa_i(\mu^{-\f12}\<v\>^k)\tilde{h}\tilde{f}dv+\sum_{i,j=1}^3\int_{\R^3}(a_{ij}*\tilde{g})\pa_{ij}(\mu^{-\f12}\<v\>^k)\tilde{h}\tilde{f}dv \\
&\quad{} + 2\sum_{i,j=1}^3\int_{\R^3}(a_{ij}*\tilde{g})\pa_i(\mu^{-\f12}\<v\>^k)\pa_j\tilde{f}\tilde{h}dv:=\Ga_{2,1}+\Ga_{2,2}+\Ga_{2,3}.
\end{aligned}
\]
We estimate the first term as follows: 
\begin{equation}\label{Ga21}
\begin{aligned}
|\Ga_{2,1}| &\ls \Big|\int_{\R^3}\f{(v-v_*)v}{|v-v_*|^3}\tilde{g}(v_*)(\f12+k\<v\>^{-2})\mu^{-\f12}\<v\>^k\tilde{h}\tilde{f}(\mathbf{1}_{|v-v_*|\leq1}+\mathbf{1}_{|v-v_*|>1})dv_*dv\Big| \\
&\ls C_k\|\tilde{g}\|_{L^2_6}\|\mu^{-\f12}\<v\>^k\tilde{h}\|_{L^2_{-\f12}}\big(\|\tilde{f}\|_{H^1_{-\f32}}+\|\tilde{f}\|_{L^2_{-\f12}}\big).
\end{aligned}
\end{equation}
Here we use the Hardy--Littlewood--Sobolev inequality, $|\f12+k\<v\>^{-2}|\leq C_k$ and the facts that $\<v_*\>\sim\<v\>$ when $|v-v_*|\leq 1$ and $|v-v_*|\geq \<v_*\>^{-1}\<v\>$ when $|v-v_*|>1$. 

For $\Ga_{2,2}$, using the notation $\tilde g_\ast = \tilde g(v_\ast)$, the relation $\sum_{i,j=1}^3a_{ij}(v-v_*)v_iv_j=\sum_{i,j=1}^3a_{ij}(v-v_*)(v_*)_i(v_*)_j$ and $(\mathbf{I} - z \otimes z/|z|^2):(x\otimes y)=(z/|z|\times x)\cdot(z/|z|\times y)$, we have  
\begin{equation}\label{Ga22}
\begin{aligned}
|\Ga_{2,2}| &\leq C_k\big|\int_{\R^6}\f{|(v-v_*)\times v_*|^2}{|v-v_*|^3}\tilde{g}_*\tilde{h}\mu^{-\f12}\<v\>^k \tilde{f}dv_*dv\big| \\
&\quad{} +C_k\big|\int_{\R^6}|v-v_*|^{-1}\<v_*\>^2\tilde g_* \tilde h\mu^{-\f12}\<v\>^k \tilde f dv_*dv\big| \\
&\leq C_k\int_{\R^6}|v-v_*|^{-1}(\mathbf{1}_{|v-v_*|\leq1}+\mathbf{1}_{|v-v_*|\geq1})\lr{v_*}^2|\tilde{g}_*||\tilde{h}\mu^{-\f12}\<v\>^k||\tilde{f}|dv_*dv \\
&\leq C_k \|\tilde{g}\|_{L^2_5}\|\mu^{-\f12}\<v\>^k\tilde{h}\|_{L^2_{-1/2}}\|\tilde{f}\|_{L^2_{-1/2}}.
\end{aligned}
\end{equation}

 For $\Ga_{2,3}$, a direct computation gives
\[
\begin{aligned}
|\Ga_{2,3}| &\leq C_k\bigg[\int_{\R^6}\f{|(v-v_*)\times \na \tilde{f}|\cdot|(v-v_*)\times v_*|}{|v-v_*|^3}\mathbf{1}_{|v-v_*|<1}|\tilde{g}_*|\,|\tilde{h}|\mu^{-\f12}\<v\>^kdv_*dv \\
&\quad{} + \int_{\R^6}\f{|(v-v_*)\times \na \tilde{f}|\cdot|(v-v_*)\times v_*|}{|v-v_*|^3}\mathbf{1}_{\substack{|v-v_*|\geq1\\|v-v_*|\leq\f{1}{2}|v|} }|\tilde{g}_*|\,|\tilde{h}|\mu^{-\f12}\<v\>^kdv_*dv \\
&\quad{} + \int_{\R^6}\f{|(v-v_*)\times \na \tilde{f}|\cdot|(v-v_*)\times v_*|}{|v-v_*|^3}\mathbf{1}_{\substack{|v-v_*|\geq1\\|v-v_*|>\f{1}{2}|v|} }|\tilde{g}_*|\,|\tilde{h}|\mu^{-\f12}\<v\>^kdv_*dv\bigg] \\
&:= \Ga_{2,3,1}+\Ga_{2,3,2}+\Ga_{2,3,3}.
\end{aligned}
\]
 The explicit form of the weight gives 
{
\[
\begin{aligned}
&|\Ga_{2,3,1}| \leq C_k \iint_{\R^6}|v-v_*|^{-1}\mathbf{1}_{|v-v_*|< 1}|v_*||\tilde{g}_*||\na \tilde{f}||\tilde{h}\mu^{-\f12}\<v\>^k| dv_*dv \\
&\leq C_k \|\tilde{g}\|_{L^2_5}\|\mu^{-\f12}\<v\>^k\tilde{h}\|_{L^2_{-1/2}}\|\na \tilde{f}\|_{L^2_{-3/2}}.
\end{aligned}
\]
}
In the regime $|v-v_*|\leq\f{1}{2}|v|$, we have  $|v|\sim |v_*|$, which implies that
{
\[
\begin{aligned}
&|\Ga_{2,3,2}|\leq C_k\iint_{\R^6}|v-v_*|^{-1}\mathbf{1}_{\substack{|v-v_*|\geq1\\|v-v_*|\leq\f{1}{2}|v|}}|v_*||\tilde{g}_*||\na \tilde{f}||\mu^{-\f12}\<v\>^k\tilde{h}|dv_*dv \\
&\leq C_k \|\tilde{g}\|_{L^2_5}\|\mu^{-\f12}\<v\>^k\tilde{h}\|_{L^2_{-1/2}}\|\na \tilde{f}\|_{L^2_{-3/2}}.
\end{aligned}
\]
}
In the remaining region $|v-v_*|\geq1, |v-v_*|> \f{1}{2}|v|$, we have  $|v-v_*|^{-1}\ls\<v\>^{-1}$. Hence,
\[
\begin{aligned}
|\Ga_{2,3,3}| &\leq C_k\iint_{\R^6}\f{|v\times\na \tilde{f}|}{\<v\>^2}|v_*||\tilde{g}_*||\mu^{-\f12}\<v\>^k\tilde{h}|dv_*dv \\
&\quad{} +\iint_{\R^6}\f{|\na \tilde f||v_*|^2}{\<v\>^2}|\tilde{g}_*||\mu^{-\f12}\<v\>^k\tilde{h}|dv_*dv \\
&\leq C_k \|\tilde{g}\|_{L^2_5}\|\mu^{-\f12}\<v\>^k\tilde{h}\|_{L^2_{-1/2}}\big(\|(-\D_{\S^2})^{\f12}\tilde{f}\|_{L^2_{-3/2}}+\|\na \tilde{f}\|_{L^2_{-3/2}}\big).
\end{aligned}
\]
Combining the bounds on these regions, we obtain
\begin{equation}\label{Ga23}
|\Ga_{2,3}|\leq C_k\|\tilde{g}\|_{L^2_5}\|\mu^{-\f12}\<v\>^k\tilde{h}\|_{L^2_{-1/2}}\big(\|(-\D_{\S^2})^{\f12}\tilde{f}\|_{L^2_{-3/2}}+\|\na \tilde{f}\|_{L^2_{-3/2}}\big).
\end{equation}
Together with \eqref{Ga21} and \eqref{Ga22}, this implies that
\begin{equation}\label{Ga2}
\begin{aligned}
|\Ga_2| &\leq C_k\|\tilde{g}\|_{L^2_6}\|\mu^{-\f12}\<v\>^k\tilde{h}\|_{L^2_{-1/2}} \\
&\qquad\mathbin{\times} \big(\|(-\D_{\S^2})^{\f12}\tilde{f}\|_{L^2_{-3/2}}+\|\na \tilde{f}\|_{L^2_{-3/2}}+\|\tilde{f}\|_{L^2_{-\f12}}\big) \\
&\leq C_k\|g\|_{L^2_v}\|h\|_{D_2(k)}\|f\|_{D_2(k)}.
\end{aligned}
\end{equation}
Thus the desired estimate \eqref{Gaghf} follows from \eqref{Ga2} and \eqref{Ga1}.

This completes the proof.
\end{proof}

\medskip

Before estimating the collision terms on \(\mathbb T^3\times\mathbb R^3\), we record the following interpolation inequalities. They follow from the Cauchy--Schwarz and Young inequalities together with Lemma 4.15 in~\cite{CHJ}.
\begin{lem}\label{interpolation}
Let $a_i,k_i,b\in\R, i=1,2,3,\th\in(0,1)$ satisfy $a_1=a_2\th+a_3(1-\th)$ and $k_1=k_2\th+k_3(1-\th)$. Then, for any function $f(x,v)$, we have
\begin{equation}\label{interpolation2}
\|f\|_{H^{a_1}_xH^b_{k_1}}\ls \|f\|^\th_{H^{a_2}_xH^b_{k_2}}\|f\|^{1-\th}_{H^{a_3}_xH^b_{k_3}}.
\end{equation}
For any $a\in\R$ and radial function $\phi$, we have (see Lemma 5.8 in \cite{HL})
\[
\phi(|D|)(-\D_{\S^2})^{a}=(-\D_{\S^2})^{a}\phi(|D|),\quad \mbox{and}\quad\phi(|v|)(-\D_{\S^2})^{a}=(-\D_{\S^2})^{a}\phi(|v|),
\]
Consequently, the same interpolation estimate holds with $H^b_{k_i}$ in \eqref{interpolation2} replaced by $D_1(k_i)$ or $D_2(k_i)$, i=1,2,3.
\end{lem}

\begin{lem}\label{paal1al2}
Recall the energy and dissipation functionals defined in \eqref{functional}. For any multi-index \(\alpha\) with \(|\alpha|=2\), where every sum over \(\alpha_1,\alpha_2\) below is subject to \(\alpha_1+\alpha_2=\alpha\), we have
	\begin{equation}\label{Qal1al2}
\begin{aligned}
& \sum_{|\al_1|\geq1}\int_{\T^3}\big(Q(\pa^{\al_1}_xg,\pa^{\al_2}_xh),\pa^\al_xf\chi^2_N\<v\>^{2(k-4|\al|)}\big)_{L^2_v}dx \\
&\leq C_k\|g\|_{X_{15}}(\|\chi_Nh\|_{Y_k}+\|\tilde\chi_N h\|_{X_{k-\f32}})(\|\chi_Nf\|_{Y_k}+\|\tilde\chi_N f\|_{X_{k-\f32}}),~~k\geq 15,
\end{aligned}
\end{equation}
	and
	\begin{equation}\label{Gaal1al2}
\sum_{\al_1+\al_2=\al}\int_{\T^3}\big(\Ga(\pa^{\al_1}_xg,\pa^{\al_2}_xh),\pa^\al_xf\<v\>^{2k}\big)_{L^2_v}dx\leq C_k{\|g\|_{\cE_{0}}}\|h\|_{\cD_k}\|f\|_{\cD_k},~~{k\geq0}.
\end{equation}
\end{lem}
\begin{proof}
We begin with the proof of \eqref{Qal1al2}. When $|\al_1|=|\al_2|=1$, \eqref{cutoffgfh} gives
\[
\begin{aligned}
& \int_{\T^3}\big(Q(\pa^{\al_1}_xg,\pa^{\al_2}_xh),\pa^\al_xf \chi_N^2\<v\>^{2(k-4|\al|)}\big)_{L^2_v}dx \\
&\leq C_k \int_{\T^3}\|\pa^{\al_1}_x g\|_{L^2_7}\big(\|\chi_N\pa^{\al_2}_xh\|_{D_1(k-7)}+\|\tilde \chi_N\pa^{\al_2}_xh\|_{L^2_{k-8-\f32}}\big) \\
&\qquad\mathbin{\times} \big(\|\chi_N\pa^\al_x f\|_{D_1(k-8)}+\|\tilde\chi_N\pa^\al_x f\|_{L^2_{k-8-\f32}}\big)dx.
\end{aligned}
\]
By H\"older's inequality and the Sobolev embeddings $H^1_x(\T^3)\hookrightarrow L^6_x(\T^3),H^{1/2}_x(\T^3)\hookrightarrow L^3_x(\T^3)$, the first term on the right-hand side can be bounded by 
\[
\begin{aligned}
& C_k\|\pa^{\al_1}_x g\|_{L^6_xL^2_7}\|\chi_N\pa^{\al_2}_xh\|_{L^3_xD_1(k-7)}\|\chi_N\pa^\al_x f\|_{L^2_xD_1(k-8)} \\
&\leq C_k\|\pa^{\al_1}_x g\|_{H^1_xL^2_7}\|\chi_N\pa^{\al_2}_xh\|_{H^{1/2}_xD_1(k-7)}\|\chi_N\pa^\al_x f\|_{L^2_xD_1(k-8)} \\
&\leq C_k\|g\|_{X_{15}}\|\chi_Nh\|_{Y_k}\|\chi_Nf\|_{Y_k},
\end{aligned}
\]
where the last step invokes Lemma \ref{interpolation} to obtain
\[
\begin{aligned}
&\|\chi_N\pa^{\al_2}_xh\|_{H^{1/2}_xD_1(k-7)}\ls \|\chi_N\pa^{\al_2}_xh\|_{H^{1/2}_xD_1(k-6)} \\
&\leq C_k\|\chi_Nh\|^{1/4}_{L^2_xD_1(k)}\|\chi_Nh\|^{3/4}_{H^{2}_xD_1(k-8)}\leq C_k\|\chi_Nh\|_{Y_k}.
\end{aligned}
\]
The other terms can be handled similarly, and we obtain the bound in \eqref{Qal1al2}.

For \(|\alpha_1|=2\) we use the Sobolev embedding \(H^{3/2+\delta}_x(\mathbb T^3)\hookrightarrow L^\infty_x(\mathbb T^3)\) (\(\delta>0\)) together with the interpolation estimate
\[
\|h\|_{H^{7/4}_xD_1(k-7)}\leq C_k\|h\|^{1/8}_{L^2_xD_1(k)}\|h\|^{7/8}_{H^2_xD_1(k-8)}\leq C_k\|h\|_{Y_k}
\]
supplied by Lemma \ref{interpolation}.  These bounds yield 
\[
\begin{aligned}
& \int_{\T^3}\big(Q(\pa^{\al}_xg,h),\pa^\al_xf\chi^2_N\<v\>^{2(k-4|\al|)}\big)_{L^2_v}dx \\
&\leq C_k \int_{\T^3}\|\pa^{\al}_x g\|_{L^2_7}(\|\chi_Nh\|_{D_1(k-7)}+\|\tilde\chi_N h\|_{L^2_{k-8-\f32}}) \\
&\qquad\mathbin{\times} (\|\chi_N\pa^\al_x f\|_{D_1(k-8)}+\|\tilde\chi_N \pa^\al_x f\|_{L^2_{k-8-\f32}})dx \\
&\leq C_k\|\pa^{\al}_x g\|_{L^2_xL^2_7}(\|\chi_Nh\|_{H^{7/4}_xD_1(k-7)}+\|\tilde\chi_N h\|_{H^{7/4}_xL^2_{k-8-\f32}}) \\
&\qquad\mathbin{\times} (\|\chi_N\pa^\al_x f\|_{{L_x^2} D_1(k-8)}+\|\tilde\chi_N\pa^\al_x f\|_{L^2_xL^2_{k-8-\f32}}) \\
&\leq C_k\|g\|_{X_{15}}(\|\chi_Nh\|_{Y_k}+\|\tilde\chi_N h\|_{X_{k-\f32}})(\|\chi_Nf\|_{Y_k}+\|\tilde\chi_N f\|_{X_{k-\f32}}).
\end{aligned}
\]
This completes the proof of \eqref{Qal1al2}.

A similar argument yields \eqref{Gaal1al2}, this time invoking \eqref{Gaghf} instead of \eqref{cutoffgfh}. This completes the proof.
\end{proof}

{Next, we give two important compactness lemmas.}

\begin{lem}\label{compactSobolevembedding}
  For any \(a,b>0\), multiplication by a function in \(\mathcal{S}(\mathbb{R}^3)\) is a compact operator from \(H^a(\mathbb{R}^3)\) to \(L^2(\mathbb{R}^3)\), and the embedding \(H^a_b(\mathbb{R}^3)\hookrightarrow L^2(\mathbb{R}^3)\) is compact. Moreover, for any \(c>0\), the embedding \(H^c_xH^a_b(\mathbb{T}^3\times\mathbb{R}^3)\hookrightarrow L^2(\mathbb{T}^3\times\mathbb{R}^3)\) is compact.
\end{lem}
\begin{proof}
The first assertion is proved in Lemma 1.68 in \cite{HCD}. We prove the second assertion. 
Let $\{f_j\},j\in\N$ be a bounded sequence in $H^a_b$. By Lemma 4.15 in \cite{CHJ}, we have 
\[
\|f_{i}\|^2_{H^a_b}\sim \sum_{k=-1}^\infty \|\<v\>^b\cP_k f_i\|^2_{H^a},
\]
where the dyadic operator $\cP_k$ is defined as
\[
\cP_{-1} f=\psi f,\quad \cP_kf=\vphi(2^{-k}\cdot)f,~~k\geq0,
\]
with nonnegative functions $\psi,\vphi\in C^\infty_c(\R^3)$  satisfying $\psi+\sum_{k\geq0}\vphi(2^{-k}\cdot)\equiv1$. By the first assertion and a diagonal argument, there exists a subsequence, still denoted by $\{f_{i}\}$, such that $\<v\>^b\cP_k f_i$  converges in $L^2(\R^3)$ for any $k\geq-1$. Since
\[
\begin{aligned}
\|f_i-f_j\|^2_{L^2} &\sim \sum_{k=-1}^\infty 2^{-2bk}\|\<v\>^{b}\cP_k(f_i-f_j)\|^2_{L^2} \\
&= \sum_{k< M}2^{-2bk}\|\<v\>^{b}\cP_k(f_i-f_j)\|^2_{L^2}+\sum_{k\geq M}2^{-2bk}\|\<v\>^{b}\cP_k(f_i-f_j)\|^2_{L^2} \\
&\leq M\sup_{k< M}\|\<v\>^b\cP_k(f_i-f_j)\|^2_{L^2}+C_b 2^{-2bM}\sup_{i}\|f_i\|^2_{H^a_b}.
\end{aligned}
\]
For any $\vep>0$, we first choose $M$ sufficiently large so that the second term is less than $\vep/2$. Since $\{\<v\>^b\cP_k f_i\}$ is a Cauchy sequence for any $k\geq-1$, there exists $N_M>0$ such that for $i,j>N_M$,
\[
\sup_{k< M}\|\<v\>^b\cP_k(f_i-f_j)\|^2_{L^2}<\f \vep{2M},
\] which implies that \(\{f_i\}\) is a Cauchy sequence in \(L^2(\R^3)\).

The same argument applies on \(\mathbb{T}^3\times\mathbb{R}^3\) and completes the proof.
\end{proof}

{\begin{lem}[Aubin-Lions]\label{ALlemma}
Let $E_0\subset E\subset E_1$ be reflexive Banach spaces, with the embedding $E_0\subset E$ being compact and the inclusion $E\subset E_1$ continuous. If $0<T<\infty,1\leq p<\infty,1<r<\infty$ and the sequence of functions $\{f_j\}, j\in\N$ satisfies
\[
\|f_j\|_{L^p([0,T],E_0)}\leq C_1,\quad \|\f d{dt}f_j(t)\|_{L^r([0,T],E_1)}\leq C_2,\quad \forall j\in \N
\]
for some constants $C_1,C_2<\infty$, then $\{f_j\}$ is relatively compact in $L^p([0,T],E)$.
 \end{lem}
A proof is given in Lemma 3.7 in \cite{Tsai}. For our application in the next section, $p=2,r=2, E=H^1_xH^{\f34}_7,E_1=(H^2_xH^4_6)^*$ and  $E_0=Y_k,k\geq 17$. By Lemma \ref{compactSobolevembedding} and the definition of $Y_k$ in \eqref{functional}, the embedding $Y_k\subset H^2_xH^1_{k-8-\f32}\hookrightarrow H^1_xH^{\f34}_7=E$ is compact for $k\geq 17$.} 

\medskip

We conclude this section by evaluating a Maxwellian integral used in Section \ref{local}.

\begin{lem}\label{mumainpart}
Recall the scattering-type Maxwellian \(M_{V(t),T(t)}=\f1{(2\pi T(t))^{\f32}}\exp\{-\f{|v-V(t)|^2}{2T(t)}\}\) and  $\om(t)=V(t)/\sqrt{T(t)}$. Then
\begin{equation}\label{I2}
\begin{aligned}
& T(t)^{-1}\int_{\R^3}\f{v+\om(t)}{|v+\om(t)|^3}\mu dv=\int_{\R^3}\f{v}{|v|^3}M_{V(t),T(t)}(v)dv \\
&= \f{V(t)}{|V(t)|^3}\Big[\mathrm{erf}(\f{|V(t)|}{\sqrt{2T(t)}})-\sqrt{\f{2}\pi}\f{|V(t)|}{\sqrt{T(t)}}\exp(-\f{|V(t)|^2}{2T(t)})\Big],
\end{aligned}
\end{equation}
where the expression at $V(t)=0$ is understood by continuous extension and $\erf(x)=\frac{2}{\sqrt{\pi}}\int_{0}^{x}e^{-u^{2}}d u$. Moreover, defining 
\[
\mathfrak{G}(x):=\f{\mathrm{erf}(x)-x(\mathrm{erf})'(x)}{2x^2},
\]
we have 
\begin{align}
& 0<\mathfrak{G}(x)<\min\big\{\f4{x^2}, \f2 {x},100\big\},\quad \forall x>0\label{Gx} \\
& \mathfrak{G}(x)\gs \f1{x^2} ,\quad \forall x>10.\label{Gx2}
\end{align} 
With this notation, 
\[
T(t)^{-1}\int_{\R^3}\f{v+\om(t)}{|v+\om(t)|^3}\mu dv=\f{V(t)}{|V(t)|}\f1{T(t)}\mathfrak{G}(\f{|V(t)|}{\sqrt{2T(t)}}).
\]
\end{lem}
\begin{proof}
Choose $O$ to be an orthogonal matrix such that $O^\tau V(t)=|V(t)|(0,0,1)$. The change of variables $v\rightarrow Ov$ gives
\[
\int_{\R^3}\f{v}{|v|^3}M_{V(t),T(t)}(v)dv=\int_{\R^3}\f{Ov}{|v|^3}\f1{(2\pi T(t))^{\f32}}e^{-\f{|v-|V(t)|(0,0,1)|^2}{2T(t)}}dv.
\]
Passing to spherical coordinates $v=(r\sin\theta\sin\phi,r\sin\theta\cos\phi,r\cos\theta)$ gives
\[
I:=\f{V(t)}{|V(t)|}\int_0^\infty \int_0^\pi\int_0^{2\pi}\f1{(2\pi T(t))^{\f32}}\exp(-\f{r^2+U^2-2rU\cos\theta }{2T(t)})(\f{r\cos\theta}{r^3})r^2\sin\theta d\phi d\theta dr,
\]
where we denote $|V(t)|$ by $U$. Integration in the azimuthal angle gives
\[
I=\f{V(t)}{|V(t)|}\f{2\pi}{(2\pi T)^{\f32}}\int_0^\infty \exp(-\f{r^2+U^2}{2T})\int_0^\pi \exp(\f{rU\cos\theta } T)\cos\theta\sin\theta d\theta dr.
\]
We evaluate the inner $\theta$ integral using the substitution $u=\cos\theta$ (where $du=-\sin\theta d\theta $):
\[
\int_{-1}^1 \exp(\f{rU}T u)udu=\f{T}{rU}(e^{\f{rU}T}+e^{-\f{rU}T})-(\f T{rU})^2(e^{\f{rU}T}-e^{-\f{rU}T}).
\]
Substituting this expression into the preceding integral gives 
\begin{equation}\label{I}
I=\f{V(t)}{|V(t)|}\f1{\sqrt{2\pi T(t)}U^2}\int_0^\infty \f1 {r^2}\Big(\exp(-\f{(r-U)^2}{2T})(rU- T)+\exp(-\f{(r+U)^2}{2T})(rU+T)\Big)dr.
\end{equation}
We differentiate $\f Tr e^{-\f{(r\pm U)^2}{2T}}$ with respect to $r$:
\[
\f d{dr}(\f Tr e^{-\f{(r-U)^2}{2T}})=-\f T{r^2}e^{-\f{(r-U)^2}{2T}}+\f Tr (-\f{r-U}T)e^{-\f{(r-U)^2}{2T}} = e^{-\f{(r-U)^2}{2T}}(\f Ur-\f T{r^2})-e^{-\f{(r-U)^2}{2T}}.
\]
Rearranging this gives
\[
e^{-\f{(r-U)^2}{2T}}(\f Ur-\f T{r^2})=\f d{dr}(\f Tr e^{-\f{(r-U)^2}{2T}})+e^{-\f{(r-U)^2}{2T}}.
\]
The corresponding identity for the other exponential is 
\[
e^{-\f{(r+U)^2}{2T}}(\f Ur+\f T{r^2})=\f d{dr}(-\f Tr e^{-\f{(r+U)^2}{2T}})-e^{-\f{(r+U)^2}{2T}}.
\]
Substituting these identities into \eqref{I} gives 
\[
\begin{aligned}
&I=\f{V(t)}{|V(t)|}\f1{\sqrt{2\pi T(t)}U^2}\Big(\int_0^\infty \f d{dr}(\f Tr e^{-\f{(r-U)^2}{2T}}-\f Tr e^{-\f{(r+U)^2}{2T}})dr \\
&\quad{} +\int_0^\infty(e^{-\f{(r-U)^2}{2T}}-e^{-\f{(r+U)^2}{2T}})dr\Big).
\end{aligned}
\]
The boundary values are
\[
\begin{aligned}
& (\f Tr e^{-\f{(r-U)^2}{2T}}-\f Tr e^{-\f{(r+U)^2}{2T}})=0\quad r\rightarrow \infty; \\
& (\f Tr e^{-\f{(r-U)^2}{2T}}-\f Tr e^{-\f{(r+U)^2}{2T}})=2U e^{-\f{U^2}{2T}}\quad r\rightarrow 0.
\end{aligned}
\]
The remaining terms are standard Gaussian tails. Substituting $x=\f{r-U}{\sqrt{2T}}$ and $y=\f{r+U}{\sqrt{2T}}$, we obtain
\[
\int_0^\infty(e^{-\f{(r-U)^2}{2T}}-e^{-\f{(r+U)^2}{2T}})dr=\sqrt{2\pi T} \cdot\mathrm{erf}(\f U{\sqrt{2T}}).
\]
Substituting the boundary values and the Gaussian integral proves \eqref{I2}. The properties \eqref{Gx} and \eqref{Gx2} follow directly from the explicit definition.
\end{proof}

\section{Local existence and non-negativity}\label{local}
\begingroup
\makeatletter
\@ifpackageloaded{showkeys}{%
  \let\SK@\@gobbletwo
  \renewcommand*\showkeyslabelformat[1]{}%
}{}
\makeatother

In this section, we prove local existence for \eqref{equg} together
with the equations for $V$ and $T$. We first remove the singularity
of the spherical term at the origin, using the cutoff function from
Definition~\ref{cutofffunction}. For each fixed cutoff, we construct
the solution by successive approximations and prove convergence by
estimating the difference of two consecutive approximations.
We then derive estimates independent of the cutoff. These estimates
allow us to extend all the approximate solutions to the same interval
$[0,t_0]$ and to pass to the limit in the equation for $F$.

\subsection{The equation with cutoff}

Let $0<\nu\leq1$. With the notation of
Definition~\ref{cutofffunction}, we set
\begin{equation}\label{nu:kernels}
a_\nu(v)=\frac{\Pi(v)}{|v|}\chi^c_{\nu^2}(v).
\end{equation}
Thus $a_\nu$ vanishes for $|v|\leq\nu$ and agrees with $a$ for
$|v|\geq\sqrt2\nu$. Since the cutoff is radial and $\Pi(v)v=0$,
we have
\begin{equation}\label{nu:structure}
0\leq a_\nu\leq a,\qquad a_\nu(v)v=0,\qquad \nabla_v\cdot a_\nu(v)=-2\frac{v}{|v|^3}\chi^c_{\nu^2}(v).
\end{equation}

We consider the regularized equation
\begin{equation}\label{nu:physical}
\partial_tF_\nu+v\cdot\nabla_xF_\nu+E\cdot\nabla_vF_\nu =Q(F_\nu,F_\nu)+\nabla_v\cdot(a_\nu\nabla_vF_\nu).
\end{equation}
The second identity in \eqref{nu:structure} shows that the spherical
diffusion still preserves kinetic energy. Hence the computations
leading to \eqref{dtV} and \eqref{dtT} remain valid, with
$R_\nu=\iint v|v|^{-3}\chi^c_{\nu^2}(v)F_\nu\,dx\,dv$.

We apply the change of variables \eqref{changeofvariable} and write
$G_\nu=\mu+g_\nu$. The coefficient of the spherical diffusion in the
new variables is
\[
\frac1{T_\nu}a_\nu(V_\nu+\sqrt{T_\nu}v) =T_\nu^{-3/2} \frac{\Pi(v+\om_\nu)}{|v+\om_\nu|} \chi^c_{\nu^2/T_\nu}(v+\om_\nu), \qquad \om_\nu=\frac{V_\nu}{\sqrt{T_\nu}}.
\]
The factor $\nu^2/T_\nu$ follows directly from the definition of
$\chi_N$. We retain this dependence on $T_\nu$ in the approximation.
The system to be solved is therefore
\begin{equation}\label{nu:g-system}
\left\{\begin{aligned}
\partial_tg_\nu+\sqrt{T_\nu}v\cdot\nabla_xg_\nu &= T_\nu^{-3/2}\bigl(Lg_\nu+Q(g_\nu,g_\nu)\bigr) \\
&\quad{} + \frac{T_\nu'}{2T_\nu}\nabla_v\cdot\bigl(v(\mu+g_\nu)\bigr) -\frac2{\sqrt{T_\nu}}R_\nu\cdot\nabla_v(\mu+g_\nu) \\
&\quad{} + \frac1{T_\nu}\nabla_v\cdot \bigl(a_\nu(V_\nu+\sqrt{T_\nu}v)\nabla_v(\mu+g_\nu)\bigr), \\
V_\nu' &= E-2R_\nu,\qquad T_\nu'=\frac43V_\nu\cdot R_\nu .
\end{aligned}\right.
\end{equation}
Here $R_\nu$ is given by the same integral as in \eqref{eqG},
with the cutoff inserted.

In the iteration below we also need this integral for a prescribed
function $h$. For its corresponding functions $V_h,T_h$, we write
\begin{equation}\label{nu:normalized-coefficients}
R_h(t)=T_h^{-1}\iint \frac{v+V_h/\sqrt{T_h}}{|v+V_h/\sqrt{T_h}|^3} \chi^c_{\nu^2/T_h}(v+V_h/\sqrt{T_h})(\mu+h)\,dx\,dv.
\end{equation}
In particular, $h=g_\nu$ gives $R_h=R_\nu$.
The initial conditions are those of \eqref{equg}:
$g_\nu(0)=g_0$, $V_\nu(0)=V_0$, and $T_\nu(0)=T_0>0$.
We assume $k\geq17$, $\mu+g_0\geq0$, and \eqref{conforg} at $t=0$.
The norm $\|g_0\|_{X_k}^2$ is exactly the perturbation quantity
appearing in \eqref{small}.

We first record the estimates for $R_h$ which will be used both
in the linear problem and in the continuation argument.
For any $T>0$ and $\om\in\mathbb R^3$, the near and far parts of the
integral give
\begin{equation}\label{splitregime}
\begin{aligned}
&\left|\iint\frac{v+\om}{|v+\om|^3} \chi^c_{\nu^2/T}(v+\om)h\,dx\,dv\right| \\
&\quad\leq C_m\min\left\{ \|h\|_{L_x^2H_3^{3/4}},\ |\om|^{-2}\|h\|_{L_x^2L_3^2} +|\om|^{-m}\|h\|_{L_x^2H_{m+1}^{3/4}}\right\}, \qquad m>0.
\end{aligned}
\end{equation}
The first bound is used when $\om=0$.
To obtain it, split at $|v+\om|=1$. On the first region,
H\"older's inequality and $H^{3/4}\hookrightarrow L^4$ apply
because $|v+\om|^{-2}$ belongs to $L^{4/3}$ locally.
On the second region, the weighted $L^2$ norm controls the
$L^1$ norm of $h$. For large $|\om|$, a further split into
$|v|\leq|\om|/2$ and $|v|>|\om|/2$ gives the second bound:
the first region yields $|\om|^{-2}$, while the velocity weights
of $h$ give the stated decay on the second region.
Only $0\leq\chi^c_{\nu^2/T}\leq1$ is used, so the constants
are independent of $\nu,T,\om$.

Applying \eqref{splitregime} to $\mu+h$ gives
\begin{align}
|R_h|&\leq CT_h^{-1}
\bigl(1+\|h\|_{L_x^2H_3^{3/4}}\bigr),\label{Rh} \\
|V_h\cdot R_h|&\leq CT_h^{-1/2}
\bigl(1+\|h\|_{L_x^2H_3^{3/4}}\bigr).\label{nu:VR}
\end{align}
For the second estimate, use $m=2$ and
$|\om|\min\{1,|\om|^{-2}\}\leq1$.
Consequently, the equations for $V_h,T_h$ imply
\begin{equation}\label{Tprime}
|T_h'|\leq CT_h^{-1/2} \bigl(1+\|h\|_{L_x^2H_3^{3/4}}\bigr),\qquad \frac{|T_h'|}{T_h}+\frac{|R_h|}{\sqrt{T_h}} \leq CT_h^{-3/2} \bigl(1+\|h\|_{L_x^2H_3^{3/4}}\bigr).
\end{equation}
These bounds are independent of $E$ and $V_h$. 

\subsection{The linear problem}

We now fix $\nu>0$. For a given $h$, the equations for $V_h,T_h$
are uniquely solvable by Lemma~\ref{lem:local-wellposed}.
We first solve the linear equation with these coefficients.

\begin{lem}\label{locallinear}
Let $k\geq17$ and let
$h\in L^\infty([0,s],X_k)\cap L^2([0,s],Y_k)$ satisfy
$\mu+h\geq0$. There is $\varepsilon_0>0$, depending only on $k$,
such that the following assertion holds when
$\sup_{[0,s]}\|h\|_{X_k}\leq\varepsilon_0$.
Suppose the corresponding functions $V_h,T_h$ satisfy
$0<a\leq T_h\leq b$ and $|V_h|\leq M$ on $[0,s]$.
Then the equation
\begin{equation}\label{nu:linear-g}
\begin{aligned}
\partial_tq+\sqrt{T_h}v\cdot\nabla_xq &= T_h^{-3/2}\bigl(Q(\mu+h,q)+Q(h,\mu)\bigr) \\
&\quad{} + \frac{T_h'}{2T_h}\nabla_v\cdot\bigl(v(\mu+q)\bigr) -\frac2{\sqrt{T_h}}R_h\cdot\nabla_v(\mu+q) \\
&\quad{} + \frac1{T_h}\nabla_v\cdot \bigl(a_\nu(V_h+\sqrt{T_h}v)\nabla_v(\mu+q)\bigr), \qquad q(0)=g_0,
\end{aligned}
\end{equation}
has a unique solution in $C([0,s],X_k)\cap L^2([0,s],Y_k)$.
Moreover, $\mu+q\geq0$ if $\mu+g_0\geq0$.
As long as $\|q\|_{X_k}\leq1$, for every sufficiently small
$\delta>0$ we have
\begin{equation}\label{nu:linear-energy}
\frac d{dt}\|q\|_{X_k}^2+\lambda_0T_h^{-3/2}\|q\|_{Y_k}^2 \leq T_h^{-3/2} \left\{\delta\|h\|_{Y_k}^2+ C_{k,\delta}\bigl(1+\|h\|_{X_k}^2+\|q\|_{X_k}^2\bigr)\right\}.
\end{equation}
The constants in this inequality are independent of $\nu,E,V_h$.
\end{lem}

\begin{proof}
We begin with the energy estimate. For $|\alpha|=0,2$, apply
$\partial_x^\alpha$ to \eqref{nu:linear-g}, multiply by
$\partial_x^\alpha q\langle v\rangle^{2(k-4|\alpha|)}$, and
integrate over $x,v$. The transport term vanishes by integration
over the torus.

We estimate the three collision terms separately. Throughout this
part, the cutoff estimates are used with $N\to\infty$, as explained
at the end of Lemma~\ref{cutoffQcom}; thus the terms containing
$\tilde\chi_N$ vanish. Since $\mu$ is independent of $x$,
\eqref{cutoff1} applied to $\partial_x^\alpha q$ with weight
$k-4|\alpha|\geq k-8\geq9$ gives
\[
\sum_{|\alpha|=0,2}\int_{\mathbb T^3} \bigl(Q(\mu,\partial_x^\alpha q), \partial_x^\alpha q\langle v\rangle^{2(k-4|\alpha|)} \bigr)_{L_v^2}\,dx \leq-\lambda_0\|q\|_{Y_k}^2+C_k\|q\|_{X_k}^2.
\]
Here we have discarded the additional nonpositive term in
\eqref{cutoff1}, and the unweighted remainder is bounded by
$\|q\|_{X_k}^2$.

For $Q(h,q)$, the Leibniz rule gives
\[
\partial_x^\alpha Q(h,q) =Q(h,\partial_x^\alpha q) +\sum_{0<\beta\leq\alpha}\binom{\alpha}{\beta} Q(\partial_x^\beta h,\partial_x^{\alpha-\beta}q),
\]
where the sum is empty when $|\alpha|=0$. The first estimate in
\eqref{cutoffgfh}, with $g=h$ and
$f=\partial_x^\alpha q$, bounds the term with no derivative on $h$:
\[
\sum_{|\alpha|=0,2}\left|\int_{\mathbb T^3} \bigl(Q(h,\partial_x^\alpha q), \partial_x^\alpha q\langle v\rangle^{2(k-4|\alpha|)} \bigr)_{L_v^2}\,dx\right| \leq C_k\|h\|_{L_x^\infty L_5^2}\|q\|_{Y_k}^2 \leq C_k\|h\|_{X_{15}}\|q\|_{Y_k}^2.
\]
The last inequality follows from
$H_x^2(\mathbb T^3)\hookrightarrow L_x^\infty(\mathbb T^3)$
and $\|h\|_{H_x^2L_5^2}\leq C\|h\|_{X_{15}}$.
For $|\alpha|=2$ and $|\beta|=1,2$, we use the termwise bounds
in the proof of \eqref{Qal1al2} in Lemma~\ref{paal1al2},
with $(g,h,f)=(h,q,q)$. Absorbing the finitely many binomial
coefficients into $C_k$, we obtain
\[
\sum_{|\alpha|=2}\sum_{0<\beta\leq\alpha} \binom{\alpha}{\beta}\left|\int_{\mathbb T^3} \bigl(Q(\partial_x^\beta h,\partial_x^{\alpha-\beta}q), \partial_x^\alpha q\langle v\rangle^{2(k-8)} \bigr)_{L_v^2}\,dx\right| \leq C_k\|h\|_{X_{15}}\|q\|_{Y_k}^2.
\]
Since $k\geq17$, $\|h\|_{X_{15}}\leq\|h\|_{X_k}$.
This accounts for all terms in $\partial_x^\alpha Q(h,q)$.

Finally, $\partial_x^\alpha Q(h,\mu)=Q(\partial_x^\alpha h,\mu)$.
Applying \eqref{fmuf} in Lemma~\ref{ghf2} at weight
$k-4|\alpha|$, followed by the Cauchy--Schwarz inequality in $x$,
gives
\[
\sum_{|\alpha|=0,2}\left|\int_{\mathbb T^3} \bigl(Q(\partial_x^\alpha h,\mu), \partial_x^\alpha q\langle v\rangle^{2(k-4|\alpha|)} \bigr)_{L_v^2}\,dx\right| \leq C_k\sum_{|\alpha|=0,2} \|\partial_x^\alpha h\|_{L_x^2L_3^2} \|\partial_x^\alpha q\|_{L_{x,v}^2} \leq C_k\|h\|_{X_k}\|q\|_{X_k}.
\]
Combining these estimates and using Young's inequality, we find
\[
\bigl(Q(\mu+h,q)+Q(h,\mu),q\bigr)_{X_k} \leq -\bigl(\lambda_0-C_k\|h\|_{X_k}\bigr)\|q\|_{Y_k}^2 +C_k\bigl(\|h\|_{X_k}^2+\|q\|_{X_k}^2\bigr).
\]
We choose $\varepsilon_0$ so that $C_k\varepsilon_0\leq\lambda_0/4$.
After multiplication by $T_h^{-3/2}$, this leaves sufficient
coercivity to absorb the small dissipation term from the estimate for the spherical term below. All constants used here are independent of
$\nu,E,V_h$.

We next treat the spherical term. In the following calculation we write
$V=V_h$, $T=T_h$, and $\om=V/\sqrt T$. Since
$a_\nu(V+\sqrt T v)$ is independent of $x$, the same argument
applies to each spatial derivative. We claim that, for any $f$
and any $\delta>0$,
\begin{equation}\label{nu:ion-weight}
\begin{aligned}
&\left(\frac1T\nabla_v\cdot \bigl(a_\nu(V+\sqrt T v)\nabla_v f\bigr),f\right)_{X_k} \\
&\quad \leq-\frac1{2T}\sum_{|\alpha|=0,2}\iint \langle v\rangle^{2(k-4|\alpha|)} a_\nu(V+\sqrt T v)\nabla_v\partial_x^\alpha f \cdot\nabla_v\partial_x^\alpha f +T^{-3/2} \left(\delta\|f\|_{Y_k}^2+C_{k,\delta}\|f\|_{X_k}^2\right).
\end{aligned}
\end{equation}
To prove this, fix $|\alpha|=0,2$ and write $l=k-4|\alpha|$.
Integration by parts gives the exact identity
\[
\begin{aligned}
&\left(\frac1T\nabla_v\cdot \bigl(a_\nu(V+\sqrt T v)\nabla_v\partial_x^\alpha f\bigr), \partial_x^\alpha f\langle v\rangle^{2l}\right)_{L_{x,v}^2} \\
&\quad=-\frac1T\iint\langle v\rangle^{2l} a_\nu(V+\sqrt T v)\nabla_v\partial_x^\alpha f \cdot\nabla_v\partial_x^\alpha f \\
&\qquad\quad-\frac{2l}{T}\iint \partial_x^\alpha f\langle v\rangle^{2l-2} a_\nu(V+\sqrt T v)\nabla_v\partial_x^\alpha f\cdot v.
\end{aligned}
\]
The matrix is symmetric and nonnegative. By \eqref{nu:structure},
\[
0\leq a_\nu(V+\sqrt T v) \leq T^{-1/2}|v+\om|^{-1}\mathbf I .
\]
We therefore apply the Cauchy--Schwarz inequality associated with
this matrix to the second integral. Its absolute value is bounded by
\[
\begin{aligned}
&\frac1{2T}\iint\langle v\rangle^{2l} a_\nu(V+\sqrt T v)\nabla_v\partial_x^\alpha f \cdot\nabla_v\partial_x^\alpha f +\frac{2l^2}{T}\iint |\partial_x^\alpha f|^2\langle v\rangle^{2l-4} a_\nu(V+\sqrt T v)v\cdot v \\
&\quad \leq\frac1{2T}\iint\langle v\rangle^{2l} a_\nu(V+\sqrt T v)\nabla_v\partial_x^\alpha f \cdot\nabla_v\partial_x^\alpha f +C_lT^{-3/2}\iint \frac{|\partial_x^\alpha f|^2\langle v\rangle^{2l-2}}{|v+\om|}.
\end{aligned}
\]
To bound the last integral uniformly in $\om$, we use the 
Hardy inequality
\[
\int_{\mathbb R^3}\frac{|u(v)|^2}{|v+\om|^2}\,dv \leq4\int_{\mathbb R^3}|\nabla_vu(v)|^2\,dv, \qquad \om\in\mathbb R^3.
\]
For smooth compactly supported $u$, this follows by integrating
$\nabla_v\cdot((v+\om)/|v+\om|^2)=|v+\om|^{-2}$ against $|u|^2$
and applying Cauchy's inequality; the general case follows by
density. In particular, the constant does not depend on $\om$.
Apply this inequality to
$u=\langle v\rangle^{l-3/2}\partial_x^\alpha f$ for each $x$.
The Cauchy--Schwarz inequality in $v$ and then in $x$ gives
\[
\begin{aligned}
\iint\frac{|\partial_x^\alpha f|^2 \langle v\rangle^{2l-2}}{|v+\om|} &\leq \left\|\frac{\langle v\rangle^{l-3/2} \partial_x^\alpha f}{|v+\om|}\right\|_{L_{x,v}^2} \|\partial_x^\alpha f\|_{L_x^2L_{l-1/2}^2} \\
&\quad \leq 2\|\nabla_v(\langle v\rangle^{l-3/2} \partial_x^\alpha f)\|_{L_{x,v}^2} \|\partial_x^\alpha f\|_{L_x^2L_{l-1/2}^2} \\
&\quad \leq C_l\|\partial_x^\alpha f\|_{L_x^2D_1(l)} \|\partial_x^\alpha f\|_{L_x^2L_l^2} \leq\delta\|\partial_x^\alpha f\|_{L_x^2D_1(l)}^2 +C_{l,\delta}\|\partial_x^\alpha f\|_{L_x^2L_l^2}^2.
\end{aligned}
\]
Here we used the definition \eqref{dissipationnorm}, the bound
$\langle v\rangle^{l-1/2}\leq\langle v\rangle^l$, and Young's
inequality. Summing over $|\alpha|=0,2$ and adjusting $\delta$
proves \eqref{nu:ion-weight}. The factor $T^{-3/2}$ comes from
$T^{-1}$ in the equation and $T^{-1/2}$ in the matrix bound.
No derivative of the cutoff occurs, and all constants in the
remainder are independent of $\nu,V,T$.

We estimate the source involving $\mu$ before expanding the
divergence. Since $\mu$ is independent of $x$, only
$|\alpha|=0$ contributes. We obtain
\begin{equation}\label{nu:ion-Maxwellian}
\begin{aligned}
&\left|\left(\frac1T\nabla_v\cdot \bigl(a_\nu(V+\sqrt T v)\nabla_v\mu\bigr),f\right)_{X_k}\right| \\
&\quad\leq\frac1{4T}\iint \langle v\rangle^{2k}a_\nu(V+\sqrt T v) \nabla_v f\cdot\nabla_v f +C_kT^{-3/2}\bigl(1+\|f\|_{X_k}^2\bigr).
\end{aligned}
\end{equation}
Here the term containing $\nabla_v f$ is absorbed in the
nonnegative integral of \eqref{nu:ion-weight}. For the term
containing the derivative of the weight, we use the Gaussian
decay of $\mu$ and split at $|v+V/\sqrt T|=1$.
The local integrability of $|v+V/\sqrt T|^{-2}$ gives a
constant independent of $V,T,\nu$ after the factor $T^{-3/2}$
has been extracted.

It remains to estimate the two terms containing $T_h'$ and
$R_h$. We first consider the parts acting on $q$. Since these
coefficients depend only on time, they commute with
$\partial_x^\alpha$. For $|\alpha|=0,2$, writing again
$l=k-4|\alpha|$, integration by parts gives
\[
\begin{aligned}
\left(\frac{T_h'}{2T_h}\nabla_v\cdot (v\partial_x^\alpha q), \partial_x^\alpha q\langle v\rangle^{2l}\right)_{L_{x,v}^2} &=\frac{T_h'}{2T_h}\iint \left(\frac32-l\frac{|v|^2}{\langle v\rangle^2}\right) |\partial_x^\alpha q|^2\langle v\rangle^{2l}\,dx\,dv, \\
\left(-\frac2{\sqrt{T_h}}R_h\cdot\nabla_v\partial_x^\alpha q, \partial_x^\alpha q\langle v\rangle^{2l}\right)_{L_{x,v}^2} &=\frac{2l}{\sqrt{T_h}}\iint (R_h\cdot v)|\partial_x^\alpha q|^2 \langle v\rangle^{2l-2}\,dx\,dv.
\end{aligned}
\]
Here we used
$\nabla_v\cdot(v\langle v\rangle^{2l})
=3\langle v\rangle^{2l}+2l|v|^2\langle v\rangle^{2l-2}$.
Since $|v|^2/\langle v\rangle^2\leq1$ and
$|v|/\langle v\rangle^2\leq1$, summing the absolute values
over $|\alpha|=0,2$ yields
\[
C_k\left(\frac{|T_h'|}{T_h}+\frac{|R_h|}{\sqrt{T_h}}\right) \|q\|_{X_k}^2.
\]
In particular, these terms require no velocity derivative of $q$.

For the parts acting on $\mu$, only $|\alpha|=0$ contributes.
Using $\nabla_v\cdot(v\mu)=(3-|v|^2)\mu$ and
$\nabla_v\mu=-v\mu$, followed by the Cauchy--Schwarz inequality,
we obtain
\[
\left|\left( \frac{T_h'}{2T_h}(3-|v|^2)\mu +\frac2{\sqrt{T_h}}(R_h\cdot v)\mu,\, q\langle v\rangle^{2k}\right)_{L_{x,v}^2}\right| \leq C_k \left(\frac{|T_h'|}{T_h}+\frac{|R_h|}{\sqrt{T_h}}\right) \|q\|_{L_x^2L_k^2}.
\]
The constant is finite because every polynomially weighted
$L_v^2$ norm of $\mu$, $v\mu$, and $|v|^2\mu$ is finite.
Thus the total contribution of these two terms is bounded by
\[
C_k\left(\frac{|T_h'|}{T_h}+\frac{|R_h|}{\sqrt{T_h}}\right) \bigl(\|q\|_{X_k}^2+\|q\|_{X_k}\bigr).
\]
By \eqref{Tprime}, the coefficient is at most
$CT_h^{-3/2}(1+\|h\|_{L_x^2H_3^{3/4}})$. The definitions
\eqref{dissipationnorm} and \eqref{functional}, with $k\geq17$,
give
\[
\|h\|_{L_x^2H_3^{3/4}} \leq C_k\|h\|_{L_x^2H_{k-3/2}^1} \leq C_k\|h\|_{Y_k}.
\]
For $\|q\|_{X_k}\leq1$, we have
$(\|q\|_{X_k}^2+\|q\|_{X_k})^2\leq4\|q\|_{X_k}^2$.
Young's inequality therefore yields, for every $\delta>0$,
\[
C_k(1+\|h\|_{Y_k}) (\|q\|_{X_k}^2+\|q\|_{X_k}) \leq\delta\|h\|_{Y_k}^2 +C_{k,\delta}(1+\|q\|_{X_k}^2).
\]
After multiplication by $T_h^{-3/2}$, this is the required
bound for the two terms. Its constants are independent of
$\nu,E,V_h$.

Finally, we take the small parameter in \eqref{nu:ion-weight}
to be at most $\lambda_0/4$ and use the last estimate with
$\delta/2$ in place of $\delta$. Combining all the preceding
estimates, using the choice $C_k\varepsilon_0\leq\lambda_0/4$
made above, and multiplying the energy identity by two gives
\eqref{nu:linear-energy}.

To construct the linear solution, smooth the nonnegative
function $\mu+h$ and the initial distribution $\mu+g_0$ by
the heat semigroup, and add $\eta(\Delta_x+\Delta_v)$ to the
equation for $\mu+q$. For $\eta>0$ this gives a uniformly
parabolic linear equation. The estimates just proved hold
uniformly as $\eta\downarrow0$; the extra terms produced by
smoothing $\mu$ tend to zero in the weighted spaces.
The equation also bounds the time derivative in a negative
Sobolev space. We may consequently pass to the limit by
Lemma~\ref{ALlemma}. 

The regularized distributions remain nonnegative by the
maximum principle, and their limit satisfies $\mu+q\geq0$.
Applying the energy estimate to the difference of two linear
solutions proves uniqueness. The same approximation, together
with the energy identity and the initial trace, gives
$q\in C([0,s],X_k)$. 
\end{proof}

\subsection{Convergence of the successive approximations}

We now construct a solution of \eqref{nu:g-system} for fixed $\nu$.
Let $g^0=0$. If $g^n$ is given, solve the equations for
$V^n,T^n$ with $h=g^n$ and initial data $V_0,T_0$, and denote
the corresponding integral in \eqref{nu:normalized-coefficients}
by $R^n$. We then solve \eqref{nu:linear-g} with $h=g^n$ and
the same initial value $g_0$ to obtain $g^{n+1}$.
The preceding lemma shows inductively that $\mu+g^n\geq0$.

We first choose a short interval on which the estimates for
these approximations are uniform in $n$. More precisely,
integrating \eqref{nu:linear-energy}, choosing $\delta$ small,
and then choosing $\|g_0\|_{X_k}$ and the time $s$ small, gives
\begin{equation}\label{nu:iteration-bound}
\sup_{[0,s]}\|g^n\|_{X_k}^2+ \lambda_0\int_0^s(T^n)^{-3/2}\|g^n\|_{Y_k}^2dt \leq\varepsilon_0^2,\qquad \frac{T_0}{2}\leq T^n\leq\frac{3T_0}{2}.
\end{equation}
To justify the induction, suppose that \eqref{nu:iteration-bound}
holds at step $n$ and take $\varepsilon_0\leq1$.
As long as $\|g^{n+1}\|_{X_k}\leq1$, integration of
\eqref{nu:linear-energy} with $h=g^n$ gives, for $0\leq t\leq s$,
\[
\begin{aligned}
\sup_{0\leq r\leq t}\|g^{n+1}(r)\|_{X_k}^2 +\lambda_0\int_0^t(T^n)^{-3/2}\|g^{n+1}\|_{Y_k}^2\,dr &\leq2\|g_0\|_{X_k}^2 +2\delta\int_0^t(T^n)^{-3/2}\|g^n\|_{Y_k}^2\,dr +C_{k,\delta}T_0^{-3/2}t \\
&\quad \leq2\|g_0\|_{X_k}^2 +\frac{2\delta}{\lambda_0}\varepsilon_0^2 +C_{k,\delta}T_0^{-3/2}s \leq\frac{\varepsilon_0^2}{2\cdot3^{3/2}}.
\end{aligned}
\]
The factor $2$ allows us to bound both the time supremum
and the full dissipation integral. The last inequality follows
by first choosing $\delta$, and then $\|g_0\|_{X_k}$ and $s$,
sufficiently small. It strictly improves the temporary bound
$\|g^{n+1}\|_{X_k}\leq1$, so it holds throughout $[0,s]$.

To control the next temperature, we use
$\|g^{n+1}\|_{L_x^2H_3^{3/4}}\leq C_k\|g^{n+1}\|_{Y_k}$
in \eqref{Tprime}. The preceding dissipation bound and the
Cauchy--Schwarz inequality yield
\[
\begin{aligned}
\left|(T^{n+1}(t))^{3/2}-T_0^{3/2}\right| &\leq C_k\int_0^t\bigl(1+\|g^{n+1}(r)\|_{Y_k}\bigr)\,dr \\
&\quad \leq C_k t+C_k t^{1/2}\left(\frac{3T_0}{2}\right)^{3/4} \left(\int_0^t(T^n)^{-3/2}\|g^{n+1}\|_{Y_k}^2\,dr\right)^{1/2} \\
&\quad \leq C_k\left(s+s^{1/2}T_0^{3/4}\varepsilon_0\right) \leq\frac14T_0^{3/2},
\end{aligned}
\]
after decreasing $s$ if necessary. Thus $T^{n+1}$ remains
strictly between $T_0/2$ and $3T_0/2$. In particular,
\[
(T^{n+1})^{-3/2}\leq3^{3/2}(T^n)^{-3/2},
\]
and hence
\[
\sup_{[0,s]}\|g^{n+1}\|_{X_k}^2
 +\lambda_0\int_0^s(T^{n+1})^{-3/2}\|g^{n+1}\|_{Y_k}^2\,dt
 \leq\frac12\varepsilon_0^2.
\]
This closes the energy induction with the temperature weight
appearing in \eqref{nu:iteration-bound}. Finally,
$(V^{n+1})'=E-2R^{n+1}$ and \eqref{Rh} give
\[
|V^{n+1}(t)-V_0-Et| \leq C_kT_0^{-1}\int_0^t \bigl(1+\|g^{n+1}(r)\|_{Y_k}\bigr)\,dr \leq C_kT_0^{-1} \left(t+t^{1/2}T_0^{3/4}\varepsilon_0\right).
\]
The same temperature and velocity estimates, with $g^0=0$,
start the induction. All constants and smallness choices
in this step are independent of $\nu$.

Put $G^n=\mu+g^n$. The equation defining the next approximation is
\begin{equation}\label{linear2}
\begin{aligned}
\partial_tG^{n+1}+\sqrt{T^n}v\cdot\nabla_xG^{n+1} &= (T^n)^{-3/2}Q(G^n,G^{n+1}) + \frac{(T^n)'}{2T^n}\nabla_v\cdot(vG^{n+1}) -\frac{2R^n}{\sqrt{T^n}}\cdot\nabla_vG^{n+1} \\
&\quad{} + \frac1{T^n}\nabla_v\cdot \bigl(a_\nu(V^n+\sqrt{T^n}v)\nabla_vG^{n+1}\bigr).
\end{aligned}
\end{equation}
To compare two consecutive approximations, write
$U^n=g^n-g^{n-1}=G^n-G^{n-1}$ for $n\geq1$.
All $U^{n+1}$ with $n\geq1$ have zero initial value.
In the following difference estimates $C_\nu$ may depend on
$\nu,k,T_0$ and the bounds for $V^n,T^n$. This dependence is
allowed, since $\nu$ is fixed in the present construction.

We first compare the equations for $V^n,T^n$.
Lemma~\ref{lem:local-wellposed}, with the common initial data,
gives
\begin{equation}\label{nu:ODE-difference}
\sup_{0\leq r\leq t} \bigl(|V^n-V^{n-1}|^2+|T^n-T^{n-1}|^2\bigr)(r) \leq C_\nu t^2\sup_{0\leq r\leq t} \|U^n(r)\|_{L_x^2L_7^2}^2 .
\end{equation}
The same Lipschitz estimate, followed by the equations for
$T^n$ and $T^{n-1}$, implies
\[
\left|\frac{(T^n)'}{T^n}-\frac{(T^{n-1})'}{T^{n-1}}\right| +\left|\frac{R^n}{\sqrt{T^n}}- \frac{R^{n-1}}{\sqrt{T^{n-1}}}\right| \leq C_\nu\left( \|U^n\|_{L_x^2L_7^2}+|V^n-V^{n-1}|+|T^n-T^{n-1}|\right).
\]
Thus the differences of the coefficients are controlled by
the difference of the preceding approximations.

We next subtract the two equations \eqref{linear2}. The
equation for $U^{n+1}$ has the same linear left-hand side and
the same coefficients as the equation with input $G^n$.
In addition, its right-hand side contains the following terms:
\[
\begin{aligned}
&(T^n)^{-3/2}Q(U^n,G^n),\qquad \bigl((T^n)^{-3/2}-(T^{n-1})^{-3/2}\bigr)Q(G^{n-1},G^n), \\
&\frac12\left(\frac{(T^n)'}{T^n}- \frac{(T^{n-1})'}{T^{n-1}}\right)\nabla_v\cdot(vG^n) -2\left(\frac{R^n}{\sqrt{T^n}}- \frac{R^{n-1}}{\sqrt{T^{n-1}}}\right)\cdot\nabla_vG^n, \\
&\nabla_v\cdot\left[ \left(\frac{a_\nu(V^n+\sqrt{T^n}v)}{T^n} -\frac{a_\nu(V^{n-1}+\sqrt{T^{n-1}}v)}{T^{n-1}}\right) \nabla_vG^n\right],\qquad - (\sqrt{T^n}-\sqrt{T^{n-1}})\,v\cdot\nabla_xG^n .
\end{aligned}
\]
We estimate these terms after multiplication by
$U^{n+1}\langle v\rangle^{14}$ and integration over $x,v$.
The linear terms provide the $L_x^2D_1(7)$ dissipation,
up to an integrable multiple of
$\|U^{n+1}\|_{L_x^2L_7^2}^2$, by the proof of
Lemma~\ref{locallinear}.

For the first term, the second inequality of
\eqref{cutoffgfh}, with $N\to\infty$, gives
\[
\begin{aligned}
\left|(T^n)^{-3/2} \bigl(Q(U^n,G^n),U^{n+1}\langle v\rangle^{14}\bigr)_{L^2_{x,v}}\right| &\leq\delta(T^n)^{-3/2} \|U^{n+1}\|_{L_x^2D_1(7)}^2 \\
&\quad{} +C_\delta(T^n)^{-3/2} \|G^n\|_{H_x^2D_1(8)}^2\|U^n\|_{L_x^2L_7^2}^2 .
\end{aligned}
\]
Here we use the $H_x^2$ norm of $G^n$ to obtain an
$L_x^\infty$ bound. The extra velocity weight in $D_1(8)$
is precisely the one in \eqref{cutoffgfh}.

For the terms containing the differences of $T'/T$ and
$R/\sqrt T$, we integrate by parts and put the derivative
on $U^{n+1}\langle v\rangle^{14}$. The required norms of
$G^n$ are bounded by
\[
\|G^n\|_{L_x^2L_{19/2}^2} +\|G^n\|_{H_x^1L_8^2} \leq C_k(1+\|g^n\|_{X_k}).
\]
This follows from $k\geq17$ and the interpolation estimates
of Section~\ref{Landau}. In particular, this part of the
argument does not require a velocity derivative of the
initial data. After Young's inequality, its contribution is
bounded by a small part of the dissipation and
\[
C_\nu\left(\|U^n\|_{L_x^2L_7^2}^2+ |V^n-V^{n-1}|^2+|T^n-T^{n-1}|^2\right).
\]

For the coefficient of spherical diffusion, smoothness of $a_\nu$ gives
\[
\left\|\frac{a_\nu(V^n+\sqrt{T^n}v)}{T^n} -\frac{a_\nu(V^{n-1}+\sqrt{T^{n-1}}v)}{T^{n-1}} \right\|_{L_v^\infty} \leq C_\nu\left(|V^n-V^{n-1}|+|T^n-T^{n-1}|\right).
\]
To see that this bound is uniform in $v$, differentiate the
coefficient with respect to $V$ and $T$. The velocity factor
arising from differentiation in $T$ is bounded by the decay
of $\nabla a_\nu$ at infinity; near the origin the cutoff
makes the coefficient smooth.
We use this bound after integrating the divergence by parts.
The remaining factors are $\nabla_vG^n$ and
$\nabla_v(U^{n+1}\langle v\rangle^{14})$, which are controlled
by $1+\|g^n\|_{Y_k}$ and the difference dissipation.
The collision coefficient difference is treated by
\eqref{cutoffgfh}, and the transport difference is controlled by the
$H_x^1L_8^2$ bound above. Together, these contributions are
bounded by a small part of the dissipation and
\[
C_\nu(1+\|g^n\|_{Y_k}^2) \left(|V^n-V^{n-1}|^2+|T^n-T^{n-1}|^2\right).
\]

Integrating the resulting inequality, using
\eqref{nu:ODE-difference}, and applying Gronwall's inequality,
we obtain
\begin{equation}\label{nu:contraction}
\begin{aligned}
&\sup_{[0,s]}\|U^{n+1}\|_{L_x^2L_7^2}^2 +c\int_0^s(T^n)^{-3/2} \|U^{n+1}\|_{L_x^2D_1(7)}^2dt \\
&\quad\leq\left[ C_k\int_0^s(T^n)^{-3/2}\|G^n\|_{H_x^2D_1(8)}^2dt +C_\nu\left(s+s^2+ s^2\int_0^s\|g^n\|_{Y_k}^2dt\right)\right] \sup_{[0,s]}\|U^n\|_{L_x^2L_7^2}^2 .
\end{aligned}
\end{equation}
We make the smallness of the bracket explicit. Since $k\geq17$,
the definition \eqref{functional} and the interpolation
estimate \eqref{interpolation2} imply
\[
\|G^n\|_{H_x^2D_1(8)}^2
 \leq C_k\bigl(1+\|g^n\|_{Y_k}^2\bigr),
\]
where the constant term is the contribution of $\mu$.
It follows from \eqref{nu:iteration-bound} that
\[
C_k\int_0^s(T^n)^{-3/2}\|G^n\|_{H_x^2D_1(8)}^2\,dt \leq C_k\left(sT_0^{-3/2} +\int_0^s(T^n)^{-3/2}\|g^n\|_{Y_k}^2\,dt\right) \leq C_k\left(sT_0^{-3/2} +\frac{\varepsilon_0^2}{\lambda_0}\right) \leq\frac14.
\]
Here we first choose $\varepsilon_0$ small, and then decrease
$s$. These choices are independent of $\nu$. Moreover,
\[
\int_0^s\|g^n\|_{Y_k}^2\,dt
 \leq\left(\frac{3T_0}{2}\right)^{3/2}
 \frac{\varepsilon_0^2}{\lambda_0}.
\]
For fixed $\nu$, we can therefore choose $s_\nu\leq s$
so that, uniformly in $n$,
\[
C_\nu\left(s_\nu+s_\nu^2 +s_\nu^2\int_0^{s_\nu}\|g^n\|_{Y_k}^2\,dt\right) \leq C_\nu\left[ s_\nu+s_\nu^2\left(1+ \left(\frac{3T_0}{2}\right)^{3/2} \frac{\varepsilon_0^2}{\lambda_0}\right)\right] \leq\frac14.
\]
Thus \eqref{nu:contraction} and induction on $n\geq1$ give
\[
\sup_{[0,s_\nu]}\|U^{n+1}\|_{L_x^2L_7^2}^2 +c\int_0^{s_\nu}(T^n)^{-3/2} \|U^{n+1}\|_{L_x^2D_1(7)}^2\,dt \leq\frac12\sup_{[0,s_\nu]}\|U^n\|_{L_x^2L_7^2}^2 \leq2^{-n}\sup_{[0,s_\nu]}\|g^1\|_{L_x^2L_7^2}^2.
\]
For $m>n\geq1$, summing
$g^m-g^n=\sum_{j=n}^{m-1}U^{j+1}$ and using
$T^j\leq3T_0/2$ yields
\[
\begin{aligned}
\|g^m-g^n\|_{C([0,s_\nu],L_x^2L_7^2)} +\|g^m-g^n\|_{L^2([0,s_\nu],L_x^2D_1(7))} &\leq C_{k,T_0}\sum_{j=n}^{m-1}2^{-j/2} \|g^1\|_{C([0,s_\nu],L_x^2L_7^2)} \\
&\quad \leq\frac{C_{k,T_0}2^{-n/2}}{1-2^{-1/2}} \|g^1\|_{C([0,s_\nu],L_x^2L_7^2)} \longrightarrow0\qquad(n\to\infty).
\end{aligned}
\]
Consequently, the whole sequence converges in both spaces;
in particular, $g^n$ and $g^{n+1}$ converge to the same limit
$g_\nu$. Estimate \eqref{nu:ODE-difference} also gives
\[
\begin{aligned}
\|V^n-V^{n-1}\|_{C([0,s_\nu])} +\|T^n-T^{n-1}\|_{C([0,s_\nu])} &\leq C_\nu s_\nu \|U^n\|_{C([0,s_\nu],L_x^2L_7^2)} \\
&\quad \leq C_\nu s_\nu\,2^{-(n-1)/2} \|g^1\|_{C([0,s_\nu],L_x^2L_7^2)}.
\end{aligned}
\]
These differences are summable, so $V^n,T^n$ converge
uniformly to $V_\nu,T_\nu$. Passing to the limit in
\eqref{linear2} gives a solution of \eqref{nu:g-system}.
The uniform bounds in \eqref{nu:iteration-bound} and weak
lower semicontinuity preserve the energy estimate, while
strong convergence preserves nonnegativity.

For two solutions with the same initial data, the preceding
estimates give a Gronwall inequality for the squared
$L_x^2L_7^2$ difference norm and the squared differences of $V,T$.
Its coefficient is bounded by
$C_\nu(1+\|g_1\|_{Y_k}^2+\|g_2\|_{Y_k}^2)$ and is integrable.
This proves uniqueness. The same approximation argument gives
$g_\nu\in C([0,s_\nu],X_k)$, as for the linear equation.
We have thus proved local existence and uniqueness for each
fixed $\nu$. The interval obtained at this step may depend
on $\nu$.

\subsection{Uniform estimates and continuation}

We next show that the solutions just constructed all exist
on the same time interval. This step uses the energy estimates,
rather than a uniform bound for the constants in the preceding
difference estimates.

\begin{thm}\label{thm:nu-local}
Let $k\geq17$ and let the initial data satisfy the conditions
stated above. There are positive constants
$\varepsilon_0,\lambda_0,c_k$, depending only on $k$, such that
the following holds. If
$0<\varepsilon_1<\varepsilon_0/16$ is sufficiently small and
$\|g_0\|_{X_k}<\varepsilon_1^2$, then, for every $0<\nu\leq1$,
system \eqref{nu:g-system} has a unique solution on $[0,t_0]$,
where
\begin{equation}\label{T02}
t_0=c_k\min\left\{1,\varepsilon_1^2T_0^{3/2}, \varepsilon_1^4T_0^3\right\}.
\end{equation}
It satisfies $g_\nu\in C([0,t_0],X_k)\cap L^2([0,t_0],Y_k)$,
$(V_\nu,T_\nu)\in W^{1,2}(0,t_0)$, and
\begin{equation}\label{nu:uniform-energy}
\sup_{[0,t_0]}\|g_\nu\|_{X_k}^2+ \lambda_0\int_0^{t_0}T_\nu^{-3/2}\|g_\nu\|_{Y_k}^2dt \leq\frac12\varepsilon_1^2.
\end{equation}
Moreover, $\mu+g_\nu\geq0$, $\iint(1,v,|v|^2)g_\nu=0$, and
\begin{align}
\frac12T_0&\leq T_\nu(t)\leq\frac32T_0,\label{T(0)} \\
|V_\nu(t)-V_0-Et|
&\leq C_kT_0^{-1}
\left(t+t^{1/2}T_0^{3/4}\varepsilon_1\right)
\qquad(0\leq t\leq t_0).\label{Vh001}
\end{align}
The time $t_0$ and the constants in these estimates are
independent of $\nu,E,V_0$.
\end{thm}

\begin{proof}
We start from the solution on $[0,s_\nu]$.
In the energy calculation for \eqref{nu:g-system}, the
collision contribution $Q(g_\nu,g_\nu)$ is absorbed when
$\|g_\nu\|_{X_k}\leq\varepsilon_0$.
By \eqref{Tprime}, the terms containing $T_\nu'$ and $R_\nu$
are bounded by
\[
C_kT_\nu^{-3/2} \left(1+\|g_\nu\|_{L_x^2H_3^{3/4}}\right) \left(\|g_\nu\|_{X_k}^2+\|g_\nu\|_{X_k}\right).
\]
Since $\|g_\nu\|_{L_x^2H_3^{3/4}}\leq C_k\|g_\nu\|_{Y_k}$,
Young's inequality bounds this expression by a small part
of the collision dissipation and
$C_kT_\nu^{-3/2}(1+\|g_\nu\|_{X_k}^2)$.
The spherical terms are handled by
\eqref{nu:ion-weight}--\eqref{nu:ion-Maxwellian}.
We conclude that, while $\|g_\nu\|_{X_k}\leq\varepsilon_0$,
\begin{equation}\label{energybound}
\frac d{dt}\|g_\nu\|_{X_k}^2+ \lambda_0T_\nu^{-3/2}\|g_\nu\|_{Y_k}^2 \leq C_kT_\nu^{-3/2}\bigl(1+\|g_\nu\|_{X_k}^2\bigr).
\end{equation}
The nonnegative spherical dissipation may also be retained on
the left-hand side.

Consider an interval on which
$T_0/2\leq T_\nu\leq3T_0/2$ and
$\|g_\nu\|_{X_k}^2\leq\varepsilon_1^2$.
Integrating \eqref{energybound} gives
\begin{equation}\label{nu:bootstrap-energy}
\sup_{[0,t]}\|g_\nu\|_{X_k}^2 +\lambda_0\int_0^tT_\nu^{-3/2}\|g_\nu\|_{Y_k}^2ds \leq C_ke^{C_ktT_0^{-3/2}} \left(\|g_0\|_{X_k}^2+tT_0^{-3/2}\right).
\end{equation}
The initial term is smaller than $\varepsilon_1^4$.
Thus, by first choosing $\varepsilon_1$ small and then the
constant $c_k$ in \eqref{T02} small, the right-hand side is
at most $\varepsilon_1^2/2$ for $t\leq t_0$.
This strictly improves the assumed bound on $g_\nu$.

We next estimate the temperature on the same interval.
Multiplying the first inequality in \eqref{Tprime} by
$\frac32\sqrt{T_\nu}$, integrating, and applying
Cauchy's inequality in time, we obtain
\begin{equation}\label{nu:temperature-bootstrap}
|T_\nu(t)^{3/2}-T_0^{3/2}| \leq C\left(t+t^{1/2} \left(\int_0^t\|g_\nu\|_{Y_k}^2ds\right)^{1/2}\right) \leq C_k\left(t+t^{1/2}T_0^{3/4}\varepsilon_1\right).
\end{equation}
For the choice \eqref{T02}, this is smaller than
\[
\frac12\min\{1-2^{-3/2},(3/2)^{3/2}-1\}\,T_0^{3/2},
\]
after decreasing $c_k$ if necessary. The temperature therefore
stays strictly inside the interval $[T_0/2,3T_0/2]$.
Finally, integrating $V_\nu'=E-2R_\nu$ and using \eqref{Rh}
with the same time estimate proves \eqref{Vh001}.

These strict improvements allow continuation to $t_0$.
Indeed, fix $\nu>0$ and suppose its maximal existence time
were $s_*<t_0$. The preceding estimates bound
$\|g_\nu\|_{X_k}$ strictly below the smallness threshold,
keep $T_\nu$ away from zero, and give
$|V_\nu|\leq |V_0|+|E|t_0+C_{k,T_0}$.
The local construction can therefore be applied at any
time $s\leq t_0$ with these same bounds.


For later use, retaining the spherical dissipation in the energy
calculation also gives
\begin{equation}\label{nu:ion-energy}
\sum_{|\alpha|=0,2}\int_0^{t_0}\frac1{T_\nu} \iint\langle v\rangle^{2(k-4|\alpha|)} a_\nu(V_\nu+\sqrt{T_\nu}v) \nabla_v\partial_x^\alpha g_\nu\cdot \nabla_v\partial_x^\alpha g_\nu\,dx\,dv\,dt \leq C_k\varepsilon_1^2 .
\end{equation}
In particular, this estimate contains no constant which
diverges as $\nu\downarrow0$.

It remains to verify the normalized moments. Mass is conserved
in the equation for $G_\nu=\mu+g_\nu$.
Testing that equation with $v$ and $|v|^2$, using
\eqref{nu:structure} and the definition of $R_\nu$, gives
\[
\begin{aligned}
\frac d{dt}\iint vG_\nu &=-\frac{T_\nu'}{2T_\nu}\iint vG_\nu, \\
\frac d{dt}\left(\iint|v|^2G_\nu-3\right) &=-\frac{T_\nu'}{T_\nu} \left(\iint|v|^2G_\nu-3\right) +\frac4{\sqrt{T_\nu}}R_\nu\cdot\iint vG_\nu .
\end{aligned}
\]
The initial values of both quantities are zero, so they
remain zero. Hence $V_\nu,T_\nu$ are precisely the moments
of the distribution obtained by the inverse change of
variables. This distribution solves \eqref{nu:physical}.
\end{proof}

\subsection{Passage to the limit}

We now remove the cutoff. 

\begin{thm}\label{localwellposedness}
Under the assumptions and smallness choices of
Theorem~\ref{thm:nu-local}, the system
\eqref{equg}, \eqref{dtV}, \eqref{dtT} has a unique solution
with $g\in L^\infty([0,t_0],X_k)\cap L^2([0,t_0],Y_k)$,
$(V,T)\in W^{1,2}(0,t_0)$, and $T>0$. It satisfies
\begin{equation}\label{gt0vep}
\mu+g\geq0,\qquad \|g\|_{L^\infty([0,t_0],X_k)}^2+ \lambda_0\int_0^{t_0}T^{-3/2}\|g\|_{Y_k}^2dt \leq\frac12\varepsilon_1^2 .
\end{equation}
The bounds \eqref{T(0)} and \eqref{Vh001}
also hold for this solution. Its inverse transform $F$
solves \eqref{main} with initial data $F_0$.
\end{thm}

\begin{proof}
Throughout this limiting argument, $E,V_0,T_0$ are fixed.
The bounds for $V_\nu,T_\nu,T_\nu^{-1}$ are uniform in $\nu$.
Translations and dilations with these parameters are bounded
on the polynomially weighted spaces used here. For the
angular part of $Y_k$, a translation produces an additional
bounded multiple of the velocity gradient, which is already
controlled by $Y_k$. The inverse change of variables and
\eqref{nu:uniform-energy} therefore give
\begin{equation}\label{FG}
\sup_{0<\nu\leq1} \left(\|F_\nu\|_{L^\infty([0,t_0],X_k)} +\|F_\nu\|_{L^2([0,t_0],Y_k)}\right)<\infty .
\end{equation}
Also, the smallness of $g_\nu$ and Sobolev embedding in $x$
imply $\int F_\nu(t,x,v)\,dv\geq1/2$.
The constants in \eqref{FG} may depend on $E,V_0,T_0$, but are independent of $\nu$, as required for compactness.

We bound $\partial_tF_\nu$ directly from
\eqref{nu:physical}. The transport and electric-field terms
are bounded in a negative Sobolev space by \eqref{FG}.
The collision term is treated with \eqref{upperQ2}.
For the spherical term, a smooth test function $\psi$ satisfies
\[
\left|\iint a_\nu(v)\nabla_vF_\nu\cdot\nabla_v\psi\right| \leq\|\nabla_vF_\nu\|_2\,\|a_\nu\nabla_v\psi\|_2 \leq C\|F_\nu\|_{Y_k}\|\psi\|_{L_x^2H_v^4}.
\]
The second bound is uniform in $\nu$: on $|v|<1$ use
$|a_\nu(v)|\leq C|v|^{-1}$, whose square is locally
integrable in three dimensions, and on $|v|\geq1$ use
the boundedness of the coefficient. We conclude that
\begin{equation}\label{nu:physical-time-bound}
\sup_{0<\nu\leq1} \|\partial_tF_\nu\|_{L^2([0,t_0],(H_x^2H_6^4)^*)}<\infty .
\end{equation}
Lemma~\ref{compactSobolevembedding} and the
Aubin--Lions lemma now give, after extraction of a subsequence,
\begin{equation}\label{nu:strong-F}
F_\nu\longrightarrow F \quad\hbox{strongly in }L^2([0,t_0],H_x^1H_7^{3/4}).
\end{equation}
The bounds \eqref{FG} also give weak convergence in the
spaces containing the velocity derivatives.

We next pass to the limit in the equation. For the collision
term, write
\[
Q(F_\nu,F_\nu)-Q(F,F) =Q(F_\nu-F,F_\nu)+Q(F,F_\nu-F).
\]
By \eqref{upperQ2}, when paired with a smooth compactly
supported test function, the time integral of these terms
is bounded by
\[
C_\psi\bigl(\|F_\nu\|_{L_t^2H_x^2L_7^2} +\|F\|_{L_t^2H_x^2L_7^2}\bigr) \|F_\nu-F\|_{L_t^2L_x^2L_7^2}.
\]
This bound tends to zero by \eqref{FG} and \eqref{nu:strong-F}.

For the spherical diffusion, we have
$\nabla_vF_\nu\rightharpoonup\nabla_vF$ in $L^2$.
For each smooth compactly supported $\psi$,
$a_\nu\nabla_v\psi\to a\nabla_v\psi$ strongly in $L^2$,
by dominated convergence with majorant
$C|v|^{-2}|\nabla_v\psi|^2$ for the squared norm.
Thus
\[
\int_0^{t_0}\iint a_\nu\nabla_vF_\nu\cdot\nabla_v\psi \longrightarrow \int_0^{t_0}\iint a\nabla_vF\cdot\nabla_v\psi .
\]
The transport and electric-field terms pass to the limit
by strong convergence. Including the initial term in the
weak formulation identifies the initial value as $F_0$.
Nonnegativity follows from that of $F_\nu$.

We now recover $V,T$ and the normalized solution.
Equations \eqref{Rh} and \eqref{Tprime}, together with
\eqref{nu:uniform-energy}, bound $V_\nu,T_\nu$ in
$W^{1,2}(0,t_0)$ uniformly in $\nu$. Taking a further
subsequence, we obtain $V_\nu\to V$ and $T_\nu\to T$
uniformly, with $T\geq T_0/2$.
The weighted convergence \eqref{nu:strong-F} allows us to
pass the first and second velocity moments to the limit.
Since $V_\nu,T_\nu$ were shown to be the moments of $F_\nu$,
the limits $V,T$ are exactly the moments of $F$.

For completeness, the friction integral also converges
directly. The part removed by the cutoff satisfies
\[
\left\|\iint\frac{v}{|v|^3} \bigl(\chi^c_{\nu^2}(v)-1\bigr)F_\nu\,dx\,dv\right\|_{L_t^2} \leq C\nu^{1/4}\|F_\nu\|_{L_t^2L_x^2L_v^4} \longrightarrow0.
\]
Indeed, Definition~\ref{cutofffunction} gives
$1-\chi^c_{\nu^2}=\chi_{\nu^2}$, with
$0\leq\chi_{\nu^2}\leq1$ and support in $|v|\leq\sqrt2\nu$.
Integration in spherical coordinates gives
\[
\bigl\||v|^{-2}\mathbf1_{\{|v|\leq\sqrt2\nu\}}\bigr\|_{L_v^{4/3}}
 =\left(4\pi\int_0^{\sqrt2\nu}r^{-8/3}r^2\,dr\right)^{3/4}
 =(12\pi)^{3/4}(\sqrt2\nu)^{1/4}=C\nu^{1/4}.
\]
By the triangle inequality, the absolute value of the friction
error at time $t$ is at most the integral below. H\"older's
inequality in $v$ and Cauchy--Schwarz in $x$ yield
\[
\int_{\mathbb T^3}\int_{|v|\leq\sqrt2\nu}
 \frac{|F_\nu(t,x,v)|}{|v|^2}\,dv\,dx
 \leq C\nu^{1/4}\int_{\mathbb T^3}
 \|F_\nu(t,x,\cdot)\|_{L_v^4}\,dx
 \leq C\nu^{1/4}\|F_\nu(t)\|_{L_x^2L_v^4}.
\]
Taking the $L^2([0,t_0])$ norm proves the stated estimate.
Finally, $H_v^{3/4}\hookrightarrow L_v^4$ and \eqref{FG} give
\[\sup_{0<\nu\leq1}\|F_\nu\|_{L^2([0,t_0],L_x^2L_v^4)}<\infty,\]
so the right-hand side tends to zero.
The remaining difference
$\iint v|v|^{-3}(F_\nu-F)\,dx\,dv$ tends to zero by
\eqref{nu:strong-F}: use $H^{3/4}\hookrightarrow L^4$
near the origin and the velocity weight away from it.
It follows that the limiting moments satisfy
\eqref{dtV}--\eqref{dtT}. 

Define $g$ from $F,V,T$ by \eqref{changeofvariable}.
Uniform convergence of $V_\nu,T_\nu$ and continuity of
translations and dilations give
$g_\nu\to g$ in $L_t^2H_x^1H_7^{3/4}$.
The uniform energy bounds pass to the limit by weak lower
semicontinuity, yielding \eqref{gt0vep}.
The bounds for $V,T$ pass to the limit by uniform convergence.
To pass the spherical dissipation in \eqref{nu:ion-energy}, first express it in terms of $F_\nu$ by the inverse change of variables and
restrict the integral to $r<|v|<R$.
On this region the coefficients and the transformed weights
converge uniformly. Lower semicontinuity applies there;
letting $r\downarrow0$ and $R\uparrow\infty$ gives the
corresponding estimate for the original spherical diffusion.


Finally, we prove uniqueness.
Let $F_1,F_2$ be two solutions in the stated class.
Their local mass is bounded below, and their weighted
$L^2$ norms are bounded uniformly in $x,t$. 

Subtracting the equations, we write the collision difference
as $Q(F_1,F_1-F_2)+Q(F_1-F_2,F_2)$.
Test with $(F_1-F_2)\langle v\rangle^{14}$.
The spherical diffusion is nonpositive because $a(v)v=0$,
and the electric-field term is bounded by
$C|E|\|F_1-F_2\|_{L_x^2L_7^2}^2$.
The second collision term is estimated by \eqref{cutoffgfh}
with $N\to\infty$, giving
\begin{equation}\label{nu:limit-uniqueness}
\frac d{dt}\|F_1-F_2\|_{L_x^2L_7^2}^2 \leq C\left(1+|E|+\|F_2\|_{H_x^2D_1(8)}^2\right) \|F_1-F_2\|_{L_x^2L_7^2}^2.
\end{equation}
The coefficient on the right-hand side is integrable by \eqref{FG}. 
Gronwall's inequality gives $F_1=F_2$. Their moments and
the change of variables then give uniqueness of $V,T,g$.
\end{proof}

In the next section we use the moment equations in
\eqref{eqG} and the Maxwellian integral \eqref{I2}.
Estimate \eqref{splitregime} remains valid with the cutoff
replaced by one. 

\begin{rmk}\label{nu:continuation}
The local construction also gives the continuation needed
later in the proof of global existence.
Suppose a solution has a finite maximal time $s_*$.
If $T^{-1},T,|V|$ remain bounded, $\|g\|_{X_{17}}$ remains
sufficiently small, and the high $X_k$ norm and the time
integral of the dissipation remain bounded on $[0,s_*)$,
then the solution extends beyond $s_*$.

Indeed, one first applies the local construction at index
$17$. The required smallness bound can be chosen uniformly
for all initial times below $s_*$.
For the higher index $k$, \eqref{cutoffgfh} and
\eqref{Qal1al2} show that the nonlinear collision term
requires only smallness of the lower $X_{17}$ norm.
The coefficient of the remaining drift terms is bounded by
$CT^{-3/2}(1+\|g\|_{Y_{17}})$, which is integrable on the
finite interval. Gronwall's inequality therefore propagates
the higher norm. Applying the construction at a time close
to $s_*$ gives the extension.

\end{rmk}
\par\endgroup

\section{\texorpdfstring{Sharp growth estimates for $V$ and $T$}{Sharp growth estimates for V and T}}\label{growth}

Theorem \ref{localwellposedness} gives the solution on the common interval $[0,t_0]$. The time in \eqref{T02} depends on $k,T_0$ and the chosen perturbation size, but is independent of $E$ and $V_0$. We now derive the large-time estimates needed for continuation; on each finite interval the relevant criterion is Remark~\ref{nu:continuation}.

In this section, we establish the long-time behavior of \(V(t)\), \(T(t)\) and \(R(t)\) under the assumption that the smallness of \(g\) is propagated. In particular, we obtain the lower bound \( T(t) \geq \frac{T_0}{4} \) for any $t>0$. 

\begin{lem}\label{largetime}
Let $k\geq 17$ and let \(t_0\) be as in Theorem \ref{localwellposedness}. Let \(S>t_0\) be a finite time of existence. Assume a priori that
\begin{equation}\label{apriofg}
\|g\|_{L^\infty([0,S],L^2_xL^2_5)}+\|g\|_{L^2([0,S],L^2_xH^{3/4}_5)}<\eta
\end{equation}
 for some small $0<\eta<1$. Then there exists a constant $C_{k,T_0,V_0,t_0}$ such that if $|E|>C_{k,T_0,V_0,t_0}$, then for all $t_0<t\leq S$, 
\begin{equation}\label{longtimeTVR}
\begin{aligned}
& T(t)\geq \f14 T_0+\f{c_1}{|E|}\ln (3+|E|t),~~T(t)\leq \f32 T_0+\f {c_2}{|E|}\ln (3+|E|t); \\
& |V(t)-V_0-Et|\ls T_0^{-1}+1/|E|,\quad\mbox{for}\quad c_1,c_2>0.
\end{aligned}
\end{equation}
The field threshold may depend on the fixed \(t_0\), but not on \(S\).
The constants \(c_1,c_2\) are independent of \(E,S\).
\end{lem}
\begin{proof}
Theorems \ref{thm:nu-local} and \ref{localwellposedness} give the short-time bounds on $[0,t_0]$. We now work with the uncut solution and derive estimates for $t>t_0$.

{\bf $\bullet$ Long-time behavior of $V$, $T$ and $R$.} We first integrate $V'=E-2R$ and use \eqref{Rh}, \eqref{T(0)}, and \eqref{apriofg} to obtain
\begin{equation}\label{Vttleqt0}
|V(t)|\geq |V_0+Et|-CT_0^{-1},\quad t\leq t_0.
\end{equation}
Combining \eqref{eqG}, \eqref{I2} and \eqref{splitregime} (with $m=2$), and using $\mathfrak{G}(x)\ls \f1{x^2}$ from \eqref{Gx}, we obtain
\begin{equation}\label{R4}
|R(t)|\ls |V(t)|^{-2}(1+\|g(t)\|_{L^2_xH^{3/4}_3}),\quad\forall t>0.
\end{equation}
This yields
\[
\begin{aligned}
|V(t)| &\geq |V_0+Et|-CT_0^{-1}-C\int_{t_0}^t|V(\tau)|^{-2}(1+\|g(\tau)\|_{L^2_xH^{3/4}_3})d\tau, \\
&\geq |V_0+Et|-CT_0^{-1}-C\Big(\int_{t_0}^t|V(\tau)|^{-2}d\tau+\eta \big(\int_{t_0}^t|V(\tau)|^{-4}d\tau \big)^{\f12}\Big),\quad t>t_0.
\end{aligned}
\]
Choose \(|E|t_0\) larger than a fixed multiple of
\(1+|V_0|+T_0^{-1}\). Under the temporary lower bound
\(|V(t)|\geq(3+|E|t)/8\), the two integrals above are bounded by
\(C/|E|\) and \(C\eta/|E|^{1/2}\), respectively.
Increasing \(|E|\) strictly improves this temporary bound.
Thus, by \eqref{Vttleqt0} and continuity, there exists a constant $C_{k,T_0,V_0,t_0}$ such that if $|E|>C_{k,T_0,V_0,t_0}$, then
\begin{equation}\label{Vttgeqt0}
|V(t)|\geq \f12|V_0+Et|\geq \f14(3+|E|t),\quad t> t_0.
\end{equation}
This gives the lower bound for $|V(t)|$. The same argument gives the upper bound
\begin{equation}\label{upperboundofV}
\begin{aligned}
|V(t)| &\leq |V_0+Et|+CT_0^{-1}+C\Big(\int_{t_0}^t|V(\tau)|^{-2}d\tau+\eta \big(\int_{t_0}^t|V(\tau)|^{-4}d\tau \big)^{\f12}\Big) \\
&\leq 3(3+|E|t),\quad t>t_0.
\end{aligned}
\end{equation}
The integral form of the velocity equation also gives
\begin{equation}\label{upperboundV}
|V(t)-V_0-Et|\ls \int_0^{t_0}|R(\tau)|\,d\tau+\int_{t_0}^{t}|R(\tau)|\,d\tau\ls T_0^{-1}+1/|E|,\quad t>t_0,
\end{equation}
where we used \eqref{Rh}, \eqref{R4} and \eqref{Vttgeqt0} to bound the time integrals.

We next estimate $T(t)$ for $t> t_0$. The short-time estimate $T_0/2\leq T(t)\leq 3T_0/2$ on $[0,t_0]$ is already available (see \eqref{T(0)}).

Combining \eqref{R4} with $T'(t)=\f43 V(t)\cdot R(t)$, we obtain
\begin{equation}\label{T'}
|T'(t)|\ls |V(t)|^{-1}(1+\|g\|_{L^2_xH^{3/4}_3}),
\end{equation}
Integrating in time gives
\begin{equation}\label{upperboundT}
\begin{aligned}
|T(t)|&\leq\f32 T_0+ C\int_{t_0}^t|V(\tau)|^{-1}(1+\|g(\tau)\|_{L^2_xH^{3/4}_3})d\tau \\
&\leq \f32 T_0+\f {c_2}{|E|}\ln (3+|E|t),\quad t\geq t_0
\end{aligned}
\end{equation}
for some $c_2>0$ independent of \(t_0,E,S\).
Indeed, the perturbative time integral is at most
\(C\eta/(|E|\sqrt{t_0})\) by Cauchy--Schwarz and \eqref{Vttgeqt0}.
For fixed \(t_0\), enlarge the field threshold so that
\(\log(3+|E|t_0)\geq\eta/\sqrt{t_0}\); this absorbs that term into the
displayed logarithm. Together with \eqref{Vttgeqt0}, this gives
\[
\mathfrak{G}(\f{|V(t)|}{\sqrt{2T(t)}})\geq C \f{T(t)}{|V(t)|^2}
\]
because of \eqref{Gx2} and the fact that $|V(t)|/\sqrt{2T(t)}>10$ for $t>t_0$ when $|E|>C_{k,V_0,T_0,t_0}$ is large. Using the moment equations in \eqref{eqG}, identity \eqref{I2} and estimate \eqref{splitregime} (with $m=3$), we obtain
\[
\begin{aligned}
& V(t)\cdot R(t)=\f{|V(t)|}{T(t)}\mathfrak{G}(\f{|V(t)|}{\sqrt{2T(t)}})+V(t)\cdot T(t)^{-1}\int_{\T^3\times\R^3}\f{v+\om(t)}{|v+\om(t)|^3}gdvdx \\
&\geq (C-\eta )|V(t)|^{-1}-CT(t)^{\f12}|V(t)|^{-2}\|g\|_{L^2_xH^{3/4}_4},\quad t>t_0.
\end{aligned}
\]
Taking $\eta$ sufficiently small and combining \eqref{upperboundofV}, \eqref{upperboundT} and \eqref{Vttgeqt0}, we obtain
\[
\begin{aligned}
T(t) &\geq T(t_0)+C\int_{t_0}^t |V(\tau)|^{-1}\,d\tau-C\int_{t_0}^t T(\tau)^{\f12}|V(\tau)|^{-2}\|g(\tau)\|_{L^2_xH^{3/4}_4}\,d\tau \\
&\geq \f12 T_0+\f{c_1}{|E|}\big(\ln (3+|E|t)-\ln (3+|E|t_0)\big)-\f C{|E|} \\
&\geq \f14 T_0+\f{c_1}{|E|}\ln (3+|E|t),\quad t>t_0
\end{aligned}
\]
for some $c_1>0$ and large $|E|$. 

This completes the proof of the lemma.
\end{proof}

\section{Coupled system}\label{coupledsystem}
In this section we derive a priori estimates for \(g\) starting from time \(t_0\) with \(\|g(t_0)\|_{X_k}<\varepsilon_1\) (see \eqref{gt0vep}). Since we will use the properties of $\om(t)=V(t)/\sqrt{T(t)}$ repeatedly, we first record the relevant bound.

$\bullet$ \textbf{The behavior of $\om(t)$ when $t>t_0$: } It follows from \eqref{Vttgeqt0} and \eqref{longtimeTVR} that there exists a constant $C_{k,T_0,V_0,t_0,N}$ such that if $|E|>C_{k,T_0,V_0,t_0,N}$, then
\begin{equation}\label{VTcondition0}
|\om(t)|=|V(t)|/\sqrt{T(t)}>16N,\quad t>t_0.
\end{equation}
Here $N$ is a sufficiently large constant, ultimately chosen depending on $k$ and $T_0$ in Theorem \ref{energyestimate}.

We rewrite \eqref{equg} as
\begin{equation}\label{equg2}
\begin{aligned}
&\pa_t g+T(t)^{\f12}v\cdot\na_x g+(-\f12 T(t)^{-1}T'(t)v+2T(t)^{-\f12}R(t))\cdot\na_v g = T(t)^{-\f32}L(g) \\
&\quad{} + T(t)^{-\f32}Q(g,g)+\f32 T(t)^{-1} T'(t) g(t,x,v)+T(t)^{-\f32}\dv \Big(\f{\Pi(v+\om(t))}{|v+\om(t)|}\cdot \na_v g\Big)+\cH(\mu)
\end{aligned}
\end{equation}
where \(\cH(\mu)\) is defined by
\begin{equation}\label{cH}
\begin{aligned}
\cH(\mu) &:= -(-\f12 T(t)^{-1}T'(t)v+2T(t)^{-\f12}R(t))\cdot\na_v \mu \\
&\quad{} + \f32 T(t)^{-1} T'(t) \mu+T(t)^{-\f32}\mathrm{div} \Big(\f{\Pi(v+\om(t))}{|v+\om(t)|}\cdot \na_v \mu\Big).
\end{aligned}
\end{equation}
We introduce $g_1=g_1(t,x,v)$ and $g_2=g_2(t,x,v)$ by prescribing the equations for $g_1$ and $g_2$ as follows:
\begin{equation}\label{coupg1}
\begin{aligned}
&\pa_t g_1(t,x,v)+T(t)^{\f12}v\cdot\na_x g_1+(-\f12 T(t)^{-1}T'(t)v+2T(t)^{-\f12}R(t))\cdot\na_v g_1(t,x,v) \\
&= ~ T(t)^{-\f32}(L-A\chi_M)g_1+T(t)^{-\f32}Q(g_1,g_1)+\f32 T(t)^{-1} T'(t) g_1 \\
&\quad{} + {T(t)^{-\f32}\mathrm{div} \Big(\f{\Pi(v+\om(t))}{|v+\om(t)|}\cdot \na_v (g_1+\mu^{\f12}g_2)\Big)}+T(t)^{-\f32}(Q(\mu^\f12 g_2,g_1)+Q(g_1,\mu^{\f12}g_2)) \\
&\quad{} - (-\f12 T(t)^{-1}T'(t)v+2T(t)^{-\f12}R(t))\cdot(\na_vg_2-\f v 2 g_2)\mu^{\f12}+\f32T(t)^{-1}T'(t)\mu^{\f12}g_2+\cH(\mu),
\end{aligned}
\end{equation}
and
\begin{equation}\label{coupg2}
\pa_t g_2(t,x,v)+T(t)^{\f12}v\cdot\na_x g_2=T(t)^{-\f32}\mathcal{L}g_2+T(t)^{-\f32}\mu^{-\f12}A\chi_M(v)g_1+T(t)^{-\f32}\Ga(g_2,g_2)
\end{equation}
with initial data
\[
g_1(t_0,x,v)=g(t_0,x,v),\quad g_2(t_0,x,v)=0,
\]
where \(L\), \(\mathcal L\), \(\Gamma\) and \(A\chi_M\) are defined in \eqref{linearoperator}, \eqref{linearized} and \eqref{LAM}.
Setting $g=g_1+\mu^{\f12}g_2$, we verify by adding the two equations that $g$ solves \eqref{equg2} with initial data $g(t_0,x,v)$. 
For local existence, continue $(g,V,T)$ to a short interval $[t_0,t_0+\delta]$ by Theorem~\ref{localwellposedness} and Remark~\ref{nu:continuation}, with $T$ bounded above and away from zero. Substituting $g_1=g-\mu^{1/2}g_2$ into \eqref{coupg2}, we solve
\[
\partial_tg_2+\sqrt T\,v\cdot\nabla_xg_2
=T^{-3/2}\big[(\mathcal L-A\chi_M)g_2
+A\chi_M\mu^{-1/2}g+\Gamma(g_2,g_2)\big],
\qquad g_2(t_0)=0.
\]
The compact support of $\chi_M$ gives $\|A\chi_M\mu^{-1/2}g\|_{\mathcal E_k}\leq C_{k,A,M}\|g\|_{X_k}$. Starting from $g_2^0=0$, replace the nonlinear term by $\Gamma(g_2^n,g_2^{n+1})$ and solve the resulting linear problems by the approximation used in Lemma~\ref{locallinear}. Estimates \eqref{cLf}, \eqref{Gaghf} and \eqref{Gaal1al2} give uniform bounds in $L_t^\infty\mathcal E_k\cap L_t^2\mathcal D_k$. Since the initial value is zero, both the energy and the integrated dissipation can be made small by decreasing $\delta$. Applying the same estimates to consecutive differences therefore makes the iteration a contraction in these spaces. Its limit $g_2$ solves the displayed equation. Finally, set $g_1=g-\mu^{1/2}g_2$. The Gaussian factor yields $g_1\in L_t^\infty X_k\cap L_t^2Y_k$, and subtraction of the equations verifies \eqref{coupg1}--\eqref{coupg2} with the stated initial data. 

\subsection{Evolution of $g$ in the large region.} Recall the smooth cutoff functions $\chi_N, \chi^c_N$ and $\tilde\chi_N$ in Definition \ref{cutofffunction}. Using \eqref{equg2}, we first derive an energy estimate for $g$ in the region $|v|^2\gs N$. This gives the following proposition.
\begin{prop}\label{g}
	Let $g$ be a solution of \eqref{equg2}. Then for $k\geq17$ and sufficiently large $N$ depending on $k$, the following estimate holds:
	{
\begin{equation}\label{evolutiong}
\begin{aligned}
& \f12\f d{dt}\|\chi^c_Ng\|^2_{X_k}+T(t)^{-\f32}(\f{\lam_0}2-C_k\|g\|_{X_{15}})\|\chi^c_N g\|^2_{Y_k}+T(t)^{-\f32}\mc_0k\|\chi^c_N g\|^2_{X_{k-\f32}} \\
&\leq T(t)^{-\f32}\Big(\f {\mC_0}k\|\tilde{\chi}_Ng\|^2_{X_{k-\f32}}+C_k\|\tilde{\chi}_N g\|^2_{X_{k-\f52}}+C_k\|\chi^c_N g\|^2_{X_8}+\frac{C_k}{N}\|g\|_{X_{11}}^2\Big) \\
&\quad{} +C_kT(t)^{-\f32}\|g\|_{X_{15}}\|\tilde\chi_N g\|^2_{X_{k-\f32}} \\
&\quad{} + C_k\big(T(t)^{-1}|T'(t)|+T(t)^{-\f12}|R(t)|+T(t)^{-1}|V(t)|^{-1}\big)(\|\chi^c_{N}g\|^2_{X_{k}}+\|\tilde\chi_{N}g\|^2_{X_{k}}) \\
&\quad{} + C_kT(t)^{\f32}(T(t)^{-1}|T'(t)|+T(t)^{-\f12}|R(t)|+T(t)^{-1}|V(t)|^{-1})^2,
\end{aligned}
\end{equation}
}
	where the constants $\lam_0,\mathbf{c}_0,\mathbf{C}_0$ are the uniform constants from Lemma \ref{cutoffQcom}.
\end{prop}
\begin{proof}
For $|\al|=0,2$, the equation satisfied by $\pa^\al_x g$ reads 
\begin{equation}\label{paxg}
\begin{aligned}
&\pa_t \pa^\al_xg+T(t)^{\f12}v\cdot\na_x \pa^\al_xg+(-\f12 T(t)^{-1}T'(t)v+2T(t)^{-\f12}R(t))\cdot\na_v \pa^\al_xg = T(t)^{-\f32}L(\pa^\al_xg) \\
&~+T(t)^{-\f32}\pa^\al_xQ(g,g)+\f32 T(t)^{-1} T'(t) \pa^\al_xg+T(t)^{-\f32}\dv \Big(\f{\Pi(v+\om(t))}{|v+\om(t)|}\cdot \na_v \pa^\al_x g\Big)+\pa^\al_x\cH(\mu).
\end{aligned}
\end{equation}
We multiply \eqref{paxg} by $(\chi^c_N)^2\<v\>^{2(k-4|\al|)}\pa^{\al}_x g$, integrate over $(x,v)\in \T^3\times\R^3$, and sum over $|\al|=0,2$. The first two terms on the left-hand side yield
\[
\f 12 \f d{dt}\sum_{|\al|=0,2}\int_{\T^3\times\R^3}|\pa^\al_x g|^2(\chi^c_N)^2\<v\>^{2(k-4|\al|)}dvdx=\f12\f d{dt}\|\chi^c_Ng\|^2_{X_k},
\]
where we recall the energy functional $X_k$ from \eqref{functional}.

The remaining terms, including the third term on the left-hand side, are denoted by $P_1,\ldots,P_6$ and estimated below.

For $P_1$ and $P_4$, integration by parts yields
\[
\begin{aligned}
P_1 &= -\frac{T'(t)}{4T(t)}\sum_{|\al|=0,2}\iint (\pa^\al_x g)^2\dv_v\big(v(\chi^c_N)^2\<v\>^{2(k-4|\al|)}\big)\,dvdx \\
&\quad{} + T(t)^{-\f12}\sum_{|\al|=0,2}\iint (\pa^\al_x g)^2R(t)\cdot\na_v\big((\chi^c_N)^2\<v\>^{2(k-4|\al|)}\big)\,dvdx.
\end{aligned}
\]
A direct computation gives
{
\[
\begin{aligned}
&P_4=\f32 T(t)^{-1}T'(t)\sum_{|\al|=0,2}\int_{\T^3\times\R^3}(\pa^\al_xg)^2(\chi^c_N)^2\<v\>^{2(k-4|\al|)}dvdx \\
&=\f32 T(t)^{-1}T'(t)\sum_{|\al|=0,2}\|\chi^c_N\pa^\al_xg\|^2_{L^2_xL^2_{k-4|\al|}}.
\end{aligned}
\]
}
Since $\na_v \chi^c_N=\f2 Nv(\chi^c_N)'$ and $N\leq |v|^2 \leq 2N$ on $\operatorname{supp}\nabla\chi^c_N$, we obtain 
\begin{equation}\label{P14}
|P_1|+|P_4|\leq C_k\big(T(t)^{-1}|T'(t)|+T(t)^{-\f12}|R(t)|\big)(\|\chi^c_{N}g\|^2_{X_{k}}+\|\tilde\chi_{N}g\|^2_{X_{k}}).
\end{equation}

For $P_2$ and $P_3$, we write 
\begin{equation*}
\begin{aligned}
P_2 &= T(t)^{-\f32}\sum_{|\al|=0,2}\int_{\T^3\times\R^3}L(\pa^\al_xg)(\chi^c_N)^2\<v\>^{2(k-4|\al|)}\pa^\al_x gdvdx \\
&= T(t)^{-\f32}\sum_{|\al|=0,2}\int_{\T^3\times\R^3}Q(\mu,\pa^\al_xg)(\chi^c_N)^2\<v\>^{2(k-4|\al|)}\pa^\al_x gdvdx
\end{aligned}
\end{equation*}
\begin{equation*}
	\begin{aligned}
 \quad\quad+ T(t)^{-\f32}\sum_{|\al|=0,2}\int_{\T^3\times\R^3}Q(\pa^\al_xg,\mu)(\chi^c_N)^2\<v\>^{2(k-4|\al|)}\pa^\al_x gdvdx,
	\end{aligned}
\end{equation*}
and
\[
\begin{aligned}
P_3 &= T(t)^{-\f32}\sum_{|\al|=0,2}\int_{\T^3}\big(Q(g,\pa^\al_xg),(\chi^c_N)^2\<v\>^{2(k-4|\al|)}\pa^\al_x g\big)_{L^2_v}dx \\
&\quad{} + T(t)^{-\f32}\sum_{|\al|=0,2}\sum_{\substack{\al_1+\al_2=\al\\|\al_1|\geq1}}\binom{\al}{\al_1}\int_{\T^3}\big(Q(\pa^{\al_1}_xg,\pa^{\al_2}_xg),(\chi^c_N)^2\<v\>^{2(k-4|\al|)}\pa^\al_x g\big)_{L^2_v}dx.
\end{aligned}
\]
The nonlocal term \(Q(\partial_x^\alpha g,\mu)\) requires its
Maxwellian tail to be retained. The nondivergence identity and
Cauchy--Schwarz give
\[
Q(h,\mu)=(a*h):\nabla_v^2\mu+8\pi h\mu,\qquad \|a*h\|_{L^\infty_v}\leq C\|h\|_{L^2_3}.
\]
Since \(\chi_N^c=0\) on \(|v|^2\leq N\), for each fixed \(l\) we obtain
\[
\left|(Q(h,\mu),h(\chi_N^c)^2\langle v\rangle^{2l})\right| \leq C_l e^{-cN} \left(\|h\|_{L^2_3}\|\chi_N^ch\|_{L^2_{l-3/2}} +\|\chi_N^ch\|_{L^2_{l-3/2}}^2\right).
\]
We apply this with \(h=\partial_x^\alpha g\) and \(l=k-4|\alpha|\),
and integrate in \(x\). Young's inequality, followed by
\(e^{-cN}\leq C/N\), bounds the sum by
\[
\frac{\lambda_0}{8}\|\chi_N^c g\|_{Y_k}^2 +\frac{C_k}{N}\|g\|_{X_{11}}^2
\]
for sufficiently large \(N\). Combining this with the stronger
coercivity retained in Lemma~\ref{cutoffQcom}, we obtain
\begin{equation}\label{P2}
\begin{aligned}
P_2 &\leq -T^{-3/2}\lambda_0\|\chi_N^c g\|_{Y_k}^2 \\
&\quad{} + T^{-3/2}\left(-\mathbf c_0 k\|\chi_N^c g\|_{X_{k-3/2}}^2 +\frac{\mathbf C_0}{k}\|\widetilde\chi_Ng\|_{X_{k-3/2}}^2 +C_k\|\widetilde\chi_Ng\|_{X_{k-5/2}}^2\right) \\
&\quad{} + C_kT^{-3/2}\|\chi_N^cg\|_{X_8}^2 +\frac{C_k}{N}T^{-3/2}\|g\|_{X_{11}}^2 .
\end{aligned}
\end{equation}
Here \(\|\partial_x^\alpha g\|_{L_x^2L_3^2}\leq\|g\|_{X_{11}}\)
for \(|\alpha|=0,2\). In particular, the final term is not discarded.

Similarly, \eqref{cutoffgfh} and \eqref{Qal1al2} give
\begin{equation}\label{P3}
|P_3|\leq C_kT(t)^{-\f32}\|g\|_{X_{15}}(\|\chi_N^c g\|^2_{Y_k}+\|\tilde\chi_N g\|^2_{X_{k-\f32}}).
\end{equation}
For $P_5$, integration by parts gives
\[
\begin{aligned}
P_5 &= -T(t)^{-\f32}\sum_{|\al|=0,2}\int_{\T^3\times\R^3}\f{\Pi(v+\om(t))}{|v+\om(t)|} \na_v \pa^\al_xg\na_v((\chi^c_N)^2\<v\>^{2(k-4|\al|)}\pa^\al_xg)dvdx \\
&= -P_{5,0}+P_{5,1}+P_{5,2}
\end{aligned}
\] 
with 
\[
\begin{aligned}
P_{5,0} &= T(t)^{-\f32}\sum_{|\al|=0,2}\int_{\T^3\times\R^3}\f{\Pi(v+\om(t))}{|v+\om(t)|}\na_v\pa^\al_xg\na_v\pa^\al_xg(\chi^c_N)^2\<v\>^{2(k-4|\al|)}dxdv \\
P_{5,1} &= -T(t)^{-\f32}\sum_{|\al|=0,2}2(k-4|\al|)\int_{\T^3\times\R^3}\f{\Pi(v+\om(t))}{|v+\om(t)|}\na_v\pa^\al_xg\pa^\al_xg(\chi^c_N)^2\<v\>^{2(k-4|\al|-1)}vdxdv \\
P_{5,2} &= -T(t)^{-\f32}\sum_{|\al|=0,2}2\int_{\T^3\times\R^3}\f{\Pi(v+\om(t))}{|v+\om(t)|}\na_v\pa^\al_xg\pa^\al_xg\na_v\chi^c_N\chi^c_N\<v\>^{2(k-4|\al|)}dxdv.
\end{aligned}
\]

For $P_{5,1}$, the Cauchy-Schwarz inequality gives 
\[
|P_{5,1}| \leq \f14 P_{5,0}+C_kT(t)^{-\f32}\sum_{|\al|=0,2}\int_{\T^3\times\R^3}\f1{|v+\om(t)|}(\pa^\al_xg)^2(\chi^c_N)^2\<v\>^{2(k-4|\al|-1)}dxdv.
\]
For the second term, split into $|v|\leq|\om(t)|/2$ and
$|v|>|\om(t)|/2$. On the latter region we use the smooth function
$\chi^c_{|\om(t)|^2/8}$ from Definition~\ref{cutofffunction}, which
equals one there and vanishes when $|v|\leq|\om(t)|/(2\sqrt2)$.
The Hardy inequality gives the upper bound
\[
\begin{aligned}
& C_kT(t)^{-1}|V(t)|^{-1}\|\chi^c_Ng\|^2_{X_{k-1}} \\
&\quad{} + C_kT(t)^{-\f32}\sum_{|\al|=0,2} \|\chi^c_{|\om(t)|^2/8}\chi^c_N\pa^\al_xg\|^2_{L_x^2H^{1/2}_{k-4|\al|-1}}.
\end{aligned}
\]
Since $|\om(t)|>16N$ and $N\geq1$, multiplication by this additional
cutoff is uniformly bounded on the weighted $H^1$ spaces.
Interpolation and its support property yield, for every $\vep>0$,
\[
\begin{aligned}
& T(t)^{-\f32}\|\chi^c_{|\om(t)|^2/8}\chi^c_Ng\|^2_{L_x^2H^{1/2}_{k-1}} \\
&\leq C_kT(t)^{-\f32}|\om(t)|^{-1/2} \|\chi^c_Ng\|_{L_x^2H^1_{k-3/2}}\|\chi^c_Ng\|_{L_x^2L^2_k} \\
&\leq \vep T(t)^{-\f32}\|\chi^c_Ng\|^2_{L_x^2H^1_{k-3/2}} +C_{k,\vep}T(t)^{-1}|V(t)|^{-1}\|\chi^c_Ng\|^2_{L_x^2L^2_k}.
\end{aligned}
\]
The same estimate applies to each spatial derivative.
This yields
\[
|P_{5,1}| \leq \f14 P_{5,0}+\f{\lam_0}4 T(t)^{-\f32}\|\chi^c_N g\|^2_{Y_k}+C_{k}T(t)^{-1}|V(t)|^{-1}\|\chi^c_N g\|^2_{X_k},
\]
upon choosing $\vep=\f{\lam_0}4$.

For $P_{5,2}$, since $\na_v \chi^c_N=\f2 Nv(\chi^c_N)'$ and $N\leq |v|^2 \leq 2N$ on $\operatorname{supp}\nabla\chi^c_N$, we have $|v|<\f12|\om(t)|$ (see \eqref{VTcondition0}) and consequently
\[
|P_{5,2}|\leq \f14P_{5,0}+C_kT(t)^{-1}|V(t)|^{-1}\|\tilde\chi_N g\|^2_{X_{k}}.
\]
Hence,
\begin{equation}\label{P5}
P_5\leq -\f12 P_{5,0}+\f{\lam_0}4 T(t)^{-\f32}\|\chi^c_N g\|^2_{Y_k}+C_{k}T(t)^{-1}|V(t)|^{-1}(\|\chi^c_N g\|^2_{X_k}+\|\tilde\chi_N g\|^2_{X_{k}}).
\end{equation}

For $P_6$, since $\pa^\al_x\cH(\mu)=0$ when $|\al|=2$, \eqref{cH} gives
\[
\begin{aligned}
|P_6| &= \big|(\cH(\mu),(\chi^c_N)^2\<v\>^{2k} g)_{L^2_{x,v}}\big| \\
&\leq C_k\big( T(t)^{-1}|T'(t)|+T(t)^{-\f12}|R(t)|\big)\|\chi^c_N g\|_{L^2_{x,v}}+|P_{6,1}|+|P_{6,2}|
\end{aligned}
\]
with 
\[
\begin{aligned}
P_{6,1} &:= T(t)^{-\f32}\Big(\f{\Pi(v+\om(t))}{|v+\om(t)|} \na_v \mu,\chi^c_N\<v\>^{2k} \na_v(\chi^c_Ng) \Big)_{L^2_{x,v}}, \\
P_{6,2} &:= T(t)^{-\f32}\Big(\f{\Pi(v+\om(t))}{|v+\om(t)|} \na_v \mu,\na_v\big(\chi^c_N\<v\>^{2k}\big) \chi^c_N g \Big)_{L^2_{x,v}}.
\end{aligned}
\]
For $P_{6,1}$, applying the same domain splitting as in the estimate of $P_{5,1}$ yields
\[
|P_{6,1}| \leq C_kT(t)^{-1}|V(t)|^{-1}\|\chi^c_Ng \|_{L^2_xH^1_v} \leq \f{\lam_0}8T(t)^{-\f32}\|\chi^c_N g\|^2_{Y_k}+C_{k}T(t)^{\f32}(T(t)^{-1}|V(t)|^{-1})^2,
\]
where we used $\<v\>^{-1}\lesssim T(t)^{\f12}|V(t)|^{-1}$ for $|v|>\f12|\om(t)|$. The same argument gives the corresponding bound for $P_{6,2}$. 
Therefore,
\begin{equation}\label{P6}
|P_6|\leq \f{\lam_0}4T(t)^{-\f32}\|\chi^c_N g\|^2_{Y_k}+C_kT(t)^{\f32}(T(t)^{-1}|T'(t)|+T(t)^{-\f12}|R(t)|+T(t)^{-1}|V(t)|^{-1})^2.
\end{equation}

Combining the estimates for $P_1$ through $P_6$ in \eqref{P14}, \eqref{P2}, \eqref{P3}, \eqref{P5} and \eqref{P6} yields the desired inequality and completes the proof.

\end{proof}

\subsection{Evolution of $g_1$ in the small region.} Using \eqref{coupg1}, we next derive an energy estimate for $g_1$ in the region  $|v|^2\ls N$, leading to the following proposition:
\begin{prop}\label{g1}
Let $g_1$ and $g_2$ be solutions of the coupled system \eqref{coupg1}--\eqref{coupg2}. Then for $k\geq 17$, the following estimate holds:
\begin{equation}\label{evolutiong1}
\begin{aligned}
& \f12\f d{dt}\|\chi_{4N}g_1\|^2_{X_k}+T(t)^{-\f32}(\f{\lam_0}2-C_k\|g_1\|_{X_{17}}-C_k\|g_2\|_{\cE_0})\|\chi_{4N}g_1\|^2_{Y_k}+T(t)^{-\f32}\mc_0k\|\chi_{4N}g_1\|^2_{X_{k-\f32}} \\
&\leq T(t)^{-\f32}\Big(\f{\mC_0}k\|\tilde\chi_{4N}g_1\|^2_{X_{k-\f32}}+C_k\|g_1\|_{X_{17}}\|\tilde\chi_{4N}g_1\|^2_{X_{k-\f32}} \\
&\qquad{} +C_k\|\chi^c_{4N}g_1\|^2_{X_{11}}+C_k\|\tilde\chi_{4N}g_1\|^2_{X_{k-\f52}}\Big) \\
&\quad{} + C_k\big(T(t)^{-1}|T'(t)|+T(t)^{-\f12}|R(t)|+T(t)^{-1}|V(t)|^{-1}\big) \\
&\qquad\mathbin{\times} (\|\chi_{4N}g_1\|^2_{X_{k}}+\|\tilde\chi_{4N}g_1\|^2_{X_k}+\|g_2\|_{\cD_0}\|\chi_{4N}g_1\|_{X_8}) \\
&\quad{} + C_kT(t)^{-1}|V(t)|^{-1}\|g_2\|^2_{\cD_0} \\
&\quad{} +C_kT(t)^{-\f32}\Big(\big(\|g_2\|_{\cE_0}+\|g_1\|_{X_{17}}\big)\|\tilde\chi_{4N}g_1\|^2_{X_{k-\f32}}+\|g_1\|_{X_{17}}\|g_2\|^2_{\cD_0}\Big) \\
&\quad{} + C_kT(t)^{\f32}\big(T(t)^{-1}|T'(t)|+T(t)^{-\f12}|R(t)|+T(t)^{-1}|V(t)|^{-1}\big)^2,
\end{aligned}
\end{equation}
where the constants $\lam_0,\mc_0,\mC_0$ are as in Lemma \ref{cutoffQcom}.
\end{prop}
\begin{proof}
For $|\al|=0,2$, the equation satisfied by $\pa^\al_x g_1$ reads 
\begin{equation}\label{paxg1}
\begin{aligned}
& \pa_t \pa^\al_xg_1(t,x,v)+T(t)^{\f12}v\cdot\na_x \pa^\al_xg_1+(-\f12 T(t)^{-1}T'(t)v+2T(t)^{-\f12}R(t))\cdot\na_v \pa^\al_xg_1(t,x,v) \\
&= T(t)^{-\f32}(L-A\chi_M)\pa^\al_xg_1+T(t)^{-\f32}\pa^\al_xQ(g_1,g_1)+\f32 T(t)^{-1} T'(t) \pa^\al_xg_1+T(t)^{-\f32} \\
&\quad{} \times \dv \Big(\f{\Pi(v+\om(t))}{|v+\om(t)|}\na_v (\pa^\al_xg_1+\mu^{\f12}\pa^\al_xg_2)\Big)+T(t)^{-\f32}\big(\pa^\al_xQ(\mu^\f12 g_2,g_1)+\pa^\al_xQ(g_1,\mu^{\f12}g_2)\big) \\
&\quad{} - (-\f12 T(t)^{-1}T'(t)v+2T(t)^{-\f12}R(t))\cdot(\na_v\pa^\al_xg_2-\f v 2 \pa^\al_xg_2)\mu^{\f12}+\f32T(t)^{-1}T'(t)\mu^{\f12}\pa^\al_xg_2+\pa^\al_x\cH(\mu).
\end{aligned}
\end{equation}
We multiply \eqref{paxg1} by $(\chi_{4N})^2\<v\>^{2(k-4|\al|)}\pa^\al_x g_1$, integrate over $(x,v)\in\T^3\times\R^3$, and sum over $|\al|=0,2$. The first two terms on the left-hand side yield
\[
\f 12 \f d{dt}\sum_{|\al|=0,2}\int_{\T^3\times\R^3}|\pa^\al_x g_1|^2(\chi_{4N})^2\<v\>^{2(k-4|\al|)}dvdx=\f12\f d{dt}\|\chi_{4N}g_1\|^2_{X_k},
\]
where we recall the energy functional $X_k$ from \eqref{functional}.

The remaining terms, including the third term on the left-hand side, are denoted by $I_1,\ldots,I_9$ and estimated below. For $I_1$ and $I_4$, the estimates of $P_1$ and $P_4$ in Proposition~\ref{g} give
\[
\begin{aligned}
I_1 &= -\frac{T'(t)}{4T(t)}\sum_{|\al|=0,2}\iint (\pa^\al_x g_1)^2\dv_v\big(v(\chi_{4N})^2\<v\>^{2(k-4|\al|)}\big)\,dvdx \\
&\quad{} + T(t)^{-\f12}\sum_{|\al|=0,2}\iint (\pa^\al_x g_1)^2R(t)\cdot\na_v\big((\chi_{4N})^2\<v\>^{2(k-4|\al|)}\big)\,dvdx.
\end{aligned}
\]
and
{
\[
\begin{aligned}
&I_4=\f32 T(t)^{-1}T'(t)\sum_{|\al|=0,2}\int_{\T^3\times\R^3}(\pa^\al_xg_1)^2(\chi_{4N})^2\<v\>^{2(k-4|\al|)}dvdx \\
&=\f32 T(t)^{-1}T'(t)\sum_{|\al|=0,2}\|\chi_{4N}\pa^\al_xg_1\|^2_{L^2_xL^2_{k-4|\al|}}.
\end{aligned}
\]
}
Thus, 
\begin{equation}\label{I14}
|I_1|+|I_4|\leq C_k\big(T(t)^{-1}|T'(t)|+T(t)^{-\f12}|R(t)|\big)(\|\chi_{4N}g_1\|^2_{X_{k}}+\|\tilde\chi_{4N}g_1\|^2_{X_k}).
\end{equation}

For $I_2$ and $I_3$, we write
\[
\begin{aligned}
I_2+I_3 &= T(t)^{-\f32}\sum_{|\al|=0,2}\Big(\int_{\T^3}\big((L-A\chi_M)(\pa^\al_xg_1),(\chi_{4N})^2\<v\>^{2(k-4|\al|)}\pa^\al_x g_1\big)_{L^2_v}dx \\
&\quad{} + \int_{\T^3}\big(Q(g_1,\pa^\al_xg_1),(\chi_{4N})^2\<v\>^{2(k-4|\al|)}\pa^\al_x g_1\big)_{L^2_v}dx\Big) \\
&\quad{} + T(t)^{-\f32}\sum_{|\al|=0,2}\sum_{\substack{\al_1+\al_2=\al\\|\al_1|\geq1}}\binom{\al}{\al_1}\int_{\T^3}\big(Q(\pa^{\al_1}_xg_1,\pa^{\al_2}_xg_1),(\chi_{4N})^2\<v\>^{2(k-4|\al|)}\pa^\al_x g_1\big)_{L^2_v}dx.
\end{aligned}
\]
By Lemma~\ref{QGgg} and \eqref{Qal1al2} from Lemma~\ref{paal1al2}, we obtain 
\begin{equation}\label{I23}
\begin{aligned}
I_2+I_3 &\leq T(t)^{-\f32}\Big(-\lam_0 \|\chi_{4N}g_1\|^2_{Y_k}-\mc_0k\|\chi_{4N}g_1\|^2_{X_{k-\f32}}+C_k\|g_1\|_{X_{17}}\|\chi_{4N}g_1\|^2_{Y_k} \\
&\quad{} + \f{\mC_0}k\|\tilde\chi_{4N}g_1\|^2_{X_{k-\f32}}+C_k\|g_1\|_{X_{17}}\|\tilde\chi_{4N}g_1\|^2_{X_{k-\f32}}+C_k\|\chi^c_{4N}g_1\|^2_{X_{11}} \\
&\quad{} +C_k\|\tilde\chi_{4N}g_1\|^2_{X_{k-\f52}}\Big).
\end{aligned}
\end{equation}

For $I_5$, we write
\[
\begin{aligned}
I_5 &= -T(t)^{-\f32}\sum_{|\al|=0,2}\int_{\T^3\times\R^3}\Big(\f{\Pi(v+\om(t))}{|v+\om(t)|}\cdot \na_v (\pa^\al_xg_1+\mu^{\f12}\pa^\al_xg_2)\Big)\na_v((\chi_{4N})^2\<v\>^{2(k-4|\al|)}\pa^\al_xg_1)dvdx \\
&= -T(t)^{-\f32}\sum_{|\al|=0,2}\int_{\T^3\times\R^3}\f{\Pi(v+\om(t))}{|v+\om(t)|}\na_v\pa^\al_xg_1\na_v((\chi_{4N})^2\<v\>^{2(k-4|\al|)}\pa^\al_xg_1)dvdx \\
& {-T(t)^{-\f32}\sum_{|\al|=0,2}\int_{\T^3\times\R^3}\f{\Pi(v+\om(t))}{|v+\om(t)|}\na_v(\mu^{\f12}\pa^\al_xg_2)\na_v((\chi_{4N})^2\<v\>^{2(k-4|\al|)}\pa^\al_xg_1)dvdx} \\
&=: I_{5,1}+I_{5,2}.
\end{aligned}
\] 
For $I_{5,1}$, we argue as in the estimate of $P_5$, and apply the Cauchy--Schwarz inequality to obtain
\[
\begin{aligned}
I_{5,1} &\leq -\f34T(t)^{-\f32}\sum_{|\al|=0,2}\int_{\T^3\times\R^3}\f{\Pi(v+\om(t))}{|v+\om(t)|}\na_v\pa^\al_xg_1\cdot \na_v\pa^\al_xg_1(\chi_{4N})^2\<v\>^{2(k-4|\al|)}dvdx \\
&\quad{} + C_kT(t)^{-\f32}\sum_{|\al|=0,2}\int_{\T^3\times\R^3}\f1{|v+\om(t)|}(\pa^\al_xg_1)^2\Big((\chi_{4N})^2\<v\>^{2(k-4|\al|-1)}+(\na_v\chi_{4N})^2\<v\>^{2(k-4|\al|)}\Big)dvdx \\
&\leq -\f34T(t)^{-\f32}\sum_{|\al|=0,2}\int_{\T^3\times\R^3}\f{\Pi(v+\om(t))}{|v+\om(t)|}\na_v\pa^\al_xg_1\cdot \na_v\pa^\al_xg_1(\chi_{4N})^2\<v\>^{2(k-4|\al|)}dvdx \\
&\quad{} + C_kT(t)^{-1}|V(t)|^{-1}\big(\|\chi_{4N}g_1\|^2_{X_{k}}+\|\tilde\chi_{4N}g_1\|^2_{X_{k}}\big),
\end{aligned}
\]
where we used $4N<|v|^2<8N$ and $|v|<\sqrt{8N}<\f12|\om(t)|$ on $\operatorname{supp}\nabla\chi_{4N}$(see \eqref{VTcondition0}).

For $I_{5,2}$, since $|v|\leq\sqrt{8N}<\f12|\om(t)|$ on $\operatorname{supp}\chi_{4N}$, the exponential decay of $\mu$ and the Cauchy-Schwarz inequality give
	\[
\begin{aligned}
|I_{5,2}| &\leq \f14T(t)^{-\f32}\sum_{|\al|=0,2}\int_{\T^3\times\R^3}\f{\Pi(v+\om(t))}{|v+\om(t)|}\na_v\pa^\al_xg_1\cdot \na_v\pa^\al_xg_1(\chi_{4N})^2\<v\>^{2(k-4|\al|)}dvdx \\
&\quad{} + C_kT(t)^{-1}|V(t)|^{-1}\big(\|\chi_{4N}g_1\|^2_{X_{k}}+\|g_2\|^2_{\cD_0}\big).
\end{aligned}
\]
Hence,
\begin{equation}\label{I5}
\begin{aligned}
I_5 &\leq -\f12T(t)^{-\f32}\sum_{|\al|=0,2}\int_{\T^3\times\R^3}\f{\Pi(v+\om(t))}{|v+\om(t)|}\na_v\pa^\al_xg_1\cdot \na_v\pa^\al_xg_1(\chi_{4N})^2\<v\>^{2(k-4|\al|)}dvdx \\
&\quad{} + C_kT(t)^{-1}|V(t)|^{-1}\big(\|\chi_{4N}g_1\|^2_{X_{k}}+\|\tilde\chi_{4N}g_1\|^2_{X_{k}}+\|g_2\|^2_{\cD_0}\big).
\end{aligned}
\end{equation}

For $I_6$, we write
\[
\begin{aligned}
I_6 &= T(t)^{-\f32}\sum_{|\al|=0,2}\sum_{\al_1+\al_2=\al}\binom{\al}{\al_1}\Big(\int_{\T^3}\big(Q(\mu^{\f12}\pa^{\al_1}_xg_2,\pa^{\al_2}_xg_1),(\chi_{4N})^2\<v\>^{2(k-4|\al|)}\pa^\al_x g_1\big)_{L^2_v}dx \\
&\quad{} + \int_{\T^3}\big(Q(\pa^{\al_1}_xg_1,\mu^{\f12}\pa^{\al_2}_xg_2),(\chi_{4N})^2\<v\>^{2(k-4|\al|)}\pa^\al_x g_1\big)_{L^2_v}dx\Big).
\end{aligned}
\] 
By \eqref{cutoffgfh}, \eqref{Qal1al2} and the exponential decay of $\mu$, we obtain
\begin{equation}\label{I6}
\begin{aligned}
|I_6| &\leq C_kT(t)^{-\f32}\Big(\|g_2\|_{\cE_0}\big(\|\chi_{4N}g_1\|^2_{Y_k}+\|\tilde\chi_{4N}g_1\|^2_{X_{k-\f32}}\big) \\
&\quad{} +\|g_1\|_{X_{17}}\|g_2\|_{\cD_0}\big(\|\chi_{4N}g_1\|_{Y_k}+\|\tilde\chi_{4N}g_1\|_{X_{k-\f32}}\big)\Big) \\
&\leq C_kT(t)^{-\f32}\Big(\big(\|g_2\|_{\cE_0}+\|g_1\|_{X_{17}}\big)\big(\|\chi_{4N}g_1\|^2_{Y_k}+\|\tilde\chi_{4N}g_1\|^2_{X_{k-\f32}}\big) \\
&\quad{} +\|g_1\|_{X_{17}}\|g_2\|^2_{\cD_0}\Big).
\end{aligned}
\end{equation}

For $I_7$ and $I_8$, we write 
\[
\begin{aligned}
I_7+I_8 &= \sum_{|\al|=0,2}\Big((\f12 T(t)^{-1}T'(t)v-2T(t)^{-\f12}R(t))\cdot(\na_v\pa^\al_xg_2-\f v 2 \pa^\al_xg_2)\mu^{\f12} \\
&\quad{} + \f32T(t)^{-1}T'(t)\mu^{\f12}\pa^\al_xg_2,(\chi_{4N})^2\<v\>^{2(k-4|\al|)}\pa^\al_x g_1\Big)_{L^2_{x,v}}.
\end{aligned}
\]
A direct estimate gives
\begin{equation}\label{I78}
|I_7+I_8| \leq C_k(T(t)^{-1}|T'(t)|+T(t)^{-\f12}|R(t)|)\|g_2\|_{\cD_0}\|\chi_{4N}g_1\|_{X_8}.
\end{equation}

For $I_9$, since $\pa^\al_x\cH(\mu)=0$ when $|\al|=2$, we write
\[
\begin{aligned}
I_9 &= \int_{\T^3\times\R^3}\Big(\big(\f12 T(t)^{-1}T'(t)v-2T(t)^{-\f12}R(t)\big)\cdot\na_v \mu+\f32 T(t)^{-1} T'(t) \mu\Big)g_1(\chi_{4N})^2\<v\>^{2k}dvdx \\
&\quad{} + \int_{\T^3\times\R^3}T(t)^{-\f32}\dv \Big(\f{\Pi(v+\om(t))}{|v+\om(t)|}\cdot \na_v \mu\Big)g_1(\chi_{4N})^2\<v\>^{2k}dvdx:=I_{9,1}+I_{9,2}.
\end{aligned}
\]
By the exponential decay of $\mu$, we obtain
\[
\begin{aligned}
I_{9,1} &\leq C_k\big(T(t)^{-1}|T'(t)|+T(t)^{-\f12}|R(t)|\big)\|\chi_{4N}g_1\|_{Y_k} \\
&\leq T(t)^{-\f32}\f{\lam_0}4\|\chi_{4N}g_1\|^2_{Y_k}+C_kT(t)^{\f32}\big(T(t)^{-1}|T'(t)|+T(t)^{-\f12}|R(t)|\big)^2.
\end{aligned}
\]
For $I_{9,2}$, since $|v|\leq\sqrt{8N}<\f12|\om(t)|$ on $\operatorname{supp}\chi_{4N}$ and 
\[
\left|T(t)^{-\f32}\dv \Big(\f{\Pi(v+\om(t))}{|v+\om(t)|}\cdot \na_v \mu\Big)\right|\ls T(t)^{-\f32}\Big(\f1{|v+\om(t)|}+\f1{|v+\om(t)|^2}\Big)\mu^{\f12}\ls T(t)^{-1}|V(t)|^{-1}\mu^{\f12}.
\]
This yields
\[
|I_{9,2}| \leq C_k T(t)^{-1}|V(t)|^{-1}\|\chi_{4N}g_1\|_{Y_k} \leq T(t)^{-\f32}\f{\lam_0}4\|\chi_{4N}g_1\|^2_{Y_k}+C_kT(t)^{\f32}(T(t)^{-1}|V(t)|^{-1})^2.
\]
Therefore,
\begin{equation}\label{I9}
|I_9|\leq T(t)^{-\f32}\f{\lam_0}2\|\chi_{4N}g_1\|^2_{Y_k}+C_kT(t)^{\f32}\big(T(t)^{-1}|T'(t)|+T(t)^{-\f12}|R(t)|+T(t)^{-1}|V(t)|^{-1}\big)^2.
\end{equation}

Combining the estimates for $I_1$ through $I_9$ in \eqref{I14}, \eqref{I23}, \eqref{I5}, \eqref{I6}, \eqref{I78} and \eqref{I9} yields \eqref{evolutiong1} and completes the proof.

\end{proof}

\subsection{Combining $g_1$ and $g$.} Combining the energy estimates for $g_1$ and $g$, we obtain
\begin{prop}\label{gg1}
	Let $g$ be a solution of \eqref{equg2} and $g_1,g_2$ be solutions of the coupled system \eqref{coupg1}--\eqref{coupg2}. Fix an integer $k_0>\max\{17,\sqrt{3200\mC_0/\mc_0}\}$. Then for any $k\geq k_0$ and $N$ sufficiently large (depending only on $k$), the following estimate holds: 
	\begin{equation}\label{evolutiongand g}
\begin{aligned}
& \f12\f d{dt}(\|\chi^c_{N}g\|^2_{X_k}+\|\chi_{4N}g_1\|^2_{X_k})+T(t)^{-\f32}(\f{\lam_0}8-C_k\|g_1\|_{X_{17}}-C_k\|g_2\|_{\cE_0})\|g_1\|^2_{Y_k} \\
&\quad{} + T(t)^{-\f32}\f{\mc_0}2k(\|\chi^c_{N}g\|^2_{X_{k-\f32}}+\|\chi_{4N}g_1\|^2_{X_{k-\f32}}) \\
&\leq C_kT(t)^{-\f32}\big(\|g_2\|_{\cE_0}+\|g_1\|_{X_{17}}+\f1N\big)\times \|g_2\|^2_{\cD_0} \\
&\quad{} +C_k\big(T(t)^{-1}|T'(t)|+T(t)^{-\f12}|R(t)|+T(t)^{-1}|V(t)|^{-1}\big) \\
&\qquad\mathbin{\times} (\|g_1\|^2_{X_k}+\|g_2\|^2_{\cE_0}+\|g_2\|_{\cD_0}\|g_1\|_{X_k}) \\
&\quad{} + C_kT(t)^{\f32}\big(T(t)^{-1}|T'(t)|+T(t)^{-\f12}|R(t)|+T(t)^{-1}|V(t)|^{-1}\big)^2.
\end{aligned}
\end{equation}
\end{prop}
\begin{proof}
Adding \eqref{evolutiong} and \eqref{evolutiong1} yields 
	\begin{equation}\label{gg12}
\begin{aligned}
& \f12\f d{dt}(\|\chi^c_{N}g\|^2_{X_k}+\|\chi_{4N}g_1\|^2_{X_k})+T(t)^{-\f32}(\f{\lam_0}2-C_k\|g\|_{X_{15}})\|\chi^c_{N}g\|^2_{Y_k} \\
&\quad{} + ~T(t)^{-\f32}(\f{\lam_0}2-C_k\|g_1\|_{X_{17}}-C_k\|g_2\|_{\cE_0})\|\chi_{4N}g_1\|^2_{Y_k} \\
&\quad{} +T(t)^{-\f32}\mc_0k(\|\chi^c_{N}g\|^2_{X_{k-\f32}}+\|\chi_{4N}g_1\|^2_{X_{k-\f32}}) \\
&\leq L_1+L_2+L_3+C_kT(t)^{\f32}\big(T(t)^{-1}|T'(t)|+T(t)^{-\f12}|R(t)|+T(t)^{-1}|V(t)|^{-1}\big)^2
\end{aligned}
\end{equation}
with 
\[
\begin{aligned}
L_1 &= T(t)^{-\f32}\Big(\f {\mC_0}k\|\tilde{\chi}_Ng\|^2_{X_{k-\f32}}+C_k\|\tilde{\chi}_N g\|^2_{X_{k-\f52}}+C_k\|\chi^c_N g\|^2_{X_8}+\frac{C_k}{N}\|g\|_{X_{11}}^2\Big) \\
&\quad{} + ~T(t)^{-\f32}\Big(\f{\mC_0}k\|\tilde\chi_{4N}g_1\|^2_{X_{k-\f32}}+C_k\|g_1\|_{X_{17}}\|\tilde\chi_{4N}g_1\|^2_{X_{k-\f32}} \\
&\quad{} +C_k\|\chi^c_{4N}g_1\|^2_{X_{11}}+C_k\|\tilde\chi_{4N}g_1\|^2_{X_{k-\f52}}\Big), \\
L_2 &= C_kT(t)^{-\f32}\|g\|_{X_{15}}\|\tilde\chi_N g\|^2_{X_{k-\f32}} \\
&\quad{} +C_kT(t)^{-\f32}\Big(\big(\|g_2\|_{\cE_0}+\|g_1\|_{X_{17}}\big)\|\tilde\chi_{4N}g_1\|^2_{X_{k-\f32}}+\|g_1\|_{X_{17}}\|g_2\|^2_{\cD_0}\Big), \\
L_3 &= C_kT(t)^{-1}|V(t)|^{-1}\|g_2\|_{\cD_0}^2+C_k\big(T(t)^{-1}|T'(t)|+T(t)^{-\f12}|R(t)|+T(t)^{-1}|V(t)|^{-1}\big) \\
&\quad{} \times \big(\|\chi^c_{N}g\|^2_{X_{k}}+\|\tilde\chi_{N}g\|^2_{X_{k}}+\|\chi_{4N}g_1\|^2_{X_{k}}+\|\tilde\chi_{4N}g_1\|^2_{X_k}+\|g_2\|_{\cD_0}\|\chi_{4N}g_1\|_{X_8}\big).
\end{aligned}
\]

Since $g=g_1+\mu^{\f12}g_2$, we first note that
\begin{equation}\label{a0}
\|g\|_{X_{15}}\leq \|g_1\|_{X_{15}}+\|\mu^{\f12}g_2\|_{X_{15}}\lesssim \|g_1\|_{X_{17}}+\|g_2\|_{\cE_0}.
\end{equation}
Moreover, for any $l\in\R$,
\begin{equation}\label{a1}
\begin{aligned}
\|\chi^c_{N}g\|^2_{X_{l}}+\|\chi_{4N}g_1\|^2_{X_{l}} &\geq \f12(\|\chi^c_{N}g_1\|^2_{X_{l}}+ \|\chi_{4N}g_1\|^2_{X_{l}})-\|\chi^c_N\mu^{\f12}g_2\|^2_{X_{l}} \\
&\geq \f12\|g_1\|^2_{X_{l}}-\f{C_l}N\min\{\|g_2\|^2_{\cD_0},\|g_2\|^2_{\cE_0}\},
\end{aligned}
\end{equation}
where we used $|v|^2\geq N$ on $\operatorname{supp}\chi^c_{N}$. The analogous bound holds with $X_l$ replaced by $Y_l$:
\begin{equation}\label{b1}
\begin{aligned}
&\|\chi^c_{N}g\|^2_{Y_{l}}+\|\chi_{4N}g_1\|^2_{Y_{l}} \geq \f12(\|\chi^c_{N}g_1\|^2_{Y_{l}}+ \|\chi_{4N}g_1\|^2_{Y_{l}})-\|\chi^c_N\mu^{\f12}g_2\|^2_{Y_{l}} \\
&\geq (\f12-\f{C_l}N)\|g_1\|^2_{Y_{l}}-\f{C_l}N\|g_2\|^2_{\cD_0} \geq \f14\|g_1\|^2_{Y_{l}}-\f{C_l}N\|g_2\|^2_{\cD_0}
\end{aligned}
\end{equation}
for sufficiently large \(N\) depending on \(l\).
To retain the cutoff cross terms, we apply integration by parts to
\(u=\langle v\rangle^l f\):
\[
\begin{aligned}
\|\chi_N^cu\|_{H^1}^2+\|\chi_{4N}u\|_{H^1}^2 &= \int\big((\chi_N^c)^2+\chi_{4N}^2\big) \big(|u|^2+|\nabla_vu|^2\big)\,dv \\
&\quad{} - \int\big(\chi_N^c\Delta_v\chi_N^c +\chi_{4N}\Delta_v\chi_{4N}\big)|u|^2\,dv \\
&\geq \|u\|_{H^1}^2-\frac C N\|u\|_{L^2}^2 .
\end{aligned}
\]
Here we have used \((\chi_N^c)^2+\chi_{4N}^2\geq1\) and
\(|\Delta_v\chi_N^c|+|\Delta_v\chi_{4N}|\leq C/N\).
The angular derivatives commute with these radial cutoffs.
Applying this identity to each spatial derivative in \(Y_l\) and
using the Gaussian tail to estimate
\(\|\chi_N^c\sqrt\mu g_2\|_{Y_l}^2\leq C_lN^{-1}\|g_2\|_{\mathcal D_0}^2\)
proves \eqref{b1}.

Similarly, since $|\tilde\chi|\leq 10$, we have
\begin{equation}\label{a2}
\begin{aligned}
&\|\chi^c_{N}g\|^2_{X_{l}},\|\tilde\chi_{N}g\|^2_{X_{l}},\|\chi_{4N}g_1\|^2_{X_{l}},\|\tilde\chi_{4N}g_1\|^2_{X_l} \\
&\leq 200\|g_1\|^2_{X_l}+\f{C_l}N\min\{\|g_2\|^2_{\cD_0},\|g_2\|^2_{\cE_0}\}.
\end{aligned}
\end{equation}

For $L_1$, \eqref{a1} and \eqref{a2} give 
\[
\begin{aligned}
& \f{\mC_0}k(\|\tilde{\chi}_Ng\|^2_{X_{k-\f32}}+\|\tilde{\chi}_{4N}g_1\|^2_{X_{k-\f32}})\leq \f{400\mC_0}{k}\|g_1\|^2_{X_{k-\f32}}+\f{C_k}N\|g_2\|^2_{\cD_0} \\
&\leq \f{800\mC_0}k\big(\|\chi^c_{N}g\|^2_{X_{k-\f32}}+\|\chi_{4N}g_1\|^2_{X_{k-\f32}}\big)+\f{C_k}N\|g_2\|^2_{\cD_0} \\
&\leq \f{\mc_0k}4\big(\|\chi^c_{N}g\|^2_{X_{k-\f32}}+\|\chi_{4N}g_1\|^2_{X_{k-\f32}}\big)+\f{C_k}N\|g_2\|^2_{\cD_0},
\end{aligned}
\]
when $k\geq k_0$. Moreover,
\[
\begin{aligned}
& C_k\big(\|\tilde{\chi}_N g\|^2_{X_{k-\f52}}+\|\chi^c_N g\|^2_{X_8}+\|\chi^c_{4N}g_1\|^2_{X_{11}}+\|\tilde\chi_{4N}g_1\|^2_{X_{k-\f52}}\big)+\frac{C_k}{N}\|g\|_{X_{11}}^2 \\
&\leq \f{C_k}N\big(\|\chi^c_{N}g\|^2_{X_{k-\f32}}+\|\chi_{4N}g_1\|^2_{X_{k-\f32}}\big)+\f{C_k}N\|g_2\|^2_{\cD_0} \\
&\leq \f{\mc_0k}4\big(\|\chi^c_{N}g\|^2_{X_{k-\f32}}+\|\chi_{4N}g_1\|^2_{X_{k-\f32}}\big)+\f{C_k}N\|g_2\|^2_{\cD_0},
\end{aligned}
\]
for sufficiently large \(N\). Here the additional nonlocal remainder
is included using
\[
\|g\|_{X_{11}}^2\leq C_k\|g_1\|_{X_{k-3/2}}^2+C_k\|g_2\|_{\mathcal D_0}^2
\]
and \eqref{a1}. It is therefore absorbed together with the other
terms having coefficient \(C_k/N\). Also,
\[
\begin{aligned}
&C_k\|g_1\|_{X_{17}}\|\tilde\chi_{4N}g_1\|^2_{X_{k-\f32}}\leq C_k\|g_1\|_{X_{17}}\big(\|\chi^c_{N}g\|^2_{X_{k-\f32}}+\|\chi_{4N}g_1\|^2_{X_{k-\f32}}\big) \\
&\quad{} +C_k\|g_1\|_{X_{17}}\|g_2\|^2_{\cD_0}.
\end{aligned}
\]
Using also $\|\cdot\|_{X_{k-\f32}}\leq \|\cdot\|_{Y_k}$, we obtain
\begin{equation}\label{L1}
\begin{aligned}
T(t)^{\f32}L_1 &\leq \f{\mc_0k}2\big(\|\chi^c_{N}g\|^2_{X_{k-\f32}}+\|\chi_{4N}g_1\|^2_{X_{k-\f32}}\big)+\f{C_k}N\|g_2\|^2_{\cD_0} \\
&\quad{} + C_k\|g_1\|_{X_{17}}\big(\|\chi^c_{N}g\|^2_{Y_{k}}+\|\chi_{4N}g_1\|^2_{Y_{k}}\big)+C_k\|g_1\|_{X_{17}}\|g_2\|^2_{\cD_0}.
\end{aligned}
\end{equation}

For $L_2$, \eqref{a1} gives
\begin{equation}\label{L2}
\begin{aligned}
L_2 &\leq C_kT(t)^{-\f32}(\|g_1\|_{X_{17}}+\|g_2\|_{\cE_0})\|g_1\|^2_{X_{k-\f32}}+C_kT(t)^{-\f32}\|g_1\|_{X_{17}}\|g_2\|^2_{\cD_0} \\
&\leq C_kT(t)^{-\f32}(\|g_1\|_{X_{17}}+\|g_2\|_{\cE_0})(\|\chi^c_{N}g\|^2_{Y_{k}}+\|\chi_{4N}g_1\|^2_{Y_{k}}) \\
&\quad{} +C_kT(t)^{-\f32}(\|g_1\|_{X_{17}}+\|g_2\|_{\cE_0})\|g_2\|^2_{\cD_0}.
\end{aligned}
\end{equation}

For $L_3$, \eqref{VTcondition0} and \eqref{a2} yield
\begin{equation}\label{L3}
\begin{aligned}
L_3 &\leq T(t)^{-\f32}\f{C_k}N\|g_2\|_{\cD_0}^2+C_k\big(T(t)^{-1}|T'(t)|+T(t)^{-\f12}|R(t)|+T(t)^{-1}|V(t)|^{-1}\big) \\
&\quad{} \times (\|g_1\|^2_{X_k}+\|g_2\|^2_{\cE_0}+\|g_2\|_{\cD_0}\|g_1\|_{X_k}).
\end{aligned}
\end{equation}

The same product rule, now used for an upper bound, gives
\[
\|\chi_N^c g\|_{Y_k}^2+\|\chi_{4N}g_1\|_{Y_k}^2 \leq C_k\|g_1\|_{Y_k}^2+\frac{C_k}{N}\|g_2\|_{\mathcal D_0}^2 .
\]
We use this upper bound for the terms with nonlinear coefficients,
and \eqref{b1} for the positive linear dissipation.
Inserting \eqref{a0},\eqref{b1} and the estimates \eqref{L1}, \eqref{L2}, \eqref{L3} into \eqref{gg12} yields the desired result. The factor \(1/4\) in \eqref{b1} reduces the linear gradient coefficient to \(\lambda_0/8\), as stated.
\end{proof}

\subsection{Evolution of $g_2$.} We now use \eqref{coupg2} to derive energy estimates for $g_2$. We begin with the microscopic estimates.
\begin{lem}\label{micro}
Let $g_1$ and $g_2$ be solutions of the coupled system \eqref{coupg1} and \eqref{coupg2}. Then, for $k\geq 17$, we have
\begin{equation}\label{cEkcDk}
\f 12\f d{dt}\|g_2\|^2_{\cE_k}+T(t)^{-\f32}(\f{\lam_0} 2-C_k\|g_2\|_{\cE_{17}})\|g_2\|^2_{\mathcal{D}_k}\leq C_{k}T(t)^{-\f32}(\|g_2\|^2_{\cD_0}+\|g_1\|^2_{Y_k}).
\end{equation}
Moreover,
\begin{equation}\label{cE0cD0}
\begin{aligned}
& \f 12\f d{dt}\|g_2\|^2_{\cE_0}+T(t)^{-\f32}(\lam_0-C_0\|g_2\|_{\cE_0})\|(\mathbf{I-P})g_2\|^2_{\mathcal{D}_0} \\
&\leq C_kT(t)^{-\f32}\|g_2\|_{\cD_0}\|g_1\|_{Y_0}+C_kT(t)^{-\f32}\|g_2\|_{\cE_0}\|\mP g_2\|^2_{H^2_{x}L^2_v},
\end{aligned}
\end{equation} 
where the projection $\mP$ is defined in \eqref{mP}.
\end{lem}
\begin{proof}
For any $|\al|=0,2$, the equation for $\pa^\al_xg_2$ is 
\begin{equation}\label{paxg2}
\begin{aligned}
& \pa_t \pa^\al_xg_2(t,x,v)+T(t)^{\f12}v\cdot\na_x \pa^\al_xg_2 \\
&= T(t)^{-\f32}\mathcal{L}\pa^\al_xg_2+T(t)^{-\f32}\mu^{-\f12}A\chi_M(v)\pa^\al_xg_1+T(t)^{-\f32}\pa^\al_x\Ga(g_2,g_2).
\end{aligned}
\end{equation}
We first prove the estimate \eqref{cEkcDk}. We use $\<v\>^{2k}\pa^\al_x g_2,k\geq17$ as a multiplier in \eqref{paxg2}, integrate in $(x,v)\in\T^3\times\R^3$ and sum over $|\al|=0,2$. Then the first two terms on the left-hand side give
\[
\f 12 \f d{dt}\sum_{|\al|=0,2}\int_{\T^3\times\R^3}|\pa^\al_x g_2|^2\<v\>^{2k}dvdx=\f12\f d{dt}\|g_2\|^2_{\cE_k}.
\]

We denote the remaining terms by $J_{1,k}$ through $J_{3,k}$ and estimate them separately. For $J_{1,k}$, \eqref{cLf} in Lemma \ref{linearizedoperator} gives
\begin{equation}\label{J1}
J_{1,k}=T(t)^{-\f32}\sum_{|\al|=0,2}\int_{\T^3}\big(\cL\pa^\al_xg_2,\pa^\al_xg_2\<v\>^{2k}\big)_{L^2_v}dx\leq T(t)^{-\f32}\Big( -\lam_0\|g_2\|^2_{\cD_k}+C_k\|g_2\|^2_{\cD_0}\Big).
\end{equation}

For $J_{2,k}$, we write 
\[
J_{2,k}=T(t)^{-\f32}\sum_{|\al|=0,2}\int_{\T^3\times\R^3} \mu^{-\f12}A\chi_M(v)\pa^\al_xg_1\<v\>^{2k}\pa^\al_xg_2dvdx.
\]
Since $A\chi_M$ has compact support and $A,M$ depend only on $k$, it follows that
\begin{equation}\label{J2}
J_{2,k}\leq \varepsilon T(t)^{-\f32}\|g_2\|^2_{\cD_0}+ C_{\vep,k}T(t)^{-\f32}\|g_1\|^2_{Y_k}
\end{equation}
for any $\varepsilon>0$. We henceforth choose $\vep=\lam_0/2$.

For $J_{3,k}$, we write
\[
J_{3,k}=T(t)^{-\f32}\sum_{|\al|=0,2}\sum_{\al_1+\al_2=\al}\binom{\al}{\al_1}\int_{\T^3}\big(\Ga(\pa^{\al_1}_xg_2,\pa^{\al_2}_xg_2),\<v\>^{2k}\pa^\al_xg_2\big)_{L^2_v}dx.
\]
By \eqref{Gaal1al2} in Lemma \ref{paal1al2}, we obtain
\begin{equation}\label{J3}
J_{3,k}\leq C_kT(t)^{-\f32} \|g_2\|_{\cE_{17}}\|g_2\|^2_{\cD_k}.
\end{equation}

Combining the estimates \eqref{J1}, \eqref{J2} and \eqref{J3}, we obtain the estimate \eqref{cEkcDk}.

\medskip

A similar argument yields \eqref{cE0cD0}; the only difference is the estimate for \(J_{1,0}\), where we appeal to \eqref{spectralgap} instead of \eqref{cLf}. We omit the details and record the resulting bounds for each term:
\[
\begin{aligned}
&J_{1,0}\leq -T(t)^{-\f32}\lam_0\|(\mathbf{I-P})g_2\|^2_{\cD_0}, \\
&J_{2,0}\ls T(t)^{-\f32}\|g_2\|_{\cD_0}\|g_1\|_{Y_0}, \qquad J_{3,0}\leq C_0T(t)^{-\f32} \|g_2\|_{\cE_0}\|g_2\|^2_{\cD_0}.
\end{aligned}
\]
Decomposing \(g_2=\mathbf P g_2+(\mathbf I-\mathbf P)g_2\) and using \(\|\mP g_2\|_{\cD_0}\ls \|\mP g_2\|_{H^2_{x}L^2_v}\), we obtain \eqref{cE0cD0}, and the proof of the lemma is complete.
\end{proof}

Estimate \eqref{cE0cD0} yields dissipation for the microscopic component. To control the macroscopic part $\mP g_2$, we derive the macroscopic equations.
\begin{lem}
	We rewrite \eqref{coupg2} as
	\begin{equation}\label{coupg22}
\pa_t g_2(t,x,v)+T(t)^{\f12}v\cdot\na_x g_2=T(t)^{-\f32}\mathcal{L}g_2+r(g_1,g_2)
\end{equation}
	with 
	\begin{equation}\label{rg2}
r(g_1,g_2)=T(t)^{-\f32}\mu^{-\f12}A\chi_M(v)g_1+T(t)^{-\f32}\Ga(g_2,g_2),
\end{equation}
	and write the macroscopic part as $\mP g_2=\big(a^{g_2}+b^{g_2}\cdot v+c^{g_2}(|v|^2-3)\big)\mu^{\f12}$ with
	\begin{equation}\label{afbfcfa}
a^{g_2}(t,x)=\int_{\R^3}g_2\mu^{\f12}(v)dv,~b^{g_2}(t,x)=\int_{\R^3}g_2v\mu^{\f12}(v)dv,~\mbox{and}~c^{g_2}(t,x)=\f16\int_{\R^3}g_2(|v|^2-3)\mu^{\f12}(v)dv.
\end{equation}
	Taking velocity moments yields
\begin{equation}\label{macroeq}
\left\{ 
\begin{aligned}
&\pa_t a^{g_2}+T(t)^{\f12}\na_x\cdot b^{g_2}=\sfT_{11}; \\
&\pa_t b^{g_2}+T(t)^{\f12}(\na_x a^{g_2}+2\na_xc^{g_2})=\sfT_{21}; \\
&\pa_t c^{g_2}+\f13T(t)^{\f12}\na_x\cdot b^{g_2}=\sfT_{31}; \\
&\pa_t c^{g_2}+T(t)^{\f12}\pa_i b^{g_2}_i=(\sfT_{41})_{ii}+\pa_t(\sfT_{42})_{ii}; \\
&T(t)^{\f12}(\pa_{j}b^{g_2}_i+\pa_i b^{g_2}_j)=(\sfT_{41})_{ij}+\pa_t (\sfT_{42})_{ij},~~i\neq j; \\
&T(t)^{\f12}\na_x c^{g_2}=\sfT_{51}+\pa_t\sfT_{52},
\end{aligned}\right.
\end{equation}
where $\sfT_{k1}$ and $\sfT_{l2}$, with $1\le k\le 5$ and $l=4,5$, are defined as follows: for $i,j=1,2,3$, $i\neq j$,
\begin{equation}\label{Tij1}
\left\{ 
\begin{aligned}
&\sfT_{11}=(r(g_1,g_2),\mu^{\f12})_{L^2_v},~~\sfT_{31}=\f16\big(-T(t)^{\f12}v\cdot\na_x(\mathbf{I-P})g_2+r(g_1,g_2),\mu^{\f12}(|v|^2-3)\big)_{L^2_v}; \\
&\sfT_{21}=\big(-T(t)^{\f12}v\cdot\na_x(\mathbf{I-P})g_2+r(g_1,g_2),\mu^{\f12}v\big)_{L^2_v},~~\sfT_{52}=-\big((\mathbf{I-P})g_2,\mu^{\f12}\f{v(|v|^2-5)}{10}\big)_{L^2_v}; \\
&\sfT_{51}=\big(-T(t)^{\f12}v\cdot\na_x(\mathbf{I-P})g_2+T(t)^{-\f32}\cL(\mathbf{I-P})g_2+r(g_1,g_2),\mu^{\f12}\f{v(|v|^2-5)}{10}\big)_{L^2_v}.
\end{aligned}\right.
\end{equation}
\begin{equation}\label{Tij2}
\left\{ 
\begin{aligned}
&(\sfT_{41})_{ij}=\big(-(T(t)^{\f12}v\cdot\na_x)(\mathbf{I-P})g_2+T(t)^{-\f32}\cL(\mathbf{I-P})g_2+r(g_1,g_2),\mu^{\f12}v_iv_j\big)_{L^2_v}; \\
&(\sfT_{41})_{ii}=\big(-(T(t)^{\f12}v\cdot\na_x)(\mathbf{I-P})g_2+T(t)^{-\f32}\cL(\mathbf{I-P})g_2+r(g_1,g_2),\mu^{\f12}\f{v^2_i-1}2\big)_{L^2_v}; \\
&(\sfT_{42})_{ij}=-\big((\mathbf{I-P})g_2,\mu^{\f12}v_iv_j\big)_{L^2_v},\quad (\sfT_{42})_{ii}=\big(-(\mathbf{I-P})g_2,\mu^{\f12}\f{v^2_i-1}2\big)_{L^2_v}.
\end{aligned}\right.
\end{equation}
\end{lem}
\begin{proof}
We begin with the decomposition 
	\[
g_2=\mP g_2+(\mathbf{I-P})g_2=[a^{g_2}+b^{g_2}\cdot v+c^{g_2}(|v|^2-3)]\mu^{\f12}+(\mathbf{I-P})g_2.
\]
	We rewrite \eqref{coupg22} as
	\begin{equation}\label{macro1}
\begin{aligned}
& (\pa_ta^{g_2}+\pa_t b^{g_2}\cdot v+\pa_t c^{g_2}(|v|^2-3))\mu^{\f12}+T(t)^{\f12}v\cdot\na_x(a^{g_2}+b^{g_2}\cdot v+c^{g_2}(|v|^2-3))\mu^{\f12} \\
&= -(\pa_t+T(t)^{\f12}v\cdot\na_x)(\mathbf{I-P})g_2+T(t)^{-\f32}\cL(\mathbf{I-P})g_2+r(g_1,g_2).
\end{aligned}
\end{equation}
	We use the Gaussian moment identities
	\begin{equation}\label{macro4}
\int_{\R^3}(1,v_i^2,v_i^2v_j^2,v_i^4,|v|^2,|v|^4)\mu dv=(1,1,1,3,3,15),\quad i,j=1,2,3,i\neq j,
\end{equation}
	and
	\begin{equation}\label{macro3}
\begin{aligned}
& \big((\mathbf{I-P})g_2,\psi(v)\big)_{L^2_v}=0,~\big(\cL(\mathbf{I-P})g_2,\psi(v)\big)_{L^2_v}=0, \\
& \big(\pa_t(\mathbf{I-P})g_2,\psi(v)\big)_{L^2_v}=0,~\big(v\cdot\na_x(\mathbf{I-P})g_2,\mu^{\f12}\big)_{L^2_v} = 0,\quad\mbox{for}\quad\psi=\mu^{\f12}(1,v,|v|^2).
\end{aligned}
\end{equation}
We prove the first equation in \eqref{macroeq} as a typical case. Taking the inner product of \eqref{macro1} with $\mu^{\f12}$ over $v\in\R^3$ and using \eqref{macro4}, we obtain
	\[
\begin{aligned}
\pa_t a^{g_2}+T(t)^{\f12}\na_x\cdot b^{g_2} &= \big(-(\pa_t+T(t)^{\f12}v\cdot\na_x)(\mathbf{I-P})g_2+T(t)^{-\f32}\cL(\mathbf{I-P})g_2+r(g_1,g_2),\mu^{\f12}\big) \\
&= (r(g_1,g_2),\mu^{\f12}):=\sfT_{11},
\end{aligned}
\]
	where the last step follows from \eqref{macro3}.

	We next take the inner products of \eqref{macro1} with $\mu^{\f12}v_iv_j$ ($i\neq j$) and $\mu^{\f12}(v^2_i-1)/2$ ($i=1,2,3$) over $v\in\R^3$, then with $\mu^{\f12}v$, $\mu^{\f12}\f{|v|^2-3}6$ and finally with $\mu^{\f12}\f{v(|v|^2-5)}{10}$. Applying \eqref{macro4} and \eqref{macro3} then yields the remaining identities in \eqref{macroeq}.

\end{proof}

Moreover, we have the following upper bound for the terms $\sfT_{k1}$ and $\sfT_{l2}$ with $1\leq k\leq 5$ and $l=4,5$.
\begin{lem}
For any multi-index \(\alpha\) with \(|\alpha|\leq 1\),
\begin{equation}\label{Tijupperbound}
\begin{aligned}
& \sum_{k=1}^5\|\pa^\al_x\sfT_{k1}\|^2_{L^2_x}\leq C_\al \Big((T(t)+T(t)^{-3}) \|\<v\>^{-2}(\mathbf{I-P})g_2\|^2_{H^2_xL^2_v} \\
&\quad{} +\sum_{i=1}^{13}\int_{\T^3}\big(\pa^\al_xr(g_1,g_2),\mu^{\f12}e_i\big)^2_{L^2_v}dx\Big), \\
& \sum_{l=4,5}\|{\<D_x\>}\pa^\al_x\sfT_{l2}\|^2_{L^2_x}\leq C_\al\|\<v\>^{-2}(\mathbf{I-P})g_2\|^2_{H^2_xL^2_v},
\end{aligned}
\end{equation}
where
\begin{equation}\label{ei}
\begin{aligned}
&e_1=1,e_2=v_1,e_3=v_2,e_4=v_3,e_5=v_1^2,e_6=v_2^2,e_7=v_3^2, \\
&e_8=v_1v_2,e_9=v_1v_3,e_{10}=v_2v_3,e_{11}=v_1|v|^2,e_{12}=v_2|v|^2,e_{13}=v_3|v|^2.
\end{aligned}
\end{equation}
\end{lem}
\begin{proof}
	We give the estimate for $\sfT_{51}$ with $|\al|\leq 1$; the other terms are treated in the same way. By definition,
	\[
\pa^\al_x\sfT_{51}=\big(-T(t)^{\f12}v\cdot\na_x\pa^\al_x(\mathbf{I-P})g_2+T(t)^{-\f32}\cL(\mathbf{I-P})\pa^\al_xg_2+\pa^\al_xr(g_1,g_2),\mu^{\f12}\f{v(|v|^2-5)}{10}\big)_{L^2_v},
\]
The first two terms satisfy 
\[
\begin{aligned}
& \big|\big(v\cdot\na_x\pa^\al_x(\mathbf{I-P})g_2,\mu^{\f12}\f{v(|v|^2-5)}{10}\big)_{L^2_v}\big|\ls \|\<v\>^{-2}\na_x\pa^\al_x(\mathbf{I-P})g_2\|_{L^2_{v}}, \\
& \big|\big(\cL(\pa^\al_x(\mathbf{I-P})g_2),\mu^{\f12}\f{v(|v|^2-5)}{10}\big)_{L^2_v}\big|\ls \|\<v\>^{-2}\pa^\al_x(\mathbf{I-P})g_2\|_{L^2_{v}},
\end{aligned}
\]
 where we used \eqref{linearized} and \eqref{upperQ2} in the second inequality. Combining these bounds with 
\[
\int_{\T^3}\big(\pa^\al_xr(g_1,g_2),\mu^{\f12}\f{v(|v|^2-5)}{10}\big)^2_{L^2_v}dx\ls \sum_{i=1}^{13}\int_{\T^3}\big(\pa^\al_xr(g_1,g_2),\mu^{\f12}e_i\big)^2_{L^2_v}dx,
\]
we obtain the asserted estimate.
\end{proof}

We now establish the macroscopic estimates.
\begin{lem}\label{macro}
	Let $g_1$ and $g_2$ be solutions of the coupled system \eqref{coupg1} and \eqref{coupg2}. Then there exists a functional $\cG$ satisfying $|\cG|\ls \|g_2\|^2_{H^2_xL^2_v}$ and 
	\begin{equation}\label{macroscopic}
\begin{aligned}
\frac d{dt}\mathcal G+cT^{1/2}\|\mathbf P g_2\|_{H_x^2L_v^2}^2 &\lesssim (T^{1/2}+T^{-7/2})\|(\mathbf I-\mathbf P)g_2\|_{\mathcal D_0}^2 +T^{1/2}\|g_1\|_{L_x^2L_4^2}^2 \\
&\quad{} + T^{-1/2}\sum_{|\alpha|\leq1}\sum_{i=1}^{13} \int_{\mathbb T^3}(\partial_x^\alpha r(g_1,g_2),\mu^{1/2}e_i)_{L_v^2}^2\,dx .
\end{aligned}
\end{equation}
Here \(c>0\) is a numerical constant independent of \(T,E\).
\end{lem}
\begin{proof}
We estimate the three macroscopic components separately and then combine the resulting inequalities.
\smallskip

\textit{Step 1: Control of $\na_x\pa^\al_xc^{g_2}$.} Applying $\pa^\al_x,|\al|\leq 1$ to both sides of the last equation in \eqref{macroeq} and taking the inner product with $\na_x\pa^\al_x c^{g_2}$ over $x\in\T^3$, we obtain
\begin{equation}\label{a001}
\begin{aligned}
&T(t)^{\f12}\|\na_x\pa^\al_x c^{g_2}\|^2_{L^2_x}=\f d{dt}\int_{\T^3}\pa^\al_x\sfT_{52}\na_x\pa^\al_x c^{g_2}dx+\int_{\T^3}\na_x\cdot\pa^\al_x\sfT_{52}\pa^\al_x \pa_t c^{g_2}dx \\
&\quad{} +\int_{\T^3}\pa^\al_x\sfT_{51}\na_x\pa^\al_xc^{g_2}dx.
\end{aligned}
\end{equation}
For the term $\pa_t\pa^\al_xc^{g_2}$ in \eqref{a001}, it follows from the third equation in \eqref{macroeq} that
\[
\pa_t \pa^\al_xc^{g_2}=-\f13T(t)^{\f12}\na_x\cdot\pa^\al_xb^{g_2}+\pa^\al_x\sfT_{31}.
\]
Substituting this identity, integrating by parts, and applying Young's inequality gives
\[
\begin{aligned}
\int_{\mathbb T^3}\nabla_x\cdot\partial_x^\alpha\mathsf T_{52} \,\partial_t\partial_x^\alpha c^{g_2}\,dx &= -\tfrac13T^{1/2}\int_{\mathbb T^3} (\nabla_x\cdot\partial_x^\alpha\mathsf T_{52}) (\nabla_x\cdot\partial_x^\alpha b^{g_2})\,dx \\
&\quad{} + \int_{\mathbb T^3} (\nabla_x\cdot\partial_x^\alpha\mathsf T_{52}) \partial_x^\alpha\mathsf T_{31}\,dx \\
&\leq \varepsilon T^{1/2}\|\nabla_x\partial_x^\alpha b^{g_2}\|_{L_x^2}^2 +C_\varepsilon\left(T^{1/2}\|\nabla_x\partial_x^\alpha\mathsf T_{52}\|_{L_x^2}^2 +T^{-1/2}\|\partial_x^\alpha\mathsf T_{31}\|_{L_x^2}^2\right).
\end{aligned}
\]
For the third term on the right-hand side of \eqref{a001}, the Cauchy--Schwarz inequality gives 
\[
\int_{\T^3}\pa^\al_x\sfT_{51}\na_x\pa^\al_xc^{g_2}dx\leq \f12 T(t)^{\f12}\|\na_x\pa^\al_xc^{g_2}\|^2_{L^2_x}+2 T(t)^{-\f12}\|\pa^\al_x \sfT_{51}\|^2_{L^2_x}.
\]
Combining these bounds with \eqref{Tijupperbound} introduces the coefficient \(T^{-1/2}(T+T^{-3})=T^{1/2}+T^{-7/2}\) and gives
\begin{equation}\label{c}
\begin{aligned}
& \f12T(t)^{\f12}\|\na_x\pa^\al_x c^{g_2}\|^2_{L^2_x}\leq \f d{dt}\int_{\T^3}\pa^\al_x\sfT_{52}\na_x\pa^\al_x c^{g_2}dx+\vep T(t)^{\f12}\|\na_x\pa^\al_x b^{g_2}\|^2_{L^2_x} \\
&\quad{} + C_{\al,\vep}\Big( (T(t)^{\f12}+T(t)^{-\f72})\|(\mathbf{I-P})g_2\|^2_{\cD_0}+T(t)^{-\f12}\sum_{i=1}^{13}\int_{\T^3}\big(\pa^\al_xr(g_1,g_2),\mu^{\f12}e_i\big)^2_{L^2_v}dx\Big),
\end{aligned}
\end{equation}
where we have used $\|\<v\>^{-2}(\mathbf{I-P})g_2\|^2_{H^2_xL^2_v}\ls \|(\mathbf{I-P})g_2\|^2_{\cD_0}$.

\textit{Step 2: Control of $\nabla_x\partial_x^\alpha b^{g_2}$.}
Combining the fourth and fifth equations in \eqref{macroeq} gives, for each
\(i=1,2,3\),
\[
\begin{aligned}
T^{1/2}(\Delta_x b_i^{g_2}+\partial_{ii}b_i^{g_2}) &= \sum_{j\ne i}\big(\partial_j(\mathsf T_{41})_{ij} -\partial_i(\mathsf T_{41})_{jj}\big)+2\partial_i(\mathsf T_{41})_{ii} \\
&\quad{} + \partial_t\left[ \sum_{j\ne i}\big(\partial_j(\mathsf T_{42})_{ij} -\partial_i(\mathsf T_{42})_{jj}\big)+2\partial_i(\mathsf T_{42})_{ii} \right].
\end{aligned}
\]
We apply \(\partial_x^\alpha\), multiply by
\(-\partial_x^\alpha b_i^{g_2}\), and sum in \(i\). The left-hand side
becomes
\[
T^{1/2}\left(\|\nabla_x\partial_x^\alpha b^{g_2}\|_{L_x^2}^2 +\sum_i\|\partial_i\partial_x^\alpha b_i^{g_2}\|_{L_x^2}^2\right).
\]
For the time derivatives on the right-hand side, integration by parts
in time gives minus the derivative of
\[
\sum_i\int_{\mathbb T^3} \partial_x^\alpha\left[ \sum_{j\ne i}\big(\partial_j(\mathsf T_{42})_{ij} -\partial_i(\mathsf T_{42})_{jj}\big)+2\partial_i(\mathsf T_{42})_{ii} \right]\partial_x^\alpha b_i^{g_2}\,dx ,
\]
and the same integral with \(\partial_x^\alpha b_i^{g_2}\) replaced by
\(\partial_t\partial_x^\alpha b_i^{g_2}\), with a plus sign.
Substitute
\[
\partial_t\partial_x^\alpha b^{g_2} =-T^{1/2}\nabla_x\partial_x^\alpha(a^{g_2}+2c^{g_2}) +\partial_x^\alpha\mathsf T_{21}.
\]
For the terms containing \(\mathsf T_{41}\), we integrate by parts in \(x\).
Using \eqref{Tijupperbound} and Young's inequality, we obtain
\begin{equation}\label{b}
\begin{aligned}
&\frac12T^{1/2}\|\nabla_x\partial_x^\alpha b^{g_2}\|_{L_x^2}^2 \\
&\leq -\frac d{dt}\sum_i\int_{\mathbb T^3}
\partial_x^\alpha\left[
\sum_{j\ne i}\big(\partial_j(\mathsf T_{42})_{ij}
-\partial_i(\mathsf T_{42})_{jj}\big)
+2\partial_i(\mathsf T_{42})_{ii}\right]
\partial_x^\alpha b_i^{g_2}\,dx \\
&\quad{} + \varepsilon T^{1/2}\left(
\|\nabla_x\partial_x^\alpha a^{g_2}\|_{L_x^2}^2
+\|\nabla_x\partial_x^\alpha c^{g_2}\|_{L_x^2}^2\right) \\
&\quad{} + C_{\alpha,\varepsilon}(T^{1/2}+T^{-7/2})
\|(\mathbf I-\mathbf P)g_2\|_{\mathcal D_0}^2 \\
&\quad{} + C_{\alpha,\varepsilon}T^{-1/2}\sum_{i=1}^{13}
\int_{\mathbb T^3}(\partial_x^\alpha r(g_1,g_2),\mu^{1/2}e_i)_{L_v^2}^2\,dx .
\end{aligned}
\end{equation}
In particular, all the components of \(\mathsf T_{42}\), including the
diagonal components, have the same sign in this time derivative.

\textit{Step 3: Control of $\na_x\pa^\al_xa^{g_2}$.} To estimate $\na_x\pa^\al_x a^{g_2}$, we use the second and sixth equations in \eqref{macroeq} to obtain
\[
T(t)^{\f12}\na_x a^{g_2}=-\pa_tb^{g_2}+\sfT_{21}-2\sfT_{51}-2\pa_t\sfT_{52}.
\]
Applying $\pa^\al_x$ to both sides and taking the inner product with $\na_x\pa^\al_x a^{g_2}$ over $\T^3$, we obtain
{
\[
\begin{aligned}
&T(t)^{\f12}\|\na_x\pa^\al_x a^{g_2}\|^2_{L^2_x}=-\int_{\T^3}\pa_t (\pa^\al_xb^{g_2}+2\pa^\al_x\sfT_{52})\na_x\pa^\al_x a^{g_2}dx  +\int_{\T^3}(\pa^\al_x\sfT_{21}-2\pa^\al_x\sfT_{51})\na_x\pa^\al_x a^{g_2}dx \\
&= -\f d{dt}\int_{\T^3}(\pa^\al_xb^{g_2}+2\pa^\al_x\sfT_{52})\na_x\pa^\al_x a^{g_2}dx-\int_{\T^3}\na_x\cdot(\pa^\al_xb^{g_2}+2\pa^\al_x\sfT_{52})\pa^\al_x \pa_ta^{g_2}dx \\
&\quad{} +\int_{\T^3}(\pa^\al_x\sfT_{21}-2\pa^\al_x\sfT_{51})\na_x\pa^\al_x a^{g_2}dx.
\end{aligned}
\]
}
Substituting the time derivative from \eqref{macroeq} gives
\[
\pa_t \pa^\al_xa^{g_2}+T(t)^{\f12}\na_x\cdot \pa^\al_xb^{g_2}=\pa^\al_x\sfT_{11}.
\]
The Cauchy--Schwarz inequality and \eqref{Tijupperbound} then yield
{
\begin{equation}\label{a}
\begin{aligned}
&\f12T(t)^{\f12}\|\na_x\pa^\al_x a^{g_2}\|^2_{L^2_x} \leq -\f d{dt}\int_{\T^3}(\pa^\al_xb^{g_2}+2\pa^\al_x\sfT_{52})\na_x\pa^\al_x a^{g_2}dx \\
&\quad{} +2T(t)^{\f12}\|\na_x\pa^\al_x b^{g_2}\|^2_{L^2_x} \\
&\quad{} + C_\al\Big( (T(t)^{\f12}+T(t)^{-\f72})\|(\mathbf{I-P})g_2\|^2_{\cD_0}+T(t)^{-\f12}\sum_{i=1}^{13}\int_{\T^3}\big(\pa^\al_xr(g_1,g_2),\mu^{\f12}e_i\big)^2_{L^2_v}dx\Big).
\end{aligned}
\end{equation}
}

\textit{Step 4: Macroscopic estimates.} Taking the linear combination $\eqref{c}+\kappa\times \eqref{b}+\eqref{a}$ for some constant $\kappa$ and summing over $|\al|\leq 1$ yields
\[
\begin{aligned}
& \f d{dt}\cG+T(t)^{\f12}\sum_{|\al|\leq 1}\Big((\f12-\ka\vep)\|\na_x\pa^\al_x a^{g_2}\|^2_{L^2_x}+(\f\ka2-2-\vep)\|\na_x\pa^\al_x b^{g_2}\|^2_{L^2_x} +(\f12-\ka\vep)\|\na_x\pa^\al_x c^{g_2}\|^2_{L^2_x}\Big) \\
&\leq C_{\vep,\al,\ka}\Big( (T(t)^{\f12}+T(t)^{-\f72})\|(\mathbf{I-P})g_2\|^2_{\cD_0} +T(t)^{-\f12}\sum_{|\al|\leq 1}\sum_{i=1}^{13}\int_{\T^3}\big(\pa^\al_xr(g_1,g_2),\mu^{\f12}e_i\big)^2_{L^2_v}dx\Big).
\end{aligned}
\]
with 
\begin{equation}\label{cG}
\begin{aligned}
\cG &= \sum_{|\al|\leq 1}\Big(-\int_{\T^3}\pa^\al_x\sfT_{52}\na_x\pa^\al_x c^{g_2}dx+\int_{\T^3}(\pa^\al_xb^{g_2}+2\pa^\al_x\sfT_{52})\na_x\pa^\al_x a^{g_2}dx \\
&\quad{} + \ka\sum_{i=1}^3\sum_{j\neq i}\int_{\T^3}\pa^\al_x\big(\pa_j(\sfT_{42})_{ij}-\pa_i(\sfT_{42})_{jj}\big)\pa^\al_xb_i^{g_2}dx  +2\ka\sum_{i=1}^3\int_{\T^3}\pa^\al_x\pa_i(\sfT_{42})_{ii}\pa^\al_xb_i^{g_2}dx\Big).
\end{aligned}
\end{equation}
Choosing $\ka=8$ and $\vep=1/32$ gives
\begin{equation}\label{cG2}
\begin{aligned}
& \f d{dt}\cG+\f14T(t)^{\f12}\sum_{|\al|\leq 1}\Big(\|\na_x\pa^\al_x a^{g_2}\|^2_{L^2_x}+\|\na_x\pa^\al_x b^{g_2}\|^2_{L^2_x}+\|\na_x\pa^\al_x c^{g_2}\|^2_{L^2_x}\Big) \\
&\ls (T(t)^{\f12}+T(t)^{-\f72})\|(\mathbf{I-P})g_2\|^2_{\cD_0}+T(t)^{-\f12}\sum_{|\al|\leq 1}\sum_{i=1}^{13}\int_{\T^3}\big(\pa^\al_xr(g_1,g_2),\mu^{\f12}e_i\big)^2_{L^2_v}dx.
\end{aligned}
\end{equation}

First, \eqref{cG} and \eqref{Tijupperbound} give the bound on $\cG$:
\begin{equation}\label{cGupperbound}
\begin{aligned}
&|\cG|\ls\|(\mathbf{I-P})g_2\|^2_{H^{2}_xL^2_{v}} \\
&\quad{} +\sum_{|\al|\leq 1}\Big(\|\na_x\pa^\al_x a^{g_2}\|^2_{L^2_x}+\|\pa^\al_x b^{g_2}\|^2_{L^2_x}+\|\na_x\pa^\al_x c^{g_2}\|^2_{L^2_x}\Big)\ls \|g_2\|^2_{H^2_xL^2_v}.
\end{aligned}
\end{equation}
Next, \eqref{conforg} and the relation $g=g_1+\mu^{\f12}g_2$ imply
\[
\int_{\T^3}(a^{g_2}(x),b^{g_2}(x),c^{g_2}(x))dx=-\int_{\T^3\times\R^3}g_1(1,v,\f16(|v|^2-3))dvdx=:\big(\bar{a},\bar{b},\bar{c}).
\]
The Poincar\'e inequality on $\T^3$ together with $|\bar{a}|^2+|\bar{b}|^2+|\bar{c}|^2\ls\|g_1\|^2_{L^2_{x}L^2_{4}}$ then gives
\[
\begin{aligned}
& \sum_{|\al|\leq 1}\Big(\|\na_x\pa^\al_x a^{g_2}\|^2_{L^2_x}+\|\na_x\pa^\al_x b^{g_2}\|^2_{L^2_x}+\|\na_x\pa^\al_x c^{g_2}\|^2_{L^2_x}\Big) \\
&\gs \sum_{|\al|\leq 2}\Big(\|\pa^\al_x( a^{g_2}-\bar{a})\|^2_{L^2_x}+\|\pa^\al_x (b^{g_2}-\bar{b})\|^2_{L^2_x}+\|\pa^\al_x (c^{g_2}-\bar{c})\|^2_{L^2_x}\Big) \\
&\gs \sum_{|\al|\leq 2}\Big(\|\pa^\al_xa^{g_2}\|^2_{L^2_x}+\|\pa^\al_x b^{g_2}\|^2_{L^2_x}+\|\pa^\al_x c^{g_2}\|^2_{L^2_x}\Big)-\|g_1\|^2_{L^2_xL^2_4}.
\end{aligned}
\]
Together with \eqref{cG2} and \eqref{cGupperbound}, this yields the desired result \eqref{macroscopic} and completes the proof of the lemma.
\end{proof}

Combining Lemmas \ref{micro} and \ref{macro}, we obtain the following estimate.
\begin{prop}\label{g2}
Let \(k\geq17\) and \(T(t)\geq T_0/4>0\).
There is an energy functional equivalent to \(\|g_2\|_{\mathcal E_k}^2\)
such that
\begin{equation}\label{barcEk2}
\begin{aligned}
\frac12\frac d{dt}\|g_2\|_{\bar{\mathcal E}_k}^2 +T^{-3/2}\big(\lambda_2-C_k\|g_2\|_{\bar{\mathcal E}_{17}}\big) \|g_2\|_{\mathcal D_k}^2 &\leq C_k(T^{-3/2}\lambda_2^{-1}+T^{-11/2})\|g_1\|_{Y_k}^2 \\
&\quad{} + C_kT^{-3}|T'|\|g_2\|_{\mathcal E_0}^2 +C_kT^{-11/2}\|g_2\|_{\mathcal E_0}^4 .
\end{aligned}
\end{equation}
Here \(0<\lambda_2\lesssim_k1\),
\(\lambda_2\gtrsim_k\min\{T_0^4,1\}\), and
\[
\min\{T_0^4,1\}\|g_2\|_{\mathcal E_k}^2 \lesssim_k\|g_2\|_{\bar{\mathcal E}_k}^2 \lesssim\|g_2\|_{\mathcal E_k}^2 .
\]
The constants are independent of \(E,t\).
\end{prop}
\begin{proof}
Add \eqref{cE0cD0} to \(\kappa_1T^{-2}/2\) times
\eqref{macroscopic}. The microscopic contribution of the latter is
\[
C\kappa_1T^{-3/2}(1+T^{-4}) \|(\mathbf I-\mathbf P)g_2\|_{\mathcal D_0}^2 .
\]
We therefore choose
\(\kappa_1=\kappa_2\min\{T_0^4,1\}\), with a sufficiently small
fixed \(\kappa_2>0\). Since \(T\geq T_0/4\), this choice ensures both
\[
C\kappa_1(1+T^{-4})\leq\lambda_0/4,\qquad \kappa_1T^{-2}|\mathcal G|\leq\tfrac14\|g_2\|_{\mathcal E_0}^2 .
\]
In particular, the first inequality accounts for the full negative
temperature power in the macroscopic estimate. Set
\[
\|g_2\|_{\bar{\mathcal E}_0}^2 =\|g_2\|_{\mathcal E_0}^2+\kappa_1T^{-2}\mathcal G .
\]
Differentiating the coefficient \(T^{-2}\) and using
\eqref{cGupperbound} gives the following inequality:
\begin{equation}\label{001}
\begin{aligned}
\frac12\frac d{dt}\|g_2\|_{\bar{\mathcal E}_0}^2 &+T^{-3/2}\big(\tfrac34\lambda_0-C\|g_2\|_{\mathcal E_0}\big) \|(\mathbf I-\mathbf P)g_2\|_{\mathcal D_0}^2 \\
&\quad{} + T^{-3/2}\big(c\kappa_1-C\|g_2\|_{\mathcal E_0}\big) \|\mathbf P g_2\|_{H_x^2L_v^2}^2 \\
&\leq C_kT^{-3/2} \left(\|g_2\|_{\mathcal D_0}\|g_1\|_{Y_0} +\|g_1\|_{L_x^2L_4^2}^2\right) +C\kappa_1T^{-3}|T'|\|g_2\|_{\mathcal E_0}^2 \\
&\quad{} + C\kappa_1T^{-5/2}\sum_{|\alpha|\leq1}\sum_{i=1}^{13} \int_{\mathbb T^3} (\partial_x^\alpha r(g_1,g_2),\mu^{1/2}e_i)_{L_v^2}^2\,dx .
\end{aligned}
\end{equation}
The definition \eqref{rg2}, the compact support of \(A\chi_M\), and
\eqref{upperQ2} give
\[
T^{-5/2}\sum_{|\alpha|\leq1}\sum_{i=1}^{13} \int_{\mathbb T^3} (\partial_x^\alpha r(g_1,g_2),\mu^{1/2}e_i)_{L_v^2}^2\,dx \leq C_kT^{-11/2} \left(\|g_1\|_{H_x^2L_v^2}^2+\|g_2\|_{\mathcal E_0}^4\right).
\]
Moreover,
\(\|(\mathbf I-\mathbf P)g_2\|_{\mathcal D_0}^2+
\|\mathbf P g_2\|_{H_x^2L_v^2}^2\gtrsim\|g_2\|_{\mathcal D_0}^2\).
Choose \(0<\lambda_1<1\) comparable to \(\kappa_1\), sufficiently small
to absorb the mixed term in \eqref{001} by Young's inequality. Then
\begin{equation}\label{barcE0}
\begin{aligned}
\frac12\frac d{dt}\|g_2\|_{\bar{\mathcal E}_0}^2 +T^{-3/2}(\lambda_1-C\|g_2\|_{\mathcal E_0}) \|g_2\|_{\mathcal D_0}^2 &\leq C_kT^{-3/2} \left(\lambda_1^{-1}\|g_1\|_{Y_0}^2 +\|g_1\|_{L_x^2L_4^2}^2\right) \\
&\quad{} + C_kT^{-3}|T'|\|g_2\|_{\mathcal E_0}^2 +C_kT^{-11/2} \left(\|g_1\|_{H_x^2L_v^2}^2+\|g_2\|_{\mathcal E_0}^4\right).
\end{aligned}
\end{equation}

Next, we add \(\kappa_3\) times \eqref{cEkcDk}, choosing
\(0<\kappa_3<1\) so that \(C_k\kappa_3\leq\lambda_1/4\) and
\(\kappa_3\) is comparable, with constants depending on \(k\), to
\(\lambda_1\). Define
\begin{equation}\label{func1}
\|g_2\|_{\bar{\mathcal E}_k}^2 =\|g_2\|_{\bar{\mathcal E}_0}^2+\kappa_3\|g_2\|_{\mathcal E_k}^2 .
\end{equation}
This functional has the stated equivalence constants because
\(\lambda_1,\kappa_3\gtrsim_k\min\{T_0^4,1\}\).
Use
\[
\|g_1\|_{Y_0}+\|g_1\|_{L_x^2L_4^2} +\|g_1\|_{H_x^2L_v^2}\lesssim_k\|g_1\|_{Y_k}, \qquad \|g_2\|_{\mathcal E_0} +\kappa_3^{1/2}\|g_2\|_{\mathcal E_{17}} \lesssim\|g_2\|_{\bar{\mathcal E}_{17}} .
\]
The remaining microscopic and weighted dissipations give
\eqref{barcEk2}, with \(\lambda_2\) a sufficiently small fixed multiple
of \(\kappa_3\). For two fixed indices, all these coefficients may be
chosen using the larger constants in the two estimates, so that the
same functionals are available at both indices.
\end{proof}

We now combine the estimates for the coupled system.
\begin{thm}\label{energyestimate}
Let \(k\geq k_0\), where \(k_0\geq17\) is the fixed integer in
Proposition~\ref{gg1}. Let $(g_1,g_2)$ solve \eqref{coupg1}--\eqref{coupg2}, with $g=g_1+\mu^{1/2}g_2$ satisfying \eqref{conforg}. Assume $T(t)\geq T_0/4>0$ and \eqref{VTcondition0} on the interval under consideration, with $N$ chosen below.
There are equivalent energy and dissipation functionals satisfying
\begin{equation}\label{normT(0)}
\begin{aligned}
\min\{T_0^8,1\} \big(\|g_1\|_{X_k}^2+\|g_2\|_{\mathcal E_k}^2\big) &\lesssim_k\|(g_1,g_2)\|_{\mathbf E_k}^2 \lesssim_k\|g_1\|_{X_k}^2+\|g_2\|_{\mathcal E_k}^2, \\
\min\{T_0^8,1\} \big(\|g_1\|_{Y_k}^2+\|g_2\|_{\mathcal D_k}^2\big) &\lesssim_k\|(g_1,g_2)\|_{\mathbf D_k}^2 \lesssim_k\|g_1\|_{Y_k}^2+\|g_2\|_{\mathcal D_k}^2 .
\end{aligned}
\end{equation}
Moreover,
\begin{equation}\label{energyinequality}
\begin{aligned}
&\frac12\frac d{dt}\|(g_1,g_2)\|_{\mathbf E_k}^2 +T^{-3/2}\left(\frac{\lambda_0}{32} -C_{k,T_0}\|(g_1,g_2)\|_{\mathbf E_{k_0}}\right) \|(g_1,g_2)\|_{\mathbf D_k}^2 \\
&\leq C_k\left(T^{-3}|T'|\|g_2\|_{\mathcal E_0}^2 +T^{-11/2}\|g_2\|_{\mathcal E_0}^4\right) \\
&\quad+C_k\left(T^{-1}|T'|+T^{-1/2}|R|+T^{-1}|V|^{-1}\right) \left(\|g_1\|_{X_k}^2+\|g_2\|_{\mathcal E_0}^2 +\|g_2\|_{\mathcal D_0}\|g_1\|_{X_k}\right) \\
&\quad+C_kT^{3/2} \left(T^{-1}|T'|+T^{-1/2}|R|+T^{-1}|V|^{-1}\right)^2 .
\end{aligned}
\end{equation}
All constants are independent of \(E\) and time. For any two fixed
indices \(k_0,k_1\), the cutoff and the coefficients of the functionals
may be chosen simultaneously for these indices; constants may then
depend on both indices.
\end{thm}
\begin{proof}
We first fix \(A,M\) so that Lemma~\ref{QGgg} applies at the required
indices. When two indices are used, take common \(A,M\) at this stage;
all constants below are then fixed independently of \(T,E,t_0\).
We add \(\kappa_4\) times \eqref{barcEk2} to \eqref{evolutiongand g}.
Choose \(0<\kappa_4<1\) so that
\[
C_k\kappa_4\big(\lambda_2^{-1}+(4/T_0)^4\big) \leq\lambda_0/32 .
\]
Since \(\lambda_2\gtrsim_k\min\{T_0^4,1\}\), we may take
\(\kappa_4\) comparable, with constants depending on \(k\), to
\(\min\{T_0^4,1\}\). This absorbs the terms containing
\(\|g_1\|_{Y_k}^2\) from \eqref{barcEk2}.
We then take \(N\) sufficiently large that the terms with coefficient
\(C_k/N\) are at most \(\kappa_4\lambda_2/8\) times the corresponding
\(g_2\) norm.

Define
\begin{equation}\label{func2}
\|(g_1,g_2)\|_{\mathbf E_k}^2 =\|\chi_N^c g\|_{X_k}^2+\|\chi_{4N}g_1\|_{X_k}^2 +\kappa_4\|g_2\|_{\bar{\mathcal E}_k}^2 .
\end{equation}
By \eqref{a1} and \eqref{func1}, we increase \(N\) if necessary to obtain
\[
\|(g_1,g_2)\|_{\mathbf E_k}^2 \gtrsim_k\|g_1\|_{X_k}^2 +\kappa_4\min\{T_0^4,1\}\|g_2\|_{\mathcal E_k}^2 .
\]
The reverse bound follows from \(g=g_1+\sqrt\mu g_2\).
This shows that the energy equivalence in \eqref{normT(0)} has power eight.
For the dissipation, set
\[
\|(g_1,g_2)\|_{\mathbf D_k}^2 =\|g_1\|_{Y_k}^2+\kappa_4\lambda_2\|g_2\|_{\mathcal D_k}^2 .
\]
This gives the second equivalence in \eqref{normT(0)}.

After the preceding absorptions, the linear dissipation bounds
\(\lambda_0\|g_1\|_{Y_k}^2/16+
\kappa_4\lambda_2\|g_2\|_{\mathcal D_k}^2/4\),
times \(T^{-3/2}\). We may decrease the uniform constant
\(\lambda_0\) so that \(\lambda_0\leq1\).
The nonlinear terms are bounded by
\[
C_{k,T_0}T^{-3/2}\|(g_1,g_2)\|_{\mathbf E_{k_0}} \|(g_1,g_2)\|_{\mathbf D_k}^2 ,
\]
using \eqref{normT(0)} and \(k_0\geq17\). The remaining terms are exactly
the right-hand side of \eqref{energyinequality}. This proves the estimate.
For two fixed indices, take the larger of the required \(N\) and
the smaller of the admissible positive coefficients throughout.
\end{proof}

\section{Proof of the main results}\label{proof}
We combine the estimates on a finite interval of existence.
The small initial size and the initial time will be fixed before
the electric field is chosen. In particular, neither will be decreased
as \(|E|\) increases.

\begin{proof}[Proof of Theorems \ref{longtimebehavior} and \ref{longtimebehavior1}]
Fix \(k_1>k_0+3\), and choose the coefficients in
Theorem~\ref{energyestimate} simultaneously at \(k_0,k_1\).
In this proof, \(\varepsilon_1\) denotes the small parameter of
Theorem~\ref{thm:nu-local}. We shall choose it depending only on
\(k_0,k_1,T_0,V_0\), and require
\(\|g_0\|_{X_{k_1}}<\varepsilon_1^2\).
Theorems~\ref{thm:nu-local} and \ref{localwellposedness} then give
the solution up to the time \(t_0\) in \eqref{T02}, with
\(\|g(t_0)\|_{X_{k_1}}^2\leq\varepsilon_1^2/2\).
We take \(t_0\leq1\), and start the coupled system at
\(g_1(t_0)=g(t_0)\), \(g_2(t_0)=0\).
Thus
\[
\|(g_1,g_2)(t_0)\|_{\mathbf E_{k_0}}^2 +\|(g_1,g_2)(t_0)\|_{\mathbf E_{k_1}}^2 \leq C_{k_0,k_1,T_0}\varepsilon_1^2 .
\]

\textit{Step 1: Estimates of the coefficients and the energy inequality.}
We first work on a finite interval \([0,S]\) on which
\eqref{apriofg} holds with a fixed sufficiently small \(\eta>0\).
Lemma~\ref{largetime} applies on this interval. Once \(t_0\) is fixed,
we may choose the field threshold, depending on \(t_0\) as well, so that
\begin{equation}\label{VT1}
\begin{aligned}
\tfrac14(3+|E|t)\leq |V(t)|&\leq3(3+|E|t), \\
\tfrac14T_0+\frac{c_1}{|E|}\log(3+|E|t) \leq T(t)&\leq\tfrac32T_0+\frac{c_2}{|E|}\log(3+|E|t), \qquad t_0\leq t\leq S .
\end{aligned}
\end{equation}
In particular, for \(|E|\geq1\),
\begin{equation}\label{VT4}
T(t)\geq T_0/4,\qquad T(t)\leq C_{T_0}(1+\log(1+t)).
\end{equation}
The latter constant is independent of \(E,t_0,S\), since
\(\log(3+|E|t)\leq\log(3+|E|)+\log(1+t)\).
Equations \eqref{R4}, \eqref{T'} and \eqref{normT(0)} also imply
\begin{equation}\label{VT 2}
|T'|+|V||R| \leq C_{k_0,k_1,T_0}|V|^{-1} \left(1+\|(g_1,g_2)\|_{\mathbf D_{k_0}}\right).
\end{equation}

For fixed \(t_0>0\), the bounds \eqref{VT1} show that
\[
\sup_{t_0\leq t\leq S} \left(T^{-3/2}|V|^{-1} +T^{1/2}|V|^{-1}+T|V|^{-2}+T^2|V|^{-4}\right) \longrightarrow0\qquad\text{as }|E|\longrightarrow\infty .
\]
The convergence is uniform in the finite time of existence $S$.
Indeed, \(t\geq t_0\) and \eqref{VT4} bound the terms containing positive powers of
\(T\) by a logarithm divided by a positive power of \(|E|t\).
We use this fact to absorb terms linear in the dissipation. We also choose
the threshold so that
\begin{equation}\label{VT5}
C_{k_0,k_1,T_0} \left(T^{1/2}|V|^{-1}+T|V|^{-2}\right) \leq |V|^{-1/2}.
\end{equation}
Indeed, after division by \(|V|^{-1/2}\), the two terms are
\(\sqrt{T/|V|}\) and \(T/|V|^{3/2}\), both uniformly small on
\([t_0,S]\), uniformly in $S$. All choices here are made after fixing \(t_0\).

For completeness, denote the three lines on the right of
\eqref{energyinequality} by \(R_1,R_2,R_3\), in that order.
For any sufficiently small fixed \(\delta>0\), \eqref{VT 2} and Young's
inequality give, for \(k=k_0,k_1\),
\[
\begin{aligned}
R_1&\leq T^{-3/2}\left(\delta+ C_{k_0,k_1,T_0}\|(g_1,g_2)\|_{\mathbf E_{k_0}}^2\right) \|(g_1,g_2)\|_{\mathbf D_k}^2, \\
R_2&\leq T^{-3/2}\left(2\delta+ C_{k_0,k_1,T_0}|V|^{-1/2}\|(g_1,g_2)\|_{\mathbf E_k}^2\right) \|(g_1,g_2)\|_{\mathbf D_k}^2 \\
&\quad+C_{k_0,k_1,T_0}T^{-1}|V|^{-1} \|(g_1,g_2)\|_{\mathbf E_k}^2, \\
R_3&\leq \delta T^{-3/2}\|(g_1,g_2)\|_{\mathbf D_k}^2 +C_{k_0,k_1,T_0}T^{-1/2}|V|^{-2}.
\end{aligned}
\]
To see the first bound, use
\[
\|g_2\|_{\mathcal E_0}^2 \lesssim_{k_0,k_1,T_0} \min\{\|(g_1,g_2)\|_{\mathbf E_{k_0}}^2, \|(g_1,g_2)\|_{\mathbf D_k}^2\}.
\]
In the second bound, the coefficient of the terms involving
\(\|(g_1,g_2)\|_{\mathbf D_k}\) is at most a constant times
\(T^{-1}|V|^{-1}+T^{-1/2}|V|^{-2}\).
Use \(ab\leq a^2+b^2\) for the mixed norms and then \eqref{VT5};
this preserves the factor \(T^{-3/2}\).
For the third bound, the coefficient before the constant and the squared
dissipation is \(C(T^{-1/2}|V|^{-2}+T^{1/2}|V|^{-4})\);
the preceding smallness conditions absorb its dissipative part.

Taking \(\delta\) sufficiently small in terms of \(\lambda_0\), we obtain
\begin{equation}\label{simpleenergyinequ}
\begin{aligned}
&\frac d{dt}\|(g_1,g_2)\|_{\mathbf E_k}^2 \\
&\quad{} +T^{-3/2}\Big(c-C_{k_0,k_1,T_0}
\big(\|(g_1,g_2)\|_{\mathbf E_{k_0}}
+\|(g_1,g_2)\|_{\mathbf E_{k_0}}^2
+|V|^{-1/2}\|(g_1,g_2)\|_{\mathbf E_k}^2\big)\Big)
\|(g_1,g_2)\|_{\mathbf D_k}^2 \\
&\leq C_{k_0,k_1,T_0}T^{-1}|V|^{-1}
\|(g_1,g_2)\|_{\mathbf E_k}^2
+C_{k_0,k_1,T_0}T^{-1/2}|V|^{-2}.
\end{aligned}
\end{equation}
Here and below \(c>0\) and the other fixed constants are independent
of \(E,t_0,S,\varepsilon_1\). The field threshold is allowed to depend
on the already fixed \(t_0,\varepsilon_1\).

\textit{Step 2: A high-order bound retaining the small initial size.}
Set, as in the interpolation of the two energy levels,
\[
\theta=\frac{k_1-k_0}{k_1-k_0+3/2},\qquad \theta_1=\frac1\theta-1,\qquad p=\theta+\theta^2 .
\]
Since \(k_1-k_0>3\), we have \(\theta>2/3\), \(p>1\),
\(p\theta_1<1\), and \(p(1+\theta_1)<2\).
Fix a sufficiently small \(\gamma>0\) so that
\begin{equation}\label{decay:exponents}
\gamma<\tfrac12,\qquad p\theta_1+\gamma(1+\theta_1)<1,\qquad (p+\gamma)(1+\theta_1)<2 .
\end{equation}
For example, \(\gamma=(1-\theta)/8\) satisfies these conditions.
This number depends only on \(k_0,k_1\).

On a finite interval starting at \(t_0\), assume in addition that
the three error terms in the dissipation coefficient of
\eqref{simpleenergyinequ} have sum at most \(c/2\), at both indices.
This is a continuity assumption, satisfied initially for small
\(\varepsilon_1\) and large \(|E|\).
By increasing the field threshold after fixing \(t_0,\varepsilon_1\), we ensure
\begin{equation}\label{decay:coefficients}
C_{k_0,k_1,T_0}T^{-1}|V|^{-1}\leq\frac{\gamma}{1+t}, \qquad C_{k_0,k_1,T_0}T^{-1/2}|V|^{-2} \leq\frac{\varepsilon_1^2}{(1+t)^2}.
\end{equation}
For example, \eqref{VT1} and \(t_0\leq1\) show that it suffices to take
\(|E|t_0\geq C_{k_0,k_1,T_0}
\max\{\gamma^{-1},\varepsilon_1^{-1}\}\).

The energy inequality now reads
\begin{equation}\label{simple2}
\frac d{dt}\|(g_1,g_2)\|_{\mathbf E_k}^2 +cT^{-3/2}\|(g_1,g_2)\|_{\mathbf D_k}^2 \leq\frac{\gamma}{1+t}\|(g_1,g_2)\|_{\mathbf E_k}^2 +\frac{\varepsilon_1^2}{(1+t)^2}, \qquad k=k_0,k_1 .
\end{equation}
We multiply its high-order version by \((1+t)^{-\gamma}\) and integrate
from \(t_0\). Since
\(\int_{t_0}^\infty(1+t)^{-2-\gamma}dt\leq1\), we obtain
\begin{equation}\label{t<}
\|(g_1,g_2)(t)\|_{\mathbf E_{k_1}}^2 \leq C_{k_0,k_1,T_0}\varepsilon_1^2(1+t)^\gamma .
\end{equation}
This bound retains the factor \(\varepsilon_1^2\).
Furthermore, \(|V(t)|\gtrsim |E|t_0(1+t)\), so
\begin{equation}\label{aprioriassumption}
|V(t)|^{-1/2}\|(g_1,g_2)(t)\|_{\mathbf E_{k_1}}^2 \leq \frac{C_{k_0,k_1,T_0}\varepsilon_1^2}{\sqrt{|E|t_0}},
\end{equation}
because \(\gamma<1/2\). The same estimate, up to a fixed constant,
holds at index \(k_0\).

\textit{Step 3: Decay, time integrability, and continuation.}
Write \(Z(t)=\|(g_1,g_2)(t)\|_{\mathbf E_{k_0}}^2\).
Weighted interpolation and \eqref{normT(0)} imply
\[
\|(g_1,g_2)\|_{\mathbf E_{k_0}} \leq C_{k_0,k_1,T_0} \|(g_1,g_2)\|_{\mathbf D_{k_0}}^\theta \|(g_1,g_2)\|_{\mathbf E_{k_1}}^{1-\theta}.
\]
Consequently, \eqref{t<} gives
\[
\|(g_1,g_2)\|_{\mathbf D_{k_0}}^2 \geq c\varepsilon_1^{-2\theta_1} (1+t)^{-\gamma\theta_1}Z^{1+\theta_1}.
\]
Also, \eqref{VT4} implies
\(T^{-3/2}\geq c_\gamma(1+t)^{-\gamma}\).
We insert these bounds into \eqref{simple2} to obtain
\begin{equation}\label{simple3}
Z'+c\varepsilon_1^{-2\theta_1} (1+t)^{-\gamma(1+\theta_1)}Z^{1+\theta_1} \leq\frac{\gamma}{1+t}Z+\frac{\varepsilon_1^2}{(1+t)^2}.
\end{equation}

We compare with \(K\varepsilon_1^2(1+t)^{-p}\), for a fixed sufficiently
large \(K\). Substituting this function into the left-hand side of
\eqref{simple3}, minus the right-hand side, gives
\[
\varepsilon_1^2\left[ cK^{1+\theta_1}(1+t)^{-(p+\gamma)(1+\theta_1)} -(p+\gamma)K(1+t)^{-p-1}-(1+t)^{-2}\right].
\]
This is nonnegative for all \(t\geq0\) if
\(cK^{1+\theta_1}\geq(p+\gamma)K+1\), by \eqref{decay:exponents}.
Increasing \(K\) also ensures
\(Z(t_0)\leq K\varepsilon_1^2(1+t_0)^{-p}\), since \(t_0\leq1\).
The scalar comparison principle therefore gives
\begin{equation}\label{Z4}
Z(t)\leq K\varepsilon_1^2(1+t)^{-p}, \qquad p=\theta+\theta^2>1,\quad t_0\leq t\leq S .
\end{equation}
All constants in this comparison are independent of
\(E,t_0,S,\varepsilon_1\). Choosing \(\varepsilon_1\) small makes
the first two error terms in \eqref{simpleenergyinequ} strictly less
than their assumed bounds; \eqref{aprioriassumption} does the same
for the third term.

Integrating the low-order inequality \eqref{simple2} and using
\eqref{Z4} yields
\begin{equation}\label{inte}
\int_{t_0}^S T(t)^{-3/2} \|(g_1,g_2)(t)\|_{\mathbf D_{k_0}}^2\,dt \leq C_{k_0,k_1,T_0}\varepsilon_1^2 .
\end{equation}
Moreover, \eqref{VT4}, \eqref{Z4}, and \(p>1\) imply
\[
\int_{t_0}^S T(t)^{9/2} \|(g_1,g_2)(t)\|_{\mathbf E_{k_0}}^2\,dt \leq C\varepsilon_1^2 \int_0^\infty(1+\log(1+t))^{9/2}(1+t)^{-p}\,dt \leq C\varepsilon_1^2 .
\]
Weighted interpolation and H\"older's inequality in time give
\[
\begin{aligned}
&\int_{t_0}^S\|g(t)\|_{L_x^2H_5^{3/4}}^2\,dt \\
&\leq C \left(\int_{t_0}^S T^{9/2} \|(g_1,g_2)\|_{\mathbf E_{k_0}}^2\,dt\right)^{1/4} \left(\int_{t_0}^S T^{-3/2} \|(g_1,g_2)\|_{\mathbf D_{k_0}}^2\,dt\right)^{3/4} \leq C_{k_0,k_1,T_0}\varepsilon_1^2 .
\end{aligned}
\]
Likewise,
\(\sup_{[t_0,S]}\|g(t)\|_{L_x^2L_5^2}\leq C\varepsilon_1\).
On \([0,t_0]\), \eqref{gt0vep} and the temperature bound
\eqref{T(0)} give the same smallness, since
\(\|g\|_{L_x^2H_5^{3/4}}\leq C_{k_1}\|g\|_{Y_{k_1}}\) for
\(k_1\geq17\). The constants depend only on \(k_1,T_0\).
Choose \(\varepsilon_1\) so that the sum of the two norms in
\eqref{apriofg} is at most \(\eta/2\). This strictly improves that
assumption, uniformly in the finite endpoint \(S\).

To make the order of choices explicit, first fix \(k_0,k_1,\eta\)
and the energy coefficients. Then choose \(\gamma,K\) as above
and fix a sufficiently small \(\varepsilon_1\), depending only on
\(k_0,k_1,T_0,V_0\). Define \(t_0>0\) by \eqref{T02}.
Finally choose the field threshold to guarantee Lemma~\ref{largetime},
\eqref{VTcondition0}, \eqref{VT5}, \eqref{decay:coefficients},
and the strict improvement in \eqref{aprioriassumption}.
 The initial-data threshold
is \(\|g_0\|_{X_{k_1}}^2<\varepsilon_1^4\) which is independent of \(E\).
This supplies the two constants required in the statements of the
main theorems.

All the preceding arguments apply on finite intervals. If the solution
had a finite maximal time \(S_*\), the strict improvements would keep
the smallness conditions valid on \([0,S_*)\).
Equations \eqref{VT1}, \eqref{t<}, \eqref{simple2}, and \eqref{inte}
would keep \(T,T^{-1},|V|\), the high-order energy, and its integrated
dissipation bounded there. Remark~\ref{nu:continuation} would then extend
the solution across \(S_*\), a contradiction. Thus the solution is global.
Finally, \eqref{Z4}, \eqref{normT(0)}, and
$g=g_1+\sqrt\mu g_2$ give the stated squared-norm decay
with exponent $-\theta-\theta^2$. The moment bounds follow
from \eqref{longtimeTVR}; both conclusions hold for $t\geq1$
since $t_0\leq1$.
The change of variables \eqref{changeofvariable} gives the conclusions
for \(F\).
\end{proof}

\section{\texorpdfstring{Appendix: the equations for $V$ and $T$}{Appendix: the equations for V and T}}
\begingroup
\makeatletter
\@ifpackageloaded{showkeys}{%
  \let\SK@\@gobbletwo
  \renewcommand*\showkeyslabelformat[1]{}%
}{}
\makeatother

We give the ODE result used in Section~\ref{local}.
The cutoff parameter is fixed and positive throughout
this lemma.

\begin{lem}\label{lem:local-wellposed}
Let $\nu>0$, $V_0\in\mathbb R^3$, $T_0>0$, and
$h\in L^2([0,S],L_x^2L_7^2)$.
The equations for $V_h,T_h$ in \eqref{nu:g-system},
with $R_h$ given by \eqref{nu:normalized-coefficients}
and initial value $(V_0,T_0)$, have a unique local
absolutely continuous solution. It belongs to $W^{1,2}$
and has $T_h>0$.
The solution continues as long as $|V_h|,T_h,T_h^{-1}$
remain bounded on a finite interval.

For two functions $h,\widetilde h$, suppose the corresponding
solutions satisfy $|V_h|,|V_{\widetilde h}|\leq M$ and
$a\leq T_h,T_{\widetilde h}\leq b$ on $[0,s]$, where $a>0$.
Then
\begin{equation}\label{nu:ODE-stability}
\begin{aligned}
&\sup_{[0,s]} \bigl(|V_h-V_{\widetilde h}|+|T_h-T_{\widetilde h}|\bigr) \\
&\quad\leq C\left( |V_h(0)-V_{\widetilde h}(0)| +|T_h(0)-T_{\widetilde h}(0)| +\int_0^s\|h-\widetilde h\|_{L_x^2L_7^2}dt\right).
\end{aligned}
\end{equation}
Here $C$ may depend on $\nu,a,b,M,s$ and the
$L^1([0,s],L_x^2L_7^2)$ norms of the two functions $h,\widetilde h$.
\end{lem}

\begin{proof}
The cutoff removes the singularity at the origin, and
direct differentiation gives
\[
\left\|\frac{u}{|u|^3}\chi^c_{\nu^2}(u)\right\|_\infty \leq C\nu^{-2},\qquad \left\|\nabla_u\left( \frac{u}{|u|^3}\chi^c_{\nu^2}(u)\right)\right\|_\infty \leq C\nu^{-3}.
\]
When $T_h,T_{\widetilde h}\geq a$, the arguments of this
smooth function satisfy
\[
\left|V_h+\sqrt{T_h}v -V_{\widetilde h}-\sqrt{T_{\widetilde h}}v\right| \leq C_a\langle v\rangle \left(|V_h-V_{\widetilde h}|+|T_h-T_{\widetilde h}|\right).
\]
In \eqref{nu:normalized-coefficients}, we write the kernel
as a function of $V_h+\sqrt{T_h}v$ and split the difference
of the two integrals into the difference $h-\widetilde h$ and
the difference of these arguments. The mean value theorem
and
$\|f\|_{L_x^1L_1^1}
 \leq C\|f\|_{L_x^2L_7^2}$
give
\begin{equation}\label{nu:ODE-Lipschitz}
\begin{aligned}
|R_h-R_{\widetilde h}| \leq C_{\nu,a,b,M}\bigl[ &\|h-\widetilde h\|_{L_x^2L_7^2} \\
&\quad{} + (1+\|\widetilde h\|_{L_x^2L_7^2}) (|V_h-V_{\widetilde h}|+|T_h-T_{\widetilde h}|)\bigr].
\end{aligned}
\end{equation}
The same bound holds for the right-hand sides of the
equations for $V_h,T_h$, since $V_h$ is bounded.
These right-hand sides are measurable in time and locally
Lipschitz in $V_h,T_h$, with an integrable Lipschitz
coefficient. Carath\'eodory's existence theorem gives the local
solution and uniqueness. Their $L^2$ bound on the stated
set gives $W^{1,2}$ regularity.

Subtracting the integral equations for two solutions and
applying Gronwall's inequality proves
\eqref{nu:ODE-stability}. The same argument at an endpoint
gives continuation while the three stated bounds persist.
For equal initial data,
\eqref{nu:ODE-stability} gives \eqref{nu:ODE-difference}
after squaring.
\end{proof}
\par\endgroup

\section*{Acknowledgement}
The research of L.-B. He was supported by NSF of China under Grant No.11771236 and New Cornerstone Investigator Program 100001127. Jie Ji was supported by Jiangsu Funding Program for Excellent Postdoctoral Talent and Basic Research Program of Jiangsu under Grant No.BK20251376.

{\emergencystretch=2em\printbibliography}
 
\end{document}